\documentclass[11pt,twoside, a4paper]{article}

\usepackage{nemish}

\usepackage{hyperref} 

\begin{document}
\title{Scattering in quantum dots via noncommutative rational functions}
 \author{
L\'aszl\'o Erd\H{o}s\footnote{\hspace{0.15cm}Partially funded by ERC Advanced Grant RANMAT No. 338804} 
\\
{\small \begin{tabular}{c}{IST Austria}\\{lerdos@ist.ac.at} \end{tabular}} 
\and Torben Kr\"uger\footnote{\hspace{0.15cm} Partially supported by VILLUM FONDEN research grant no. 29369 and the Hausdorff Center for Mathematics  \newline  Date: \today }
\addtocounter{footnote}{-2}\addtocounter{Hfootnote}{-2}\\ 
{\small \begin{tabular}{c}  University of Copenhagen \\ tk@math.ku.dk  \end{tabular}}
\and Yuriy Nemish\footnotemark\\
{\small \begin{tabular}{c} UC San Diego\\ ynemish@ucsd.edu \end{tabular}}
}
\date{} 

\maketitle

\thispagestyle{empty} 

\begin{abstract}
  In the customary random matrix model for transport in quantum dots with $M$ internal degrees of freedom coupled to a chaotic environment via $N\ll M$ channels,  the density $\rho$ of transmission eigenvalues is computed from a  specific invariant ensemble for which  explicit formula for the joint probability density of all eigenvalues is available.
 We revisit this problem in the large $N$ regime allowing for (i) arbitrary ratio $\cstS := N/M\le1$;  and (ii) general distributions  for the matrix elements of the Hamiltonian of the quantum dot. 
  In the limit $\cstS\to 0$  we recover the formula for the density $\rho$ that  Beenakker~\cite{Been97} has derived for a special matrix ensemble.
We also  prove that the inverse square root singularity of the density  at zero and full transmission  in Beenakker's formula persists for any $\cstS<1$ but in the borderline case $\cstS=1$ an anomalous $\lambda^{-2/3}$ singularity arises at zero. 
To access this level of generality we develop the theory of global and local laws on the spectral density of a large class of noncommutative rational expressions in large  random matrices with i.i.d. entries. 
\end{abstract}


\section{Introduction and results}
\label{sec:introduction}

Since the pioneering discovery of E. Wigner on the universality of eigenvalue statistics of large random matrices \cite{Wign58}, random matrix theory has become one of the most successful phenomenological theories to study disordered quantum systems, see \cite{AkemBaikDiFrBook} for a broad overview.
Among many other applications, it has been used for open quantum systems and quantum transport,  in particular  to predict the distribution of  transmission eigenvalues of scattering in  quantum dots and wires. 
The theory has been  developed over many excellent works starting with the ground-breaking papers by Mello, Pereyra, Seligman~\cite{MellPereSeli85} and by Verbaarschot, Weidenm\"uller and Zirnbauer~\cite{VerbWeidZirn85}; for a complete overview  with extensive references see reviews  by Beenakker  \cite{Been97, Been11b}, Fyodorov and Savin \cite{FyodSavi11} and Schomerus \cite{Scho17}.

We will focus on quantum dots, i.e. systems without internal structure, coupled to an environment (electron reservoir)
via scattering channels. Quantum wires, with a typically quasi one-dimensional internal structure,
will be left for further works. In the simplest setup the quantum dot 
is described by a self-adjoint Hamiltonian (complex Hermitian matrix\footnote{Our method works for the real symmetric case  
  as well but  for simplicity we stay in the complex Hermitian symmetry class.}) $H\in \CC^{M\times M}$ acting  on  an 
$M$ dimensional state space $\CC^M$.
It is coupled to an environment with $N_0$ effective degrees of freedom via an $M\times N_0$ complex coupling matrix $W$.
 Following Wigner's paradigm, both the  Hamiltonian $H$ and the coupling matrix $W$ are  drawn
from  random matrix ensembles respecting the basic symmetries of the model. Typically the entries of $W$ are
independent, identically distributed ({\it i.i.d.}), while $H$ is a Wigner matrix, i.e. 
it has i.i.d. entries on and above the diagonal. We allow for general distributions in contrast to most existing works in the
literature that assume $H$ has Gaussian  or Lorentzian distribution.

The Hamiltonian of the total system at Fermi energy $E\in \RR$ is given by  (see  \cite[Eq. (80)]{Been97})
$$
{\mathcal H} = \sum_{a=1}^{N_0} | a \rangle E \langle a | + \sum_{\mu,\nu=1}^M |\mu\rangle H_{\mu\nu} \langle \nu|
+ \sum_{\mu=1}^M\sum_{a=1}^{N_0} \Big[ |\mu \rangle W_{\mu a} \langle a| + |a\rangle W_{\mu a}^* \langle \mu| \Big].
$$
One common assumption is that the interaction $W$ is independent of the Fermi energy $E$.

At any fixed energy $E\in \RR$ we define the {\it scattering matrix} (see      \cite[Eq. (81)]{Been97})
\begin{equation}
  \label{eq:1}
  S(E)
  :=
  I - 2\pi \imu W^* (E\cdot I - H + \imu \pi W W^*)^{-1} W \in \CC^{N_0\times N_0}.
\end{equation}
This is the finite dimensional analogue of the  Mahaux-Weidenm\"uller formula in nuclear physics \cite{MahaWeid68}
that can be derived from ${\mathcal H}$ in the $N_0\to \infty$ limit.
The definition \eqref{eq:1}  will be the  starting point of our
mathematical  analysis.
Since $H=H^*$, one can easily check that $S(E)$ is unitary.

To distinguish between transmission and reflection, we assume that the scattering channels are split into two groups, left and right channels,
with dimensions $N_1+N_2=N_0$ and the interaction Hamiltonian
is also split accordingly; $W = (W_1 , W_2) \in \CC^{M\times (N_1+N_2)}$.
Therefore $S(E)$ 
has a natural $2\times 2$ block structure and we can write it as
(see  \cite[Eq. (23)]{Been97}) 
\begin{equation}
  \label{eq:2}
  S(E)
  =
  \left(
    \begin{array}{cc}
      R & T'
      \\
      T & R'
    \end{array}
  \right)
  ,
\end{equation}
where $R\in \CC^{N_1 \times N_1}$, $R'\in \CC^{N_2 \times N_2}$ are the {\it reflection matrices}
and $T\in \CC^{N_2 \times N_1}$, $T' \in \CC^{N_1 \times N_2}$ are  the {\it transmission matrices}.

As a consequence of unitarity, one finds that $TT^*$, $T'(T')^*$, $I-RR^*$ and $I  -R' (R')^*$ have the same set of nonzero eigenvalues.
For simplicity, we assume that $N_1=N_2=N$
i.e. generically these four matrices have no zero eigenvalues,
the general $N_1 \neq N_2$ case has been considered in \cite{BrouBeen96}.
We denote these \emph{transmission eigenvalues}  by $\lambda_1, \lambda_2, \ldots, \lambda_N$.
They express the rate of the transmission through each channel.
By unitarity of $S$,  clearly $\lambda_i\in [0,1]$ for all $i$; $\lambda_i=0$ means the channel is  closed,
while $\lambda_i=1$ corresponds to a fully open channel. 
The transmission eigenvalues carry important physical properties of the system. For example $\Tr TT^* =\sum_i \lambda_i$
gives  the {\it zero temperature conductance}  ({\it Landauer formula}  \cite[Eq. (33)]{Been97}), while
\begin{equation}\label{shotnoise}
   \sum_i \lambda_i (1-\lambda_i) = \Tr TT^* - \Tr (TT^*)^2 
\end{equation}
is the {\it shot noise power}, giving the zero temperature fluctuation of the current ({\it B\"uttiker's formula}, \cite{Butt90},  \cite[Eq. (35)]{Been97}).
The dimensionless ratio of the shot noise power and the conductance is called the {\it Fano factor} (see \cite[Eq. (2.15)]{Been11b})
\begin{equation}\label{fano}
  F: =\frac{ \sum_i \lambda_i (1-\lambda_i)}{\sum_i \lambda_i}.
\end{equation}
The current fluctuation is therefore given by  a certain linear statistics  of the transmission eigenvalues
and  thus it can  be computed from the density 
$\rho$  of these eigenvalues. Therefore, determining $\rho$ is a main task in the theory of quantum dots.

In many physical  situations  it is found 
that $\rho$ has a {\it bimodal} structure with a peak at zero and a peak at unit transmission rates.
Furthermore, $\rho$ exhibits a power law singularity at the edges of its support $[0,1]$. One main 
result of the theory in \cite{Been97} is that in the $M\gg N\gg 1$ regime at energy $E=0$ the density of transmission eigenvalues for a quantum dot is given by
\begin{equation}\label{bimodal}
  \rho_{\tiny \mbox{Bee}}(\lambda) = \frac{1}{\pi \sqrt{\lambda(1-\lambda)}},
\end{equation}
(the answer is different for quantum wires), 
i.e. it has an inverse square root singularity at both edges, see \cite[Eq. (3.12)]{Been11b}. In this case, the
Fano factor is $F=1/4$ which fits well the experimental data.

The goal of this paper is to revisit and substantially generalize
the problem of transmission eigenvalues 
 with very different methods than Beenakker and collaborators used.  While
those works  used invariant matrix  ensembles for $H$ and  
 relied on explicit computations for the circular ensemble,
we consider very general  distribution for the matrix elements of $H$.
In particular,  we show that  in the regime $\cstS:= N/M \in (0,1]$, $M\to\infty$,
the empirical density of transmission eigenvalues has a deterministic limit $\rho=\rho_{\cstS}$ and 
we give a simple algebraic equation to compute it. The solution can be continuously extended as $\cstS\to0$,
the equation simplifies for $\cstS=0$, the density becomes explicitly computable and 
 it coincides with~\eqref{bimodal}; $\lim_{\cstS\to 0} \rho_{\cstS} =\rho_{\cstS=0}=  \rho_{\tiny \mbox{Bee}}$
 for $E=0$ and our formula holds for any $E$ in the bulk spectrum of $H$.
 While no short explicit formula is available for 
the general case  $\cstS \in (0, 1]$,  we can analyse the singularities of $ \rho_{\cstS}$
for any fixed  $\cstS \in (0, 1]$ in detail. 

More precisely,  we rigorously prove that for any fixed $\cstS\in (0,1)$ the density $\rho_{\cstS}$ has an inverse square root singularity 
both at 0 and at 1,
\begin{equation}\label{square}
\rho_{\cstS} (\lambda) \sim \lambda^{-1/2}, \;\quad  \mbox{for $0<\lambda\ll 1$}, \quad \mbox{and} \quad
\rho_{\cstS} (\lambda) \sim (1-\lambda)^{-1/2}, \;\quad \mbox{for $0<1-\lambda\ll 1$},
\end{equation}
qualitatively in line with  $\rho_{\tiny \mbox{Bee}}$ from~\eqref{bimodal}. However, $\rho_{\cstS}$ is not symmetric  around  $\lambda=\frac{1}{2}$ 
and the Fano factor at $E=0$ slightly  differs from $\frac{1}{4}$ when $\cstS\ne 0$.
Fig.~\ref{fig:1}  shows the Fano factor for different values of $\cstS$ numerically computed from our theory.
\begin{figure}[h]
  \centering
  \includegraphics{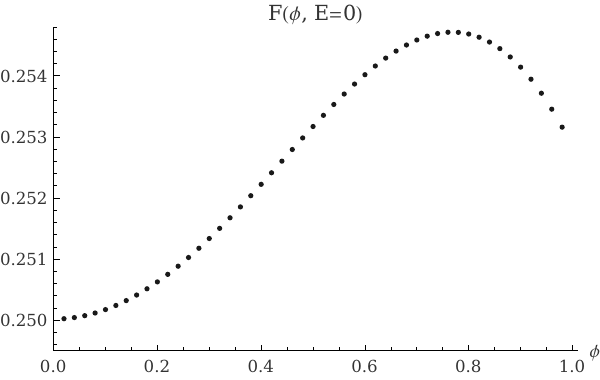}
  \caption{Fano factor for $E=0$ and $\cstS \in (0,1)$}
  \label{fig:1}
\end{figure}
We mention that the deviation from $1/4$ is within   $2\%$ for the entire
range of $\cstS \in [0,1)$ which is well below the error bar of the experimental results presented in Fig. 6 of \cite{Been97} and adapted from \cite{OberSukhStruSchoHeinHoll01}. 

 The value $\cstS=1$ is special, since the singularity of $\rho$ at $\lambda\approx 0$ changes
to 
\begin{equation}\label{cube}
\rho_{\cstS=1}(\lambda) \sim \lambda^{-2/3}, \;\quad \mbox{for $0<\lambda\ll 1$}
\end{equation}
 from $\lambda^{-1/2}$ in~\eqref{square}, while the 
inverse square root singularity at 1 persists. This enhancement of  singularity signals the emergence of a $\delta$-function 
component at 0 in $\rho_{\cstS}$ as $\cstS$ becomes larger than 1, which is a direct consequence of $TT^*$ not having full rank  when $N>M$.
 We remark that this regime is quite unphysical since it corresponds
to more scattering channels than the total number of internal states of the quantum dot. 
In realistic quantum dots $N$ is smaller than $M$ since not every mode of the dot may
participate in scattering. 
We therefore do
not pursue the detailed analysis of $\rho_{\cstS}$ for $\cstS>1$, although our method can easily be extended to these $\cstS$'s as well.

There are two main differences between our model and that of Beenakker {\it et al.}
First, the distribution imposed on the random matrix $H$ is  different and we consider random $W$.
Second, our current method works in the entire regime  $\cstS = N/M\in(0,1]$, while Beenakker assumes
$M\gg N$, i.e. $\cstS\ll 1$. We now explain both differences.

Following Wigner's original vision, any relevant distribution must 
respect the basic symmetry of the model; in our case this demands that $H$ be complex Hermitian, while
no symmetry constraint is imposed on $W$. Respecting basic symmetries,
one may define ensembles essentially in two ways. Invariant ensembles are defined 
by imposing that the entire distribution be invariant; it is typically achieved
via a global Gibbs factor times the natural flat measure
on the space of matrices satisfying the basic symmetry. Wigner ensembles and their generalizations
impose distributions directly on the matrix elements and often demand independence
(up to the basic symmetry constraint). These two procedures typically yield different ensembles.

While in the simplest case of random Hermitian matrices both types of ensembles have been actively investigated,
 for ensembles with more complicated structure, like our $S$ that 
is a rational function of the basic ingredients $H$ and $W$ \eqref{eq:1},  up to recently only the invariant approach 
was available.  Sophisticated explicit formulas have been developed to
find the joint distribution of eigenvalues for  more and more 
complicated structured ensembles (see \cite{ForrBook}), which could then be
combined with orthogonal polynomial methods to obtain local correlation functions.
The heavy reliance on explicit formulas, however, imposes a serious limitation on how 
complicated functions of random matrices, as well as
how general distributions on these matrices can be considered.
For example,  the Gibbs factor is often restricted to Gaussian or closely related ensembles
to remain in the realm of explicitly computable orthogonal polynomials.

There have been  considerable developments in the other type of ensembles in the recent years.
Departing from the invariant world, about ten years ago the Wigner-Dyson-Mehta universality of local 
eigenvalue statistics has been proven  for Wigner matrices with i.i.d. entries, see \cite{ErdoYauBook}
and references therein. Later the i.i.d. condition was relaxed and even matrices with quite general 
correlation structures among their entries can be handled \cite{ErdoKrugSchr18}. One of the key
ingredients was to better understand the {\it Matrix Dyson equation (MDE)}, the basic equation 
governing the density of states \cite{AjanErdoKrug19}. Together with the linearization trick, this allows us to handle
arbitrary polynomials in  i.i.d. random matrices \cite{ErdoKrugNemi_Poly} and in the current work 
we extend our method to a large class of rational functions. Note that even if the 
building block matrices have independent entries, the linearization of their rational expressions
will have dependence, but the general MDE can handle it (see \eqref{eq:18}).
In our work we  deal only with bounded rational functions
and relatively straightforward regularizations of unbounded rational functions,
the general theory of 
unbounded rational functions is still in development, see~\cite{MaiSpeiYin19} and references therein.
We stress that the distribution of the matrix elements of $H$ and $W$ can be practically arbitrary.
 In particular, our result is not
restricted to the Gaussian world. 

In comparison, the result of Beenakker {\it et al.} \cite[Sect. III.B]{Been11b}, see also \cite[Sect. IV]{Brou95}, postulates
that $H$ is GUE, while $W$ may even be deterministic and only its singular values are relevant.
For example, if all nonzero singular values are equal one  and $N_0=2N\le M$, 
i.e. $\cstS \le \frac{1}{2}$, then  $S$ can be written as
$$
S = \frac{I + i\pi \widetilde H}{I- i\pi \widetilde H}, \qquad \widetilde H: = Q^T(H-E \cdot I )^{-1} Q,
$$
where $Q\in \CC^{M\times 2N}$ with $Q_{ij}=\delta_{ij}$. 
In fact, for the sake of explicit calculations, 
it is necessary to replace the GUE by a Lorentzian distribution  (irrelevant constants ignored),
\begin{equation}\label{lor}
  P(H) \sim \mbox{det}[ I + H^2]^{-M},
\end{equation}
since in this way $\widetilde H$ is also Lorentzian,
and argue separately that in the sense of correlation functions \eqref{lor}
is close to a GUE when $M$ is very large \cite[Section III]{Brou95}. 
Under these conditions  $S$ becomes Haar distributed
on $U(2N)$ as $M\to \infty$ and $N$ is fixed; this is the step where $M\gg N$ is necessary.
Furthermore, one can verify \cite[Section II.A.1]{Been97} that at energy $E=0$ the transmission eigenvalues
of a Haar distributed scattering matrix follow the circular ensemble on $[0,1]$. Therefore $\lambda_i$'s have
a well known joint distribution
$$
P\big(\{\lambda_i\}\big)\sim \prod_{i<j} (\lambda_j-\lambda_i)^2  \quad \mbox{on} \;\;\; [0,1]^N,
$$
and their density can be easily computed, yielding \eqref{bimodal} in the $N\to \infty$ limit.

While Beenakker's result relies on an impressive identity, it allows little flexibility in the inputs: $H$ needs
to be Lorentzian with very large dimension, moreover $M\gg N$ and $E=0$
are required.
 In contrast, our setup allows for a large freedom in 
choosing the distribution of  $H$, it covers the entire  range $M\ge N$ and also applies to any $E$ in the bulk, 
however  for computational simplicity we  model the contacts $W$ also  by a random matrix.
While  the most relevant regime for scattering on quantum dots is still $M\gg N$,
as scattering involves surface states only, a very recent work~\cite{Fyod20} introduces absorbing channels well
inside the quantum dot that leads to  physical models with $M\sim N$.

The flexibility in our result stems from the fact that our method directly aims at the density
of states via an extension of the MDE theory to linearizations of rational functions of random matrices. It seems unnecessarily ambitious, hence requiring too restrictive conditions,
to attempt to find the joint distribution of all eigenvalues. Even for Wigner matrices this is
a hopeless task beyond the Gaussian regime. 
We remark that the present analysis of the density of transmission eigenvalues  for the quantum dot is only one 
convenient application of our 
approach. This method
  is powerful enough to answer many related questions concerning 
the density of states such as   the analysis of the scattering poles
\cite{Fyod16} as well as extensions from  quantum dots to quantum wires that we will address in future work.

\subsection{Model and main theorem}
\label{sec:model-main-theorem}
In order to accommodate all parameters that will appear in our analysis, we introduce a generalized version of the scattering matrix \eqref{eq:1} 
\begin{equation}
  \label{eq:73}
  S(\cstY):= I - 2 \,\imu \cstT \, W^*(\cstY \cdot I -H + \imu \cstT\, WW^*)^{-1} W \,
\end{equation}
with $\cstT>0$  and $\cstY \in \CC_{+}$ for $\cstS\in (0,1/2]$ and $\cstY \in \CC_+ \cup \RR$ for $\cstS \in (1/2,1]$,
where  $\CC_+ : =\{z \in \CC: \Im z >0\}$ denotes the complex upper half plane.
 The constant $\cstT>0$  
quantifies the coupling between the internal quantum dot Hamiltonian $H$ and the  external leads $W_1,W_2$, i.e. the effective Hamiltonian of the open quantum systems becomes $H-\imu \cstT W W^{*}$. The spectral parameter 
$\cstY$ is used to regularize the potentially unstable inverse in \eqref{eq:73}. 
Note that for $\cstS \in (0,1/2)$ the matrix  $WW^*$  has a nontrivial kernel (and even for $\phi=1/2$ it has
a very small  eigenvalue), hence initially a regularization 
with a small positive imaginary part is necessary  that will later be carefully  removed.
For the technically easier $\cstS \in (1/2,1]$ regime
we may directly  choose $\cstY$ to be real since  the  eigenvalues of $WW^*$
 are bounded away from zero with very high probability.

The general formula~\eqref{eq:1} recovers 
 the scattering matrix \eqref{eq:1} by setting $\cstT = \pi$, $\Re \cstY = E$ and $\Im \cstY = 0$.
The regularized scattering matrix still has the $2\times 2$-block structure from \eqref{eq:2} with $T \in \mathbb{C}^{N \times N}$.
We consider the following random matrix model (see \eqref{eq:1} and \eqref{eq:2})
\begin{equation}
  \label{eq:3}
  \Tb_{\cstY,\cstS, \cstT}
  :=
  T T^* 
  = 
  4  \cstT^2  W_{2}^*\frac{1}{\cstY-H + \imu  \cstT  W W^{*}} W_{1} W_{1}^* \frac{1}{\overline{\cstY}-H - \imu  \cstT  W W^{*}} W_{2}
  ,
\end{equation}
where the triple of parameters $(\cstY, \cstS, \cstT)$ belongs to the same set as for \eqref{eq:73}.
Furthermore, for $M,N\in\NN$, $\cstS:= N/M$ the matrices $H\in \CC^{M \times M}$ and  $W_{1},W_{2} \in \CC^{M\times N}$ are three independent random matrix ensembles satisfying the following assumptions 
\begin{itemize}
\item[\textrm{(\textbf{H1}qd)}] $H$ is a Hermitian random matrix having independent (up to symmetry constraints) centered entries of variance $1/M$;
\item[\textrm{(\textbf{H2}qd)}] $W_{1}$ and $W_{2}$ are (non-Hermitian) random matrices having independent centered entries of variance $1/M$;
\item[\textrm{(\textbf{H3}qd)}] entries of $H$, $W_{1}$ and $W_{2}$, denoted by $H(i,j)$, $W_{1}(i,j)$ and $W_{2}(i,j)$ correspondingly, have finite moments of all orders, i.e., there exist $\cstA_{n} >0$, $n\in \NN$, such that
  \begin{equation}
    \label{eq:4}
    \max_{\substack{1\leq i,j \leq M}} \EE\big[\,|\sqrt{M}H(i,j)|^{n}\big] + \max_{\substack{1\leq i \leq M \\ 1\leq j \leq N}}\Big(\EE\big[\,|\sqrt{M}W_{1}(i,j)|^{n} \big]+\EE\big[\,|\sqrt{M}W_{2}(i,j)|^{n}\big]\Big)
    \leq
    \cstA_{n}
    .
  \end{equation}
\end{itemize}
\begin{rem}[Constant matrices] \label{rem:constantMatrices}
  In \eqref{eq:3} and later in the paper, for $B \in \CC$, $n\in \NN$ and $I_{n}\in \CC^{n\times n}$ the identity matrix of size $n$, we use the shorthand notation $B \cdot I_{n} = B$.
  This notation is used only when the dimension of $I_{n}$ can be unambiguously determined from the context.
\end{rem}
\begin{rem}
   In the sequel we consider the coupling constant $\cstT$ to be a fixed positive number, therefore we will omit the dependence on $\cstT$ in our notation. 
\end{rem}
Denote by $\mu_{\Tb_{\cstY,\cstS}}(d\lambda):=\frac{1}{N}\sum_{i=1}^{N} \delta_{\lambda_{i}}$ the empirical spectral measure of $\Tb_{\cstY,\cstS}$, where $\lambda_{i}$ are the eigenvalues of the Hermitian matrix $\Tb_{\cstY,\cstS}$.
To simplify the presentation, we will assume in this paper that the  dimensions of the matrices $H$, $W_1$ and $W_2$ grow over a subsequence $(N, M)= (kn, ln)$, $n\in \NN$, i.e. $N/M=\cstS$ is kept fixed. One could easily extend our argument to include the general situation when one considers two 
sequences $N=N_n$ and $M=M_n$ tending to infinity such that $\cstS_n = N_n/M_n \to \cstS$.

 We now formulate our  two main results. The first one, Theorem~\ref{thr:Global law}, is the {\it global law} for 
 the empirical eigenvalue density; it
  shows that
$\mu_{\Tb_{\cstY,\cstS}}(d\lambda)$ has a deterministic limit denoted by $\rho_{\cstY,\cstS}(d\lambda)$. The second
result, Theorem~\ref{thm:main}, contains key properties of $\rho_{\cstY,\cstS}(d\lambda)$.
It shows that the regularization in the spectral parameter $\cstY$ can be removed and that 
$\rho_{\cstY,\cstS}$ can be continuously extended to $\cstS\to0$; moreover, it explicitly identifies
its singularities at zero and one.

  \begin{thm}[Global law] \label{thr:Global law} Fix a positive real number
  $\cstT>0$ and  a rational number $\cstS\in (0,1]\cap \QQ$.
Let $\cstY$ be a fixed spectral parameter: for $\cstS \in (0,1/2]$ we assume  that  $\cstY\in  \CC_+ $, while for
$\cstS \in (1/2,1]$ we assume $\cstY\in \CC_+ \cup \RR $.
  Then there exists a deterministic probability measure $\rho_{\cstY,\cstS}(d\lambda)$ with $\mathrm{supp}\, \rho_{\cstY,\cstS} \subset [0,1]$ such that $\mu_{\Tb_{\cstY,\cstS}}(d\lambda)$ converges weakly to $\rho_{\cstY,\cstS}(d\lambda)$ in probability (and almost surely) as $M,N \rightarrow \infty$ with $N/M = \cstS$.  
  \end{thm}

\begin{thm}[Properties of the transmission eigenvalue density]
 Let the numbers $\cstT, \cstY,\cstS$  and the measure  $\rho_{\cstY,\cstS}(d\lambda)$ be as in Theorem~\ref{thr:Global law}.  
  \label{thm:main}
  \begin{itemize}
  \item[(i)] The function
    \[\Big( \big((0,1/2]\cap \QQ \big) \times  \CC_+  \Big)\cup \Big( \big( (1/2,1]\cap \QQ \big)\times \big( \CC_+  \cup \RR \big)\Big) \ni (\cstS,\cstY) \mapsto \rho_{\cstY,\cstS}(d\lambda)\]
    from Theorem~\ref{thr:Global law} with values in probability measures on $[0,1]$ can be extended to a function on $(0,1] \times ( \CC_+  \cup \RR)$ that is continuous in the weak topology. In particular, the weak limit 
    \begin{equation}
      \label{eq:233}
      \rho_{E,\cstS}(d\lambda)=\lim_{\, \CC_+   \ni \cstY \to E}\rho_{\cstY,\cstS}(d\lambda)
    \end{equation}
   exists for all $E \in \RR$. For every $E \in \RR$ and $\cstS \in (0,1]$ the measure $\rho_{E,\cstS}(d\lambda)$ is absolutely continuous, i.e. $\rho_{E,\cstS}(d\lambda)=\rho_{E,\cstS}(\lambda)d \lambda$. 
   The function $\rho_{E,\cstS}(\lambda)$  is bounded  when $\lambda$ is away from 0 and 1.
   Furthermore, the weak limit $\rho_{E,0}(d\lambda):=\lim_{\cstS \downarrow 0}\rho_{E,\cstS}(d\lambda)$  exists for every $E \in \RR$. 
\item[(ii)] If $\cstS = 1$ and $\cstT>0$, then $\rho_{E,\cstS}(\lambda)$ has  the following asymptotics near 0 and 1:
    \begin{itemize}
    \item[(a)] for $E\in \RR$
          \begin{equation}
      \label{eq:5}
      \rho_{E, 1}(\lambda)
      =
      \frac{1}{\pi}\sqrt[3]{\frac{1+E^2}{4 \cstT^{2} }} \, \lambda^{-2/3} + \OO{\lambda^{-1/3}} \mbox{ as }\lambda\rightarrow 0_{+},
    \end{equation}
  \item[(b)] for $|E| < \frac{1}{ \cstT } \Big(2 \sqrt{1 + 6  \cstT^2  +  \cstT^4 }-2  \cstT^2  - 2 \Big)^{1/2}$
    \begin{equation}
      \label{eq:179}
      \rho_{E, 1}(\lambda)
      =
      \frac{1}{\pi} \, \frac{ -4 \xi_{0}}{\xi_{0}^{2}+  \cstT^2  E^2 + 4 }\, (1-\lambda)^{-1/2} + \OO{1} \mbox{ as } \lambda\rightarrow 1_{-}
      ,
    \end{equation}
    where $\xi_{0} = -\sqrt{2 \sqrt{1 + 6  \cstT^2 + \cstT^4 }-2  \cstT^2  -2- \cstT^2  E^2}$.
    \end{itemize}
  \item[(iii)] If $\cstS \in (0,1)$ and $\cstT>0$,  then
    \begin{itemize}
    \item[(a)] for $E\in \RR$
          \begin{equation}
      \label{eq:15}
      \rho_{E,\cstS}(\lambda)
      =
      \frac{1}{\pi} \,  \frac{ \cstT^2  E^2 \xi_0^4+ \big(\xi_0^2+2\cstS\big)^2\big(4+  \cstT^2  \xi_0^2\big)}{4 \cstT  \xi_0\big(\xi_0^2+2\cstS\big)^2}\, \lambda^{-1/2} + \OO{1}
      \mbox{ as }\lambda\rightarrow 0_{+},
    \end{equation}
    where $\xi_{0}>0$ is the unique positive solution of 
 \begin{equation}
 \xi_0^6 +  \big(E^2+8\cstS-4\big) \xi_0^4+ 4 \cstS \big(E^2+ 5\cstS-4\big) \xi_0^2+16(\cstS-1)\cstS^2=0\,.
\end{equation}
  \item[(b)] for $|E| < E_\ast:=\frac{1}{ \cstT }\Big(2 \sqrt{ 1 + 2 (1+2\cstS)  \cstT^2   + (2 \cstS-1)^2  \cstT^4  }  - 2  \cstT^2  (2\cstS-1) -2\Big)^{1/2}$ the density satisfies 
    \begin{equation}
      \label{eq:35}
      \rho_{E,  \cstS }(\lambda)
      =
      \frac{1}{\pi} \, \frac{-4 \xi_1}{\xi_1^2 + \cstT^2  E^2+4  }\,(1-\lambda)^{-1/2}  + \OO{1} 
      \mbox{ as } \lambda\rightarrow 1_{-}
      ,
    \end{equation} 
    where
        \begin{equation}
     \xi_1 = -  \cstT  \sqrt{E_\ast^2- E^2}\,.
         \end{equation} 
  \end{itemize}
  \item[(iv)] If $\cstS =0$, $\cstT>0$ and $|E| < 2$ the density is given globally by the explicit formula 
    \begin{equation}
      \label{eq:235}
           \rho_{E,  0 }(\lambda)
      =
  \frac{1}{ \pi  }\frac{\cstT\big(1+ \cstT^2 \big) }{\big(1+ \cstT^2 \big)^2\lambda +  \cstT^2  \big(4-E^2\big)(1-\lambda)} \,\sqrt{\frac{4-E^2}{\lambda(1-\lambda)}}\,, \qquad \lambda \in (0,1)\,.
           \end{equation} 
  \end{itemize}
\end{thm}

\begin{proof}[Proofs of Theorems~\ref{thr:Global law}  and~\ref{thm:main}]
  We now explain the structure of the paper that contains the proof of both theorems. 
The density function $\rho_{E, \cstS}(\lambda)$ will be derived from a deterministic equation, the Dyson equation of the linearizaton of \eqref{eq:3} (see equations \eqref{eq:18} and \eqref{eq:24} for $\cstS=1$, as well as \eqref{eq:52} and \eqref{eq:103} for $\cstS \in (0,1)$  below). This equation plays a central role in our analysis. 
 Sections \ref{sec:proofs} and \ref{sec:proof-theor-refthm:m} below contain the proofs of the model specific parts of Theorem~\ref{thm:main} which use some key conclusions of the general theory on noncommutative rational functions developed in Appendix~\ref{sec:spectr-prop-gener}.
 The  $\cstS=1$ case is treated in Section~\ref{sec:proofs} with full details, while in Section~\ref{sec:proof-theor-refthm:m} we explain the modifications for the general  rational   $\cstS \in (0,1)$ case.
 Theorem~\ref{thr:Global law} is proven in Lemma \ref{lem:global-law}  for $\cstS = 1$ and Lemma~\ref{lem:global-law-gen} for general rational $\cstS \in (0,1)$   using the general global law from Appendix~\ref{sec:spectr-prop-gener}.
 Part (ii)  is proven in Section~\ref{sec:proof-ii-iii} after having established key properties of the solution for the  Dyson equation.
 Section~\ref{sec:proof-theor-refthm:m} follows the same structure as Section~\ref{sec:proofs} and proves parts (i), (iii) and (iv) of Theorem~\ref{thm:main} for the $\cstS \in [0,1)$ case. 
\end{proof}
\begin{rem}
  The restriction $|E|<E_*$ in Theorem~\ref{thm:main} is used to identify the regime in which the density $\rho_{E,\phi}$ has two singularities.
  For $|E|>E_*$, the singularity at $\lambda = 1$ disappears and the support of the density becomes bounded away from $\lambda=1$.
  Physically this indicates that there are no fully open channels in this regime.
  This effect is most severe in the case $\phi=0$, where the system becomes completely reflective as $|E|$ increases above  the threshold $E_*=2$.
  Our approach is also applicable for $|E|>E_*$, but for the simplicity of the presentation we restrict our analysis to the physically more relevant situation when two singularities exist simultaneously.
\end{rem}
\medskip

We formulated Theorem~\ref{thr:Global law} about the specific matrix~\eqref{eq:3}.
However, our method works for very general  noncommutative (NC) rational expressions in large matrices with 
 i.i.d. entries (with or without Hermitian symmetry) generalizing  our previous work~\cite{ErdoKrugNemi_Poly} on polynomials. 
 For convenience of the readers interested
 only in the concrete scattering problem the main part of our paper focuses on this model and
  we defer the general theory to Appendix~\ref{sec:spectr-prop-gener}.  The details of this appendix are
  not necessary in order to follow the main arguments in Sections \ref{sec:proofs} and \ref{sec:proof-theor-refthm:m} provided
  some basic facts from the appendix are accepted. These facts will be re-stated for the specialization to our 
  concrete operator $\Tb_{ \cstY ,\cstS}$ from~\eqref{eq:3}.  
  On the other hand, Appendix~\ref{sec:spectr-prop-gener} 
  is written in a self-contained form, so readers interested in the general theory can directly go to it,
   skipping the concrete application. 
  
  In Appendix~\ref{sec:spectr-prop-gener}, we first give
 the precise definition of the set of admissible rational expressions  that  requires some technical 
 preparation, see Section~\ref{sec:line-rati-funct} for details.  Roughly speaking, we can consider any  rational 
 expression whose denominators are stably invertible with overwhelming probability. 
 We remark that this property  holds for~\eqref{eq:3}  since the imaginary part of $\cstY-H + \imu   \cstT  W W^{*}$ 
 has a positive lower bound  as long as $\cstS>1/2$ or $\cstY \in  \CC_+ $.
  With this definition at hand, we develop the theory of global and 
 local laws as well as the identification of the pseudospectrum for such rational expressions
 in Sections~\ref{sec:proof-local-law} and \ref{sec:conv-spectr-rati}, respectively.
  
 \medskip

 {\it Acknowledgement.} The authors are  very grateful to Yan Fyodorov for discussions on the physical background 
and for providing references, and to the anonymous referee for numerous valuable remarks.

\section{Proof of Theorem~\ref{thm:main} for the special case  $\cstS = 1$}
\label{sec:proofs}

In this section we study the model \eqref{eq:3} for  fixed $\cstT>0$, $\cstS = 1$ and $\cstY \in \CC_{+}\cup \RR$.
This special choice of parameter $\cstS $ ensures that the linearization of $\Tb_{ \cstY ,\cstS}$ has a fairly simple structure, which makes the proof of  Theorems~\ref{thr:Global law} and \ref{thm:main}  more transparent and streamlined.
Generalization to $\cstS\in  (0,1) $ is postponed to Section~\ref{sec:proof-theor-refthm:m}.
Since the parameter $\cstS$ is fixed to be equal to 1, we will omit the dependence on $\cstS$  in the notation throughout the current section.
The information about linearizations of general rational functions
is collected in the Appendix~\ref{sec:line-rati-funct}.
Here we often refer to specialization of these results to $\Tb_{ \cstY }$.

\subsection{Linearization trick and the Dyson equation for linearization}
\label{sec:linearization-trick}

 We consider the random matrix model $\Tb_{ \cstY }$  defined in \eqref{eq:3}  as  a self-adjoint rational function of random matrices $H$, $W_1$ and $W_2$. 
In order to study its eigenvalues we introduce the self-adjoint linearization matrix $\Hb_{ \cstY }\in \CC^{8N\times 8N}$
\begin{equation}
  \label{eq:6}
  \Hb_{ \cstY }
  :=  
  \left(
    \renewcommand\arraystretch{1.4}
    \begin{array}{ccc|ccc|cc}
      0 & 0 & 0 & 0 & 0 & 0 & 0 & W_{2}^* \\
      0 & 0 & \frac{\imu}{ \cstT }  & 0 & 0 & 0 & 0 & W_{2}^* \\
      0 & -\frac{\imu}{ \cstT } & 0 & 0 & 0 & 0 & W_{2}^* & 0 \\ \hline
      0 & 0 & 0 & 0 & 0 & \frac{\imu}{ \cstT} & 0 & W_{1}^* \\
      0 & 0 & 0 & 0 & -\frac{1}{4 \cstT^2} & 0 & W_{1}^* & 0 \\
      0 & 0 & 0 & -\frac{\imu}{ \cstT }& 0 & 0 & W_{1}^* & 0 \\ \hline
      0 & 0 & W_{2} & 0 & W_{1} & W_{1} & 0 &  \cstY -H \\
      W_{2} & W_{2} & 0 & W_{1} & 0 & 0 &  \overline{\cstY} -H & 0 \\
    \end{array}
  \right)
\end{equation}
 with $\cstY \in \CC_{+}\cup \RR$. 
 Denote by $J_{m} \in \CC^{m\times m}$, $m\in \NN$, a matrix whose $(1,1)$ entry is equal to $1$ and all other entries are equal to $0$.
For any $n\in \NN$ and $R\in \CC^{n \times n}$ we define $\|R\|$ to be the operator norm of $R$.
 The following definition is a specialization of the notion of generalized resolvent from Appendix~\ref{sec:trivial-bound}. 
 \begin{defn}[Generalized resolvent] \label{defn:generalized-resolvent}
  We call the matrix-valued function $z \mapsto (\Hb_{\cstY} - z J_{8}\otimes I_N)^{-1}$ defined for $z\in \CC_+$ the \emph{generalized resolvent} of $\Hb_{\cstY}$.
\end{defn}
The results below are formulated using the notion of \emph{asymptotically overwhelming probability}.
\begin{defn}
  We say that a sequence of events $\{\Omega_{N}\}_{N\in \NN}$ holds asymptotically with overwhelming probability (\emph{a.w.o.p.} for short) if for any $D>0$ there exists $C_{D}>0$ such that
  \begin{equation}
    \label{eq:228}
    \PP[\,\Omega_{N}]
    \geq
    1-\frac{C_{D}}{N^{D}}
    .
  \end{equation}
\end{defn}
 Consider the set 
  \begin{equation}
    \label{eq:11}
    \Theta_N
    :=
    \Big\{\|H\| \leq 3,\, \| W_{1}\| \leq 3, \, \|W_2\| \leq 3, \,\|(W W^*)^{-1}\| \leq  12 \Big\}\; .
  \end{equation}
  Note that $ W W^*$ is a sample covariance matrix with concentration ratio $1/2$, hence its spectrum is asymptotically supported on  an arbitrarily small neighborhood of $[(1-\frac{1}{\sqrt{2}})^{2}, (1+\frac{1}{\sqrt{2}})^{2}]$ with very high probability.
   The limiting support follows from the 
classical Bai-Yin theorem \cite{BaiYin93}, the corresponding large deviation result under somewhat different conditions follows, e.g. from Corollary~V.2.1 of \cite{FeldSodi10}.
In fact, the boundedness of $\| (WW^*)^{-1}\|$  and  $\| W_i\|$ also follows from \cite{RudeVers09},
at least for subgaussian distributions.
Alternatively, under our conditions~\eqref{eq:4} this result also follows
from Lemma~\ref{lem:norm-bounds-random} in the Appendix by choosing $n=M$, $l=1$ and $m=2$.   
  Similarly, from the properties of classical Wigner and  \emph{iid} ensembles (see, e.g., \cite[Section~5]{BaiSilvBook}), we obtain that the event $\{\,\|H\| \leq 3,\, \| W_{1}\| \leq 3, \, \|W_2\| \leq 3 \, \}$ also holds \emph{a.w.o.p.}, and we conclude that for any $D>0$ there exists $C_D>$ such that
    \begin{equation}
    \label{eq:12}
    \PP[\Theta_{N}] \geq 1- \frac{C_{D}}{N^{D}}
    .
  \end{equation}
  In the rest of this section we consider the random matrix models \eqref{eq:73} and \eqref{eq:3} \emph{restricted} to the set $\Theta_N$.
  Since on this set the smallest eigenvalue of $WW^*$ cannot be smaller that $1/12$, we have that $\|(\cstY-H + \imu  \cstT  W W^{*})^{-1} \| \leq \frac{1}{12 \cstT}$ for all $\cstY \in \CC_{+}\cup \RR$ and thus the model \eqref{eq:3} is well defined on $\Theta_N$. 
  In the next lemma we establish an a priori bound for the generalized resolvent $(\Hb_{\cstY} - z J_{8}\otimes I_N)^{-1}$.
  
\begin{lem}[Basic properties of the generalized resolvent] \leavevmode
  \label{lem:trivial-gen-res}
  \begin{itemize}
  \item[(i)] For any $ \cstT>0$ there exists $C_{\cstT}>0$ such that \emph{a.w.o.p.}
    \begin{equation}
      \label{eq:8}
       \big\| (\Hb_{ \cstY} - z J_{8} \otimes I_N)^{-1}\big\|
      \leq
      C_{\cstT} \bigg(1 + \frac{1}{\Im z}\bigg)
    \end{equation}
     uniformly for all $z\in\CC_{+}$ and $\cstY \in \CC_{+}\cup \RR$. 
  \item[(ii)] For all $ \cstY \in \CC_+ \cup \RR$, $z\in \CC_{+}$ and $1\leq i,j\leq N$
    \begin{equation}
      \label{eq:9}
      \Big[\big(\Hb_{ \cstY }-zJ_{8}\otimes I_N\big)^{-1}\Big]_{ij}
      =
      \Big[\big(\Tb_{ \cstY }- zI_N\big)^{-1}\Big]_{ij}
      ,\quad
      1\leq i,j \leq N
      .
    \end{equation}
  \end{itemize}
\end{lem}
\begin{proof}
  Let $\Hb_{ \cstY }^{\mathrm{(init)}}$ be the linearization matrix obtained via the linearization algorithm described in Appendix~\ref{sec:line-algor}
\begin{equation}
    \label{eq:182}
    \Hb_{ \cstY }^{\mathrm{(init)}}
    =
    \left(
      \begin{array}{cccccccc}
        0&0&0&0&0&2 \cstT  W_{2}^{*}&0&0
        \\
        0&0&0&0&W_{1}& -(\cstY -H) & -\sqrt{ \cstT } W_{1} & -\sqrt{ \cstT } W_{2}
        \\
        0&0&0&0&0& -\sqrt{ \cstT } W_{1}^{*} & -\imu & 0
        \\
        0&0&0&0&0& -\sqrt{ \cstT }  W_{2}^{*} & 0 & -\imu
        \\
        0& W_{1}^{*} &0&0&-1&0&0&0
        \\
        2 \cstT W_{2} & -( \overline{\cstY}  -H) & -\sqrt{ \cstT } W_{1} & - \sqrt{ \cstT } W_{2} &0&0&0&0
        \\
        0&-\sqrt{ \cstT } W_{1}^{*} & \imu&0&0&0&0&0
        \\
        0&-\sqrt{ \cstT }  W_{2}^{*} & 0&\imu&0&0&0&0
      \end{array}
    \right)
    .
  \end{equation}
   Denote by $\{E_{ij}\,, 1\leq i,j \leq 8 \,\}$ the standard basis of $\CC^{\,8\times 8}$ 
 \begin{equation*}
   E_{ij}
   =
   \big(\delta_{ki}\delta_{jl}\big)_{k,l=1}^{8}
   ,
 \end{equation*}
 where $\delta_{\alpha \beta}$ is the Kronecker delta.
  Then one can easily check that $\Hb_{ \cstY }$ in \eqref{eq:6} can be obtained from $\Hb_{ \cstY }^{\mathrm{(init)}}$ by applying the following transformation
  \begin{equation}
    \label{eq:10}
    \Hb_{ \cstY }
    =
    \widetilde{T}\widetilde{P}_{78}\widetilde{P}_{23}\widetilde{P}_{34}\widetilde{P}_{67}\widetilde{P}_{28} \Hb_{ \cstY }^{\mathrm{(init)}} \widetilde{P}_{28}\widetilde{P}_{67}\widetilde{P}_{34}\widetilde{P}_{23}\widetilde{P}_{78} \widetilde{T}
  \end{equation}
  with $\widetilde{T} = \mathrm{diag}\big(1,-2\sqrt{ \cstT },\frac{1}{2 \cstT ^{3/2}} , -2\sqrt{ \cstT },-\frac{1}{2 \cstT } , \frac{1}{2 \cstT ^{3/2}}, -2 \cstT , \frac{1}{2 \cstT }\big) \otimes I_{N}$, $\widetilde{P}_{ij}= \big( E_{ij} + E_{ji} + \sum_{l \notin \{i,j\}} E_{ll} \big)  \otimes I_N$.
 Note that all transposition matrices $\widetilde{P}_{ij}$  in~\eqref{eq:10}  leave  the upper-left $N\times N$ block intact.
  Thus \eqref{eq:9} follows from the Definition~\ref{def:linearization} of linearization and the Schur complement formula (see, e.g., \eqref{eq:132}) by taking $\Ac = \CC^{N\times N}$.

  In order to prove the bound \eqref{eq:8},  consider the set $\Theta_N$ defined in \eqref{eq:11}.
  One can see that $\Hb_{ \cstY }^{\mathrm{(init)}}$ satisfies the bound \eqref{eq:8}  by specializing Lemma~\ref{lem:trivialBound} for $\Ac = \CC^{N\times N}$, $x_1 = H$, $y_{1} = W_{1}$, $y_{2} = W_{2}$, $C =  12/\cstT $ and the rational expression $q$ being  \eqref{eq:3} on the set $\Theta_N$, as well as using the standard relation between the operator and max norms, similarly as in, e.g., \eqref{eq:212}.
  On the other hand, $\Hb_{ \cstY }^{\mathrm{(init)}}$ and $\Hb_{ \cstY }$ are related by \eqref{eq:10}.
 For any $R\in \CC^{8\times 8}$, applying the transformation $R \mapsto \widetilde{P}_{ij}  R \widetilde{P}_{ij}$ does not change the norm of $R$, while applying the map  $R \mapsto \widetilde{T}  R \, \widetilde{T}$ might change the norm by 
  an  irrelevant  non-zero constant factor only. 
  We thus conclude that  $\Hb_{ \cstY }$ also satisfies the bound \eqref{eq:8} with a constant $C_{\cstT}$ being possibly different than the one  for $\Hb_{ \cstY }^{\mathrm{(init)}}$.
  Since the set $\Theta_N$ satisfies \eqref{eq:12}, the bound \eqref{eq:8} holds \emph{a.w.o.p.}
\end{proof}

Define
\begin{equation}
  \label{eq:13}
  \cstU_{1}
  =
  \left(
    \begin{array}{ccc}
      0&0&0
      \\
      0 & 0 & \frac{\imu}{ \cstT } 
      \\
      0 & -\frac{\imu}{ \cstT } & 0
    \end{array}
  \right)
  ,\quad
  \cstU_{2}
  =
  \left(
    \begin{array}{ccc}
      0&0&\frac{\imu}{ \cstT } 
      \\
      0 & -\frac{1}{4 \cstT^2} & 0
      \\
      -\frac{\imu}{ \cstT } & 0 & 0
    \end{array}
  \right)
\end{equation}
and
\begin{equation}
  \label{eq:14}
  \cstU_{3}
  =
  \left(
    \begin{array}{cc}
      0&  \cstY 
      \\
       \overline{\cstY} & 0 
    \end{array}
  \right)
  ,\quad
  \cstU_{4}
  =
  \left(
    \begin{array}{ccc}
      0&1&1
      \\
      1 & 0 & 0
    \end{array}
  \right)
  ,\quad
  \cstU_{5}
  =
  \left(
    \begin{array}{ccc}
      0 & 0 & 1
      \\
      1 & 1 & 0
    \end{array}
  \right)
  .
\end{equation}
With this notation $\Hb_{ \cstY }$ can be rewritten as
\begin{equation}
  \label{eq:16}
  \Hb_{ \cstY }
  =
  K_0( \cstY )\otimes I_N + K_1 \otimes H + L_1\otimes W_1 + L_1^{*}\otimes W_1^* + L_2\otimes W_2 + L_2^{*}\otimes W_2^*
  ,
\end{equation}
where $K_0=K_0( \cstY ),K_1,L_1,L_2 \in \CC^{8\times 8}$ are given by their block structures as
\begin{equation}
  \label{eq:17}
  K_0
  =
  \left(
    \begin{array}{ccc}
      \cstU_{1}&0&0
      \\
      0 & \cstU_{2} & 0
      \\
      0 & 0 &  \cstU_{3}  
    \end{array}
  \right)
  ,\quad
  K_1
  =
  \left(
    \begin{array}{ccc}
      0&0&0
      \\
      0 & 0 & 0
      \\
      0 & 0 & - \cstV_{1} 
    \end{array}
  \right)
  ,\quad
  L_1
  =
  \left(
    \begin{array}{ccc}
      0&0&0
      \\
      0 & 0 & 0
      \\
      0 & \cstU_4 & 0
    \end{array}
  \right)
  ,\quad
  L_2
  =
  \left(
    \begin{array}{ccc}
      0&0&0
      \\
      0 & 0 & 0
      \\
      \cstU_5 & 0 & 0
    \end{array}
  \right)
  ,
\end{equation}
 and $\cstV_{1} = \left(
   \begin{array}{cc}
     0 & 1
     \\
     1 & 0
   \end{array}
\right)$ is the usual Pauli matrix. 
Consider the \emph{Dyson equation for linearization (DEL)}  
\begin{equation}
  \label{eq:18}
  -\frac{1}{\Sol}
  =
  z J_{8} - K_0 (\cstY) + \SuOp\big[ \Sol \big]
\end{equation}
with a linear map $\SuOp: \CC^{\,8\times 8} \rightarrow \CC^{\,8\times 8}$ given by
\begin{equation}
  \label{eq:19}
  \SuOp\big[R\big]
  :=
  K_{1} R K_{1} + L_1 R L_1^* + L_1^* R L_1+ L_2 R L_2^* + L_2^* R L_2
  ,\qquad
  R\in \CC^{\,8\times 8}.
\end{equation}
\begin{lem}[Existence and basic properties of the solution to the DEL \eqref{eq:18}] \label{lem:exist-uniq-sol} \leavevmode
  For any $\cstT>0$, $\cstY \in \CC_{+} \cup \RR$ and $z \in \CC_{+}$ define $\Sol_{z, \cstY } \in \CC^{\,8\times 8}$ as
  \begin{equation}
  \label{eq:20}
  \Sol_{z, \cstY }
  :=
  (\id_8 \otimes \, \tau_{\Sc}) \Big( \big(K_0( \cstY ) - z J_{8}\big)\otimes \scone + K_{1}\otimes \semic + L_1\otimes c_1 +L_1^{*}\otimes c_1^* + L_2 \otimes c_2 + L_2^{*} \otimes c_2^*\Big)^{-1}
  ,
\end{equation}
where $s,c_1,c_2$ are freely independent semicircular and circular elements and $\tau_{\Sc}$ is a tracial state on a $C^*$-probability space $(\Sc, \tau_{\Sc})$ with the unit element $\scone$.
Then
  \begin{itemize}
  \item[(i)] For any $\cstT>0$ there exists $ C_{\cstT}>0$ such that $\Sol_{z, \cstY }$ satisfies the a priori bound
\begin{equation}
  \label{eq:21}
  \|\Sol_{z, \cstY }\|
  \leq
  C_{\cstT} \Big(1 + \frac{1}{\Im z}\Big)
\end{equation}
uniformly for all $\cstY \in \CC_{+}\cup \RR$ and $z\in\CC_{+}$.
\item[(ii)] For any $\cstT>0$, $\cstY \in \CC_{+} \cup \RR$ and $z\in \CC_{+}$, matrix $\Sol_{z, \cstY }$ satisfies the DEL \eqref{eq:18} and has positive semi\-definite imaginary part, $\Im  \Sol_{z, \cstY } \geq 0$.
  Moreover, for all $\cstT>0$ and  $\cstY \in \CC_{+} \cup \RR$, the matrix-valued function $z\mapsto \Sol_{z, \cstY }$ is analytic on $\CC_{+}$.
\item[(iii)] For any $\cstT>0$, $\cstY \in \CC_{+}\cup \RR$ and $z\in \CC_{+}$  function $z\mapsto \Sol_{z, \cstY }$ admits the representation
\begin{equation}
  \label{eq:22}
  \Sol_{z, \cstY }
  =
  \Sol^{\infty}_{ \cstY } + \int_{\RR} \frac{V_{ \cstY }(d\lambda)}{\lambda - z}
  ,
\end{equation}
where $\Sol_{\cstY}^{\infty} \in \CC^{\,8\times 8}$, and  $V_{ \cstY }(d\lambda)$ is a positive semidefinite  matrix-valued measure on $\RR$ with compact support. In particular, $\lim_{z\to\infty} \Sol_{z, \cstY }
  =  \Sol^{\infty}_{ \cstY }$.
\end{itemize}  
\end{lem}
\begin{proof}
    Fix $\cstT>0$ and denote the noncommutative rational (in fact, polynomial) expression 
  \begin{equation}
    \label{eq:204}
    q_{1,\cstY}(x,y_1,y_2,y_1^*,y_2^*)
    :=
    \cstY-x + \imu \cstT (y_1y_1^* + y_2 y_2^*)
    .
  \end{equation}
  Using the fact that  $c_1 c_1^* + c_2 c_2^*$ has free Poisson distribution of rate 2, which in particular implies that $\|(c_1 c_1^* + c_2 c_2^*)^{-1}\| \leq (1-\frac{1}{\sqrt{2}})^{-1}$, we conclude that the triple $(s,c_1,c_2)$ belongs to the effective domain $\Dc_{\qb_{0},\{q_{1,\cstY}\};C}$ with $C = \cstT^{-1} (1-\frac{1}{\sqrt{2}})^{-1}$ \emph{for all} $\cstY \in \CC_{+}\cup \RR$ (see Section~\ref{sec:line-rati-funct} for the corresponding definitions).

Following  the proofs of Lemmas~\ref{lem:invertibility} and \ref{lem:trivialBound} and specializing them to our concrete case, we obtain  that for any fixed $\cstT>0$ there exists $C_{\cstT}>0$ such that
  \begin{equation}
    \label{eq:205}
    \Big\| \big( (K_0(\cstY ) - z J_{8})\otimes \scone + K_{1}\otimes \semic + L_1\otimes c_1 +L_1^{*}\otimes c_1^* + L_2 \otimes c_2 + L_2^{*} \otimes c_2^*\big)^{-1} \Big\|_{\CC^{\,8\times 8} \,\otimes \,\Sc}
    \leq
    C_{\cstT} \Big(1 + \frac{1}{\Im z}\Big)
  \end{equation}
  uniformly for all $\cstY \in \CC_{+}\cup \RR$,  which yields the a priori bound \eqref{eq:21}.
  Part (ii) of Lemma~\ref{lem:exist-uniq-sol}  now follows directly from parts (ii)-(iv) of Lemma~\ref{lem:exist-uniq}.
  Finally, \eqref{eq:22} follows from the representation \eqref{eq:139} in Lemma~\ref{lem:exist-uniq}.
\end{proof}

With \emph{DEL} \eqref{eq:18} we associate the corresponding \emph{stability operator} $\StOp_{z, \cstY }:\CC^{\,8\times 8} \rightarrow \CC^{\,8\times 8}$ given by
\begin{equation}
  \label{eq:23}
  \StOp_{z, \cstY }\big[R \,\big]
  :=
  R - \Sol_{z, \cstY } \, \SuOp \big[R\,\big] \,\Sol_{z, \cstY }
  ,\quad
  R\in \CC^{\,8\times 8}
  .
\end{equation}
The following lemma is directly obtained from Proposition~\ref{pr:glaw} and establishes  Theorem~\ref{thr:Global law} and the weak limits in the part $(i)$ of Theorem~\ref{thm:main} for $\cstS=1$.
\begin{lem}[Global law for $\mu_{\Tb_{ \cstY }}$] \label{lem:global-law}
  For $ \cstY \in \CC_{+} \cup \RR $ the empirical spectral measure $\mu_{\Tb_{ \cstY }}(d\lambda)$ converges weakly in probability (and almost surely) to $\rho_{ \cstY }(d\lambda)$, where $\rho_{ \cstY }(d\lambda):=\la e_1, V_{ \cstY }(d\lambda) \, e_1 \ra$ is the $(1,1)$ component of the matrix-valued measure $V_{ \cstY }(d\lambda)$ from \eqref{eq:22}.
   Moreover,  $\mathrm{supp} \, \rho_{ \cstY }(d\lambda) \subset [0,1]$ for all $\cstY \in \CC_{+}\cup \RR$ and $\rho_{ \cstY }(d\lambda)$ converges weakly to $\rho_{ E }(d\lambda)$ as $\cstY \in \CC_+$ tends to $E\in \RR$. 
\end{lem}
\begin{proof} 
  We apply Proposition~\ref{pr:glaw} to the rational expression in random matrices \eqref{eq:3}.
  By \eqref{eq:185}, for any $\cstY \in \CC_{+}$ and fixed $\cstP>0$ the generalized resolvent of the linearization $(\Hb_{ \cstY } - z J_{8} \otimes I_N)^{-1}$, corresponding to $\GRes_{z}$ in \eqref{eq:185}, converges uniformly on $\{\,z : \, \Im z \geq \cstP^{-1}, \, |z| \leq \cstP \,\}$ to $\Sol_{z,\cstY} \otimes I_N$, corresponding to $\Sol_{z}^{(\mathrm{sc})} \otimes I_N$ in \eqref{eq:185}.
  In particular, the identity \eqref{eq:9}, similarly to \eqref{eq:160} (see also Definition~\ref{defn:stochastic-domination}), implies the pointwise convergence of the Stieltjes transform of the empirical spectral measure of $\Tb_{\cstY}$: for any $\cstP, \varepsilon, D > 0$ there exists a constant $C_{\cstP, \varepsilon,D}>0$ such that
   \begin{equation}
    \label{eq:183}
    \PP\bigg[\,\Big|\frac{1}{N}\Tr\, \Big(\Tb_{\cstY} - z I_N \Big)^{-1} - \Sol_{z, \cstY}(1,1)\Big| \geq \frac{N^{\varepsilon}}{N} \bigg]
    \leq
    \frac{C_{\cstP, \varepsilon,D}}{N^{D}}
  \end{equation}
  uniformly on $\{\,z : \, \Im z \geq \cstP^{-1}, \, |z| \leq \cstP \,\}$ and  $\cstY \in \CC_{+}\cup \RR$. 
  The convergence in~\eqref{eq:183} together with \eqref{eq:22}  imply that the weak convergence
  \begin{equation}
    \label{eq:194}
    \lim_{N\to \infty} \mu_{\Tb_{\cstY}}(d\lambda) = \rho_{ \cstY }(d\lambda)
  \end{equation}
  holds in probability (see, e.g., \cite[Theorem~2.4.4]{AndeGuioZeitBook}).
  
Now we prove  the almost sure convergence.  Take any $z_0\in \CC_{+}$, a sequence $\{z_i\}\subset\CC_{+}$
of different complex numbers with $ z_{i} \to z_{0}$ such that $z_0$ is the accumulation point  of $\{z_i\}$. 
Define  the sequence of events
  \begin{equation}
    \label{eq:241}
    A_N:= \bigg\{\max_{1\leq i \leq N}\Big\{\Big|\frac{1}{N}\Tr\, \Big(\Tb_{\cstY} - z_i I_N \Big)^{-1} - \Sol_{z_{i}, \cstY}(1,1)\Big|\Big\}
    \geq
    \frac{1}{\sqrt{N}}\bigg\}
    .
  \end{equation}
  Then it follows from \eqref{eq:183} and the Borel-Cantelli lemma applied to $\{A_N\}$ that with probability 1
  \begin{equation}
    \label{eq:242}
    \lim_{N\to \infty}\frac{1}{N}\Tr\, \Big(\Tb_{\cstY} - z_i I_N \Big)^{-1} = \Sol_{z_i, \cstY}(1,1)
  \end{equation}
  for all $i\in\NN$.
  Finally, applying the Vitali-Porter theorem we conclude that the weak convergence \eqref{eq:194} holds almost surely.
  
  To prove the bound on the support of $\rho_{ \cstY }(d\lambda)$ note, that the scattering matrix $S( \cstY )$, related to $\Tb_{ \cstY }= T T^{*}$ via \eqref{eq:2}, is unitary.
This implies that all singular values of $T$ are located in the interval $[0,1]$, thus $\mathrm{supp}\, \mu_{\Tb_{ \cstY }} \subset [0,1]$.
But from \eqref{eq:194} we know that the empirical spectral measure of $T T^{*}$ converges weakly to $\rho_{\cstY}(d\lambda)$, which yields $\mathrm{supp} \rho_{\cstY} \subset [0,1]$.

  It follows from \eqref{eq:205} and the definition of $\Sol_{z,\cstY}$ in \eqref{eq:20} that for any $\cstY \in \CC_{+}\cup \RR$, $E\in \RR$ and $z\in \CC_{+}$
  \begin{equation} 
    \label{eq:214}
    \| \Sol_{z,\cstY} - \Sol_{z,E}\|
    \leq
    2\, C_{\cstT}^{2} \,\big|\cstY - E\big|\, \Big(1 + \frac{1}{\Im z}\Big)^2
    .
  \end{equation}
  This implies the pointwise  convergence  
    \begin{equation}
    \label{eq:218}
    \lim_{\cstY \to E} \Sol_{z,\cstY} = \Sol_{z,E}
  \end{equation}
  for all $z\in \CC_{+}$.  
  By \eqref{eq:22} and $\rho_{ \cstY }(d\lambda)=\la e_1, V_{ \cstY }(d\lambda) \, e_1 \ra$, the $(1,1)$-components $\Sol_{z,\cstY}(1,1)$ and $\Sol_{z,E}(1,1)$ define the Stieltjes transforms of the measures $\rho_{\cstY}(d\lambda)$ and $\rho_{E}(d\lambda)$ correspondingly, from which we conclude that \eqref{eq:218}  yields the weak convergence of $\rho_{\cstY}(d\lambda)$ to $\rho_{E}(d\lambda)$ as $\cstY \to E$.
\end{proof}

Note that $\Sol_{z, \cstY }$ is a matrix-valued Herglotz function.
Therefore, from the properties of the (matrix-valued) Herglotz functions (see, e.g., \cite[Theorems 2.2 and 5.4]{GeszTsek00}), the \emph{absolutely continuous} part of $\rho_{ \cstY }(d\lambda)$ is given by the inverse Stieltjes transform of $\Sol_{z, \cstY }(1,1)$ (see Lemma~\ref{lem:exist-uniq})
\begin{equation}
  \label{eq:24}
  \rho_{ \cstY }(\lambda)
  :=
  \lim_{\eta\downarrow 0} \frac{1}{\pi}  \Im \Sol_{\lambda+\imu \eta, \cstY }(1,1)
  .
\end{equation}
We call the function $\rho_{ \cstY }(\lambda)$, defined in \eqref{eq:24}, the \emph{self-consistent density of states} of the solution to the \emph{DEL} \eqref{eq:18}.
It will be shown in Section~\ref{sec:proof-ii-iii} that $\rho_{ \cstY }(d\lambda)$ is in fact purely absolutely continuous, i.e., $\rho_{ \cstY }(d\lambda) = \rho_{ \cstY }(\lambda) d\lambda$.
The statements \emph{(a)} and \emph{(b)} of part \emph{(ii)} of Theorem~\ref{thm:main} will be derived from the study of $\Sol_{z, \cstY }$ for the spectral parameter $z$ being close to the real line.

Note that our particular choice of linearization \eqref{eq:6} allows rewriting the original \emph{DEL} \eqref{eq:18} in a slightly simpler form. 
More precisely, if
\begin{equation}
  \label{eq:25}
  R
  =
  \left(
    \begin{array}{ccc}
      R_{11}& R_{12}& R_{13}
      \\
      R_{21}& R_{22}& R_{23}
      \\
      R_{31}& R_{32}& R_{33}
    \end{array}
  \right)
  \in \CC^{8\times 8}
\end{equation}
with $R_{11}, R_{22} \in \CC^{3\times 3}$ and $R_{33} \in \CC^{2\times 2}$, then \eqref{eq:17} yields
\begin{equation}
  \label{eq:26}
  \SuOp[R]
  =
  \renewcommand\arraystretch{1.4}
  \left( 
    \begin{array}{ccc}
      \cstU_{5}^{t} R_{33} \cstU_{5}& 0 & 0
      \\ 
      0 &\cstU_{4}^{t} R_{33} \cstU_{4}& 0
      \\ 
      0 & 0 &  \cstV_{1}  R_{33}  \cstV_{1} + \cstU_{4} R_{22} \cstU_{4}^{t} + \cstU_{5} R_{11} \cstU_{5}^{t}
    \end{array}
  \right)
  ,
\end{equation}
so that the image $\SuOp[R]$ is a block-diagonal matrix.
Together with the definition of $K_0$ in \eqref{eq:17}, this implies that the right-hand side in \eqref{eq:18} is a block-diagonal matrix with blocks of sizes 3, 3, and 2 correspondingly.
We conclude that any solution to the \emph{DEL} \eqref{eq:18} has a block-diagonal form, which, in particular, allows us to write
\begin{equation}
  \label{eq:27}
  \Sol_{z, \cstY }
  =
  \left(
    \begin{array}{ccc}
      \Soll{1} & 0 & 0
      \\
      0 & \Soll{2} & 0
      \\
      0 & 0 & \Soll{3}
    \end{array}
  \right)
\end{equation}
with $\Soll{1},\Soll{2} \in \CC^{3\times 3}$ and $\Soll{3} \in \CC^{2\times 2}$, where we omit the dependence of the blocks on $z$ and $ \cstY $.
Now \emph{DEL}~\eqref{eq:18} can be decomposed into a system of three matrix equations  of smaller dimensions  
\begin{equation}
  \label{eq:28}
  -\frac{1}{\Soll{1}}
  =
  zJ_{3}-\cstU_{1}+  \cstU_{5}^t \Soll{3} \cstU_{5}
  ,\quad
  -\frac{1}{\Soll{2}}
  =
  -\cstU_{2}+ \cstU_{4}^t \Soll{3} \cstU_{4}
\end{equation}
and
\begin{equation}
  \label{eq:29}
  -\frac{1}{\Soll{3}}
  =
  - \cstU_{3}  (\cstY)  
  - \frac{ 1 }{ -\frac{1}{2 \cstT^2z}(I_2 + \cstV_3)-\frac{1}{ \cstT }\cstV_2 +  \Soll{3}}
  - \frac{1}{-2(I_2 - \cstV_3)-\frac{1}{ \cstT }\cstV_2+ \Soll{3} }
  + \cstV_{1} \Soll{3} \cstV_{1}
  ,
\end{equation}
where $\cstV_{1} = \left(
  \begin{array}{cc}
    0 & 1
    \\
    1 & 0
  \end{array}
\right)$, $\cstV_{2}= \left(
  \begin{array}{cc}
    0 & -\imu
    \\
    \imu & 0
  \end{array}
\right)$ and $\cstV_{3}= \left(
  \begin{array}{cc}
    1 & 0
    \\
    0 & -1
  \end{array}
\right)$ are the standard Pauli matrices. 
The proof of Theorem~\ref{thm:main} is based on the study of matrices $\Soll{1}$, $\Soll{2}$ and $\Soll{3}$.

\subsection{Useful identities}
From now on and until the end of Section~\ref{sec:proofs}, we study the 
matrix-valued function $\Sol_{z,\cstY}$ with $\cstY = E \in \RR$.
We start by collecting several important relations between the components of $\Sol_{z,E}$.
\label{sec:usefull-identities}
\begin{lem}
  \label{lem:useful-identities-3}
  For all $E\in \RR$ and $z\in \CC_{+}$
  \begin{equation}
    \label{eq:30}
    \Sol_{z,-E}
    =
     ( Q^{-}\,\Sol_{z,E} \, Q^{-})^{t}
    ,
  \end{equation}
where $Q^{-} = \mathrm{diag} (-1,-1,1,-1,1,1,-1,1) \in \CC^{8\times 8} $.
In particular, for all $E\in \RR$, $z\in\CC_{+}$ and $1\leq k\leq 8$
\begin{equation}
  \label{eq:41}
  \Sol_{z,E}(k,k)
  =
  \Sol_{z,-E}(k,k)
  .
\end{equation}
\end{lem}
\begin{proof}
Let $\Hb_{E}^{(\mathrm{sc})} \in \Sc^{8\times 8}$ be given by 
  \begin{equation}
\label{eq:32}
  \Hb_{E}^{(\mathrm{sc})}
  :=  
  \left(
    \renewcommand\arraystretch{1.4}
    \begin{array}{ccc|ccc|cc}
 0 & 0 & 0 & 0 & 0 & 0 & 0 & c_{2}^* \\
 0 & 0 & \frac{\imu}{ \cstT } \scone & 0 & 0 & 0 & 0 & c_{2}^* \\
 0 & -\frac{\imu}{ \cstT } \scone & 0 & 0 & 0 & 0 & c_{2}^* & 0 \\ \hline
 0 & 0 & 0 & 0 & 0 & \frac{\imu}{ \cstT }\scone & 0 & c_{1}^* \\
 0 & 0 & 0 & 0 & -\frac{1}{4 \cstT^2} \scone & 0 & c_{1}^* & 0 \\
 0 & 0 & 0 & -\frac{\imu}{ \cstT } \scone & 0 & 0 & c_{1}^* & 0 \\ \hline
 0 & 0 & c_{2} & 0 & c_{1} & c_{1} & 0 & E \scone -\semic \\
 c_{2} & c_{2} & 0 & c_{1} & 0 & 0 & E \scone -\semic & 0 \\
\end{array}
\right)
,
\end{equation}
where $\semic$ is a semicircular element, $c_1,c_2$ are circular elements, all freely independent in a $C^*$-probability space $(\Sc,\tau_{\Sc})$, so that
\begin{equation}
  \label{eq:33}
  \Sol_{z,E}
  :=
  (\id_{8}\otimes \,\tau_{\Sc}) (\Hb_{E}^{(\mathrm{sc})} - zJ_{8}\otimes \scone)^{-1}
  .
\end{equation}
Using the fact that $-\semic$,  $-c_{1}^*$ and $-c_{2}^*$ form again a freely independent family of one semicircular and two circular elements, we can easily check that (here $\times$ denotes multiplication in $\Sc ^{8\times 8}$)
\begin{equation}
  \label{eq:39}
  (\id_{8}\otimes \,\tau_{\Sc})\Big(\big((Q^{-}\otimes \scone ) \times (\Hb_{-E}^{(\mathrm{sc})} - zJ_{8}\otimes \scone ) \times  (Q^{-}\otimes \scone )\big)^{t}\Big)^{-1}
  =
  \Sol_{z,E}
  ,
\end{equation}
from which \eqref{eq:30} follows after factorizing $Q^{-}$.
\end{proof}
\begin{lem}
  \label{lem:useful-identities-4}
For all $E\in \RR$ and $z\in\CC_{+}$
\begin{equation}
  \label{eq:31}
  \Sol_{z,E}(8,8)
  =
  4 \cstT^2 z\, \Sol_{z,E}(7,7)
  .
\end{equation}
\end{lem}
\begin{proof}
Using the Schur complement formula, the lower-right $2\times 2$ block of the inverse of $\Hb^{\mathrm{(sc)}}_{E} - zJ_{8} \otimes \scone $ can be written as
\begin{equation}
  \label{eq:43}
  \left(
    \begin{array}{cc}
      4  \cstT^2 c_{1} c_{1}^* & E \scone -\semic+ \imu  \cstT (c_{1}c_{1}^* + c_{2}c_{2}^*)
                              \\
      E \scone -\semic- \imu  \cstT (c_{1}c_{1}^* + c_{2}c_{2}^*) & \frac{1}{z} c_{2}c_{2}^*
    \end{array}
  \right)^{-1}
  .
\end{equation}
For convenience, change the rows in the above matrix, so that
\begin{equation}
  \label{eq:46}
  \Soll{3}  \left(
    \begin{array}{cc}
    0&1
    \\
    1&0
    \end{array}
       \right)
  =
  (\id_2\otimes \,\tau_{\Sc}) \left[\left(
    \begin{array}{cc}
      E \scone -a & \frac{1}{z} c_{2}c_{2}^*
      \\
      4  \cstT^2 c_{1} c_{1}^* & E \scone -a^*
    \end{array}
  \right)^{-1}  \right]
       ,
\end{equation}
where we introduced
\begin{equation}
  \label{eq:47}
  a
  :=
  \semic+\imu  \cstT  (c_{1}c_{1}^* + c_{2}c_{2}^*)
  .
\end{equation}
Notice, that since $c_{1}c_{1}^* + c_{2}c_{2}^*$ has a free Poisson distribution of rate 2, $c_{1}c_{1}^* + c_{2}c_{2}^* \geq (1-\frac{1}{\sqrt{2}})^2 \scone $ and thus both diagonal elements of the matrix on the right-hand side of \eqref{eq:46} are invertible.
Rewrite the matrix in the square brackets in the following way: for the entries of the first row apply the Schur complement formula with respect to the $(1,1)$-component, and for the second row apply the Schur complement formula with respect to the $(2,2)$-component.
This leads to the following expressions for $\Sol_{z,E}(7,7)$ and $\Sol_{z,E}(8,8)$
\begin{align}
  \label{eq:49}
  &\Sol_{z,E}(7,7)
  =
    \frac{1}{z}\tau_{\Sc} \Bigg(-\frac{1}{E \scone -a} c_{2}c_{2}^*\frac{1}{E \scone -a^*-\frac{4 \cstT^2}{z} c_{1}c_{1}^*\frac{1}{E \scone -a}c_{2}c_{2}^*} \,\Bigg)
    ,\\
  &\label{eq:50}
  \Sol_{z,E}(8,8)
  =
  4 \cstT^2 \tau_{\Sc}\Bigg(-\frac{1}{E \scone -a^*} c_{1}c_{1}^*\frac{1}{E \scone -a-\frac{4 \cstT^2}{z} c_{2}c_{2}^*\frac{1}{E \scone -a^*}c_{1}c_{1}^*} \, \Bigg)
  .
\end{align}
Under $\tau_{\Sc}$  we can swap the labels of $c_1$ and $c_2$ and replace $a$ with $-a^*$ without changing the value in \eqref{eq:50}.
After completing these operations, we obtain
\begin{equation}
  \label{eq:53}
  \Sol_{z,E}(8,8)
  =
  4 \cstT^2 \tau_{\Sc}\Bigg(-\frac{1}{E \scone +a} c_{2}c_{2}^*\frac{1}{E \scone +a^*-\frac{4 \cstT^2}{z} c_{1}c_{1}^*\frac{1}{E \scone +a}c_{2}c_{2}^*} \,\Bigg)
  .
\end{equation}
Multiplying both fractions under $\tau_{\Sc}$ in \eqref{eq:53} by $-1$, and swapping $E$ to $-E$ by \eqref{eq:41}, a comparison with \eqref{eq:49} yields \eqref{eq:31}.
\end{proof}
\begin{lem}
  \label{lem:useful-identities-5}
  For all $E\in \RR$ and $z\in\CC_{+}$
  \begin{equation}
    \label{eq:169}
    \Sol_{z,E}(8,7)-\Sol_{z,E}(7,8)
    =
    \frac{\imu}{ \cstT } \Sol_{z,E}(8,8)
  \end{equation}
\end{lem}
\begin{proof}
Denote
\begin{equation}
  \label{eq:56}
  T_1
  :=
  \left(
    \begin{array}{cc}
      -\frac{1}{ \cstT^2 z} & \frac{\imu}{ \cstT }
      \\
      -\frac{\imu}{ \cstT } & 0
    \end{array}
  \right)
  +
  \Soll{3}
  ,\quad
    T_{2}
  :=
  \left(
    \begin{array}{cc}
      0 & \frac{\imu}{ \cstT }
      \\
      -\frac{\imu}{ \cstT } & -4
    \end{array}
  \right)
  +
  \Soll{3}
  ,
\end{equation}
so that \eqref{eq:29} can be rewritten as
\begin{equation}
  \label{eq:57}
  \frac{1}{\Soll{3}}- E \cstV_{1} - \frac{1}{T_{1}} - \frac{1}{T_{2}} + \cstV_{1} \Soll{3} \cstV_{1}
  =
  0
  .
\end{equation}
  Then from \eqref{eq:31} we have that
\begin{equation}
  \label{eq:59}
  \det T_{1} - \det T_{2}
  =
  4\Sol_{z,E}(7,7)- \frac{1}{ \cstT^2 z} \Sol_{z,E}(8,8)
  =
  0
\end{equation}
Rewrite \eqref{eq:57} componentwise using \eqref{eq:59}
\begin{align}
  \label{eq:61}
  \Big(\frac{1}{\det \Soll{3}} -\frac{2}{\det T_{1}} + 1 \Big) \Sol_{z,E}(8,8) & = -\frac{4}{\det T_1}
                                                                      ,
  \\ \label{eq:63}
  \Big( -\frac{1}{\det \Soll{3}}+\frac{2}{\det T_{1}} \Big) \Sol_{z,E}(7,8) + \Sol_{z,E}(8,7)& = E - \frac{2 \imu}{ \cstT } \frac{1}{\det T_1 }
                                                                             ,
  \\ \label{eq:64}
  \Big( -\frac{1}{\det \Soll{3}}+\frac{2}{\det T_{1}} \Big) \Sol_{z,E}(8,7) + \Sol_{z,E}(7,8)& = E + \frac{2 \imu}{ \cstT } \frac{1}{\det T_1 },
  \\ \label{eq:62}
   \Big(\frac{1}{\det \Soll{3}} -\frac{2}{\det T_{1}} + 1 \Big) \Sol_{z,E}(7,7) & = -\frac{1}{ \cstT^2 z\det T_1}
                                                                      .
\end{align}
Subtracting \eqref{eq:64} from \eqref{eq:63} gives
\begin{equation}
  \label{eq:65}
  \Big( \frac{1}{\det \Soll{3}} - \frac{2}{\det T_{1}} + 1\Big) (\Sol_{z,E}(8,7)-\Sol_{z,E}(7,8)) 
  =
  - \frac{ \imu}{ \cstT } \frac{4}{\det T_1 }
  ,
\end{equation}
which together with \eqref{eq:61} implies \eqref{eq:169}.
\end{proof}

\subsection{Boundedness of $\Sol_{z,E}$  away from $z=0$ and $z=1$ }
\label{sec:proof-boundedness}
\begin{lem}[Boundedness of $ \Sol_{z,E} $]\label{lem:boundedness}
  For any  small $\epsA>0$ there exists $\cstC>0$ such that 
  \begin{equation}
    \label{eq:67}
    \sup \bigg\{\,\|\Sol_{z,E}\|  \, : \, | z |\geq \epsA,   |1- z | \geq \epsA, \, \Im z > 0,\,  |E|\leq \frac{1}{\epsA} \,\bigg\} 
    \leq
    \cstC
    .
  \end{equation}
\end{lem}
\begin{proof}
Introduce the following notation for the entries of $\Sol_{3}$
\begin{equation}
  \label{eq:58}
  \left(
    \begin{array}{cc}
      m_{11}& m_{12}
      \\
      m_{21}& m_{22}
    \end{array}
  \right)
  :=
  \left(
    \begin{array}{cc}
      \Sol_{z,E}(7,7) & \Sol_{z,E}(7,8)
      \\
      \Sol_{z,E}(8,7) & \Sol_{z,E}(8,8)
    \end{array}
  \right)
  ,
\end{equation}
so that, in particular, \eqref{eq:31} and \eqref{eq:169} can be rewritten as
\begin{align}
  \label{eq:197}
  m_{22}
  &=
    4 \cstT^2 z m_{11}
    ,
  \\ \label{eq:198}
  m_{21} - m_{12}
  &=
    \frac{\imu}{ \cstT } m_{22}
    .
\end{align}
Our goal is to show that  $\Sol_{z,E}$ and in particular $\Sol_{z,E}(1,1)$  is bounded everywhere if $  z $ is away from $0$ or $1$.
From \eqref{eq:28} we have that 
\begin{equation}
  \label{eq:78}
  \Soll{1}
  =
  -\left(
    \begin{array}{ccc}
      m_{22} + z & m_{22} & m_{21}
      \\
      m_{22} & m_{22} & m_{21} - \frac{\imu}{ \cstT }
      \\
      m_{12} & m_{12} + \frac{\imu}{ \cstT } & m_{11}
    \end{array}
  \right)^{-1}
  ,\quad
   \Soll{2}
  =
  -\left(
    \begin{array}{ccc}
      m_{22}  & m_{21} & m_{21} - \frac{\imu}{ \cstT }
      \\
      m_{12} & m_{11} + \frac{1}{4 \cstT^2} & m_{11} 
      \\
      m_{12}  + \frac{\imu}{ \cstT } & m_{11} & m_{11}
    \end{array}
  \right)^{-1}
 ,
\end{equation}
which after some elementary computations yields
\begin{equation}
  \label{eq:150}
  \det \Soll{1}
  =
  \frac{-1}{z \det T_1}
  ,\quad
  \det \Soll{2}
  =
  \frac{- \cstT^2}{\det T_1}
\end{equation}
and 
\begin{align}
  \label{eq:79}
  \Sol_{z,E}(1,1)
  =
    -\frac{1}{z} - \frac{4m_{11}}{z\det T_{1}}
    ,
\end{align}
where $T_{1}$ was defined in \eqref{eq:56}.
The functions $\{m_{ij}, 1\leq i,j \leq 2\}$, $\{\Sol_{i}, 1\leq i \leq 3\}$ and $\{T_{i}, 1\leq i \leq 2\}$ defined above all depend on the variables $z$ and $E$, but in order to make the exposition lighter  we drop the explicit dependence on these variables from the notation. 
Using \eqref{eq:78}-\eqref{eq:79} it is  enough to show that for any fixed  $(z_{\infty},E_{\infty})$ with $z_{ \infty } \in \overline{\CC_{+}}$, $z_{\infty}  \notin \{0,1\}$ and  $E_{\infty} \in \RR$ we have that  $\lim_{(z,E)\rightarrow (z_{\infty}, E_{\infty})}  |m_{ ij }| <\infty$ for $i,j\in \{1,2\}$  and  $\lim_{(z,E)\rightarrow (z_{\infty}, E_{\infty})}|\det T_{1}|>0$.

We now prove some additional relations that can be obtained from \eqref{eq:197}, \eqref{eq:198} and \eqref{eq:61}-\eqref{eq:62}.
Plugging \eqref{eq:197} and \eqref{eq:198} into \eqref{eq:63} (recall that we are using notation \eqref{eq:58}) gives
\begin{equation}
  \label{eq:81}
  \Big(-\frac{1}{\det \Soll{3}}+ \frac{2}{\det T_{1} }+1\Big) m_{12} + 4  \cstT \imu z m_{11} +  \frac{2\imu}{ \cstT \det T_{1}} - E
  =
  0
  ,
\end{equation}
which, after applying \eqref{eq:62} to the terms in the parenthesis, can be rewritten as
\begin{equation}
  \label{eq:82}
  \Big( \frac{1}{ \cstT^2 z m_{11} \det T_{1}} +2\Big) m_{12}
  =
  E - \frac{2 \imu}{ \cstT \det T_{1}} - 4  \cstT \imu z m_{11}
  .
\end{equation}
From the definitions of $T_{1}$ and $\Soll{3}$ we have
\begin{equation}
  \label{eq:83}
  \det T_{1} = \det \Soll{3} + 4 (z-1) m_{11} - \frac{1}{ \cstT^2}
  ,
\end{equation}
while \eqref{eq:62} gives
\begin{equation}
  \label{eq:84}
  \frac{1}{\det T_{1}} \Big( 2- \frac{1}{ \cstT^2 z m_{11}} \Big)
  =
  \frac{1}{\det \Soll{3}} + 1
  .
\end{equation}
Combining \eqref{eq:83} and \eqref{eq:84}, we get the following quadratic equation for $\det \Soll{3}$
\begin{equation}
  \label{eq:85}
  \big(\det \Soll{3}\big)^2 + \det \Soll{3}\bigg( 4(z-1) m_{11} + \frac{1}{ \cstT^2 z m_{11}} -\bigg(1+\frac{1}{ \cstT^2}\bigg)\bigg) + \bigg( 4(z-1)m_{11} - \frac{1}{ \cstT^2}\bigg)
  =
  0
  .
\end{equation}
Note, that \eqref{eq:82}-\eqref{eq:85} hold for all $E\in \RR$ and all $z\in\CC_{+}$.

Using the above relations, we proceed with a proof by contradiction.
Assume that there exists a sequence  $(z_{n},E_{n})_{n=1}^{\infty}\subset \CC_{+}\times [-\epsA^{-1},\epsA^{-1}]$, such that $|m_{11}^{(n)}| \rightarrow \infty$ as $n\rightarrow \infty$ (here and below we denote the evaluations at  $(z_{n},E_{n})$ by adding the superscript $(n)$).
 Solving \eqref{eq:85} for $\det \Soll{3}$ allows us to express $\det \Soll{3}$ in terms of $m_{11}$
\begin{multline}
  \label{eq:87}
  \det \Soll{3}
    =
    \frac{1}{2} \bigg\{-4(z-1)m_{11}+\Big(1+\frac{1}{ \cstT^2}\Big)
  \\
  \quad
    \pm 4(z-1)m_{11} \bigg[1 -\frac{1}{2}\bigg(\frac{1}{2(z-1)}\Big(1+\frac{1}{ \cstT^2}\Big)+\frac{1}{z-1}\bigg)\frac{1}{m_{11}} +\OO{\frac{1}{|m_{11}|^2}}\bigg] \bigg\}
  .
\end{multline}
By passing to a subsequence, we may assume that the choice of the $\pm$ sign in \eqref{eq:87} is constant for all $n$.

If we take the $-$ sign in \eqref{eq:87}, then
\begin{equation}
  \label{eq:88}
  \det \Soll{3}^{(n)}
  =
  -4(z_{n}-1)m_{11}^{(n)} + \OO{1}
\end{equation}
and by \eqref{eq:84}
\begin{equation}
  \label{eq:89}
  \det T_{1}^{(n)}
  \rightarrow
  2
  ,\quad
  n\rightarrow \infty
  .
\end{equation}
From \eqref{eq:82}, \eqref{eq:197}, \eqref{eq:198} and \eqref{eq:89},
\begin{equation}
  \label{eq:90}
  m_{12}^{(n)}
  =
  -2 \cstT \imu z_{n} m_{11}^{(n)} + \OO{1}
  ,\quad
  m_{21}^{(n)}
  =
  2 \cstT \imu z_{n} m_{11}^{(n)} + \OO{1}
  ,
\end{equation}
and therefore
\begin{equation}
  \label{eq:91}
  \det \Soll{3}^{(n)}
  =
  4 \cstT^2 z_{n} \big(m_{11}^{(n)}\big)^2-4 \cstT^2z_{n}^2 \big(m_{11}^{(n)}\big)^2
  =
  4 \cstT^2 z_{n}(1-z_{n})\big(m_{11}^{(n)}\big)^2
  +
  \OO{m_{11}^{(n)}}
  ,
\end{equation}
which contradicts to \eqref{eq:88} since $|z_{n}(1-z_{n})|$ is separated away from $0$.

If we take $+$ sign in \eqref{eq:87}, then
\begin{equation}
  \label{eq:92}
  \det \Soll{3}^{(n)}
  =
  -1   +\OO{\frac{1}{|m_{11}^{(n)}|}} 
  ,
\end{equation}
and from \eqref{eq:83}
\begin{equation}
  \label{eq:93}
  \det T_{1}^{(n)}
  =
  4(z_{n}-1)m_{11}^{(n)} + \OO{1}
  .
\end{equation}
But then again, from \eqref{eq:82}, \eqref{eq:197}, \eqref{eq:198} and \eqref{eq:93},
\begin{equation}
\label{eq:94}
  m_{12}^{(n)}
  =
  -2 \cstT \imu z_{n} m_{11}^{(n)} + \OO{1}
  ,\quad
  m_{21}^{(n)}
  =
  2 \cstT \imu z_{n} m_{11}^{(n)} + \OO{1},
\end{equation}
so that
\begin{equation}
\label{eq:95}
  \det \Soll{3}^{(n)}
  =
  4 \cstT^2 z_{n}(1-z_{n}) (m_{11}^{(n)})^2 + \OO{|m_{11}^{(n)}|}
  ,
\end{equation}
which contradicts to \eqref{eq:92}.
Therefore, we have proven that $m_{11}$ is bounded everywhere away from the points $z\in \{0,1\}$.
 It is clear from \eqref{eq:197} that the boundedness of $m_{11}$ guarantees the boundedness of $m_{22}$.
On the other hand, assuming that $m_{12}$ and $m_{21}$ (see \eqref{eq:198}) are unbounded implies that $|\det \Soll{3}^{(n)}|\rightarrow \infty$ and $|\det T_1^{(n)}|\rightarrow \infty$  on some sequence  $(z_n,E_n)_{n=1}^{\infty}$, which contradicts to the boundedness of $m_{11}$ in \eqref{eq:81}.
We conclude that all entries of $\Soll{3}$ are bounded everywhere away from the points $z\in \{0,1\}$. 

Assume now that there exists a sequence  $(z_{n},E_{n})_{n=1}^{\infty}\subset \CC_{+}\times [-\epsA^{-1},\epsA^{-1}]$ such that $\det T_{1}^{(n)} \rightarrow 0$ as $n\rightarrow \infty$.
Then by \eqref{eq:82}
\begin{equation}
  \label{eq:96}
  \bigg(\frac{1}{ \cstT^2 z_{n} m_{11}^{(n)}   }+\OO{\det T_{1}^{(n)}}\bigg)m_{12}^{(n)}
  =
  -\frac{2\imu}{ \cstT } + \OO{\det{T_{1}}^{(n)}}
  ,
\end{equation}
which together with the fact that $m_{11}^{(n)}$ is bounded implies that 
\begin{equation}
  \label{eq:97}
  m_{12}^{(n)}
  =
  -2\imu  \cstT z_{n} m_{11}^{(n)} + \OO{\det T_{1}^{(n)}}
  .
\end{equation}
Then by \eqref{eq:197} and \eqref{eq:198} we find
\begin{equation}
  \label{eq:98}
  \det \Soll{3}^{(n)}
  =
  4 \cstT^2 z_{n}(1-z_{n}) \big(m_{11}^{(n)}\big)^2 + \OO{\det T_{1}^{(n)}}
  ,
\end{equation}
and  conclude from \eqref{eq:98} and \eqref{eq:83} that $\det \Soll{3}^{(n)}$ does not vanish as $n\rightarrow \infty$.
But then  \eqref{eq:84} implies
\begin{equation}
  \label{eq:99}
  m_{11}^{(n)}
  =
  \frac{1}{2 \cstT^2 z_{n} } + \OO{\det T_{1}^{(n)}}
  ,
\end{equation}
and thus by \eqref{eq:83} and \eqref{eq:99} also
\begin{equation}
  \label{eq:100}
  \det \Soll{3}^{(n)}
  =
  \frac{1}{ \cstT^2} - 2(z_{n}-1) \frac{1}{ \cstT^2 z_{n} } + \OO{\det T_{1}^{(n)}},
\end{equation}
as well as 
\begin{equation}
  \label{eq:101}
  \det \Soll{3}^{(n)}
  =
  4 \cstT^2 z_{n}(1-z_{n}) \frac{1}{4 \cstT^4 z_{n}^2 } + \OO{\det T_{1}^{(n)}}
  =
  (1-z_{n}) \frac{1}{ \cstT^2 z_{n} } + \OO{\det T_{1}^{(n)}}
  ,
 \end{equation}
 by \eqref{eq:98} and \eqref{eq:99},
 which contradicts to \eqref{eq:100} since $z_{n}$ is away from $0$.
 We conclude that $|\det T_{1}|$ is bounded away from $0$, which together with \eqref{eq:79} establishes (a non-effective version of) \eqref{eq:67}.

 In order to keep the presentation simple, above we showed that $\|\Sol_{z,E}\|$ is bounded for $z\in \overline{\CC_{+}} \setminus \{0,1\}$ and $E\in \RR$ without providing an explicit effective bound $\cstC$ as formulated in \eqref{eq:67}.
  Note that the constants hidden in the $\OO{\cdot}$ terms (for example, in \eqref{eq:88}-\eqref{eq:91}, \eqref{eq:92}-\eqref{eq:95} or \eqref{eq:96}-\eqref{eq:101}) depend only on $|z_n|$, $|1-z_n|$ and  $E_n$.
  Therefore, using the assumptions $|m_{11}^{(n)}| \geq \cstD$ and $|\det T_{1}^{(n)}|\leq 1/\cstD$ for $n$ large enough and carefully chosen $\cstD >0$ instead of $|m_{11}^{(n)}| \rightarrow \infty$ and $|\det T_{1}^{(n)}|\rightarrow 0$ as $n\rightarrow \infty$, together with the a priori bound \eqref{eq:21}, eventually leads to the uniform bound \eqref{eq:67}.
\end{proof}

\subsection{Singularities of $\Sol_{z,E}(1,1)$}
\label{sec:proof-edge-asympt}

\begin{lem}[Behavior of $\Sol_{z,E}(1,1)$ in the vicinities of $z=0$ and $z=1$]\leavevmode \label{lem:sing-sol}
  \begin{itemize}
  \item[(a)] For all $E\in \RR$
  \begin{equation}
    \label{eq:175}
    \Sol_{z,E}(1,1)
    =
    \sqrt[3]{\frac{1+E^2}{4 \cstT^{2}}}\,z^{-2/3}  + \OO{\frac{1}{|z|^{1/3}}}
    \mbox{ as }
      z\rightarrow 0
    ;
  \end{equation}
\item[(b)] for all $|E| < \frac{1}{ \cstT } \sqrt{2 \sqrt{1 + 6  \cstT^2 +  \cstT^4}-2-2  \cstT^2}  $
    \begin{equation}
    \label{eq:176}
    \Sol_{z,E}(1,1)
    =
    - \frac{16  \cstT^2 \xi_{0}}{ \cstT^2 E^2 + 4+16 \cstT^4 \xi_{0}^{2}}\, (z-1)^{-1/2} + \OO{1}
    \mbox{ as }
    z\rightarrow 1
    ,
  \end{equation}
  where the constant $\xi_{0}<0$ is as in part (ii) of Theorem~\ref{thm:main} (see also \eqref{eq:170} below).
  \end{itemize}
  The choices of the branches for $(\cdot)^{1/3}$ and $(\cdot)^{1/2}$ are specified in the course of the proof below.
\end{lem}
\begin{proof}
We multiply \eqref{eq:29} from the left by $\Soll{3}$ and from the right by $(Z_1 + \Soll{3})(Z_1-Z_2)^{-1}(Z_2+\Soll{3})$ with the short hand notation 
\begin{equation}
  \label{eq:104}
  Z_1
  :=
  -\frac{1}{2 \cstT^2 z}(I_2+\cstV_3) -\frac{1}{ \cstT }\cstV_2
  \,, \qquad
  Z_2
  :=
  -2(I_2-\cstV_3) -\frac{1}{ \cstT }\cstV_2
  \,.  
\end{equation}
Subsequently, the equation for $\Soll{3}$ takes the form $\Delta =0$ with 
\begin{equation}
  \label{eq:105}
    \Delta
    :=
    \Big(Z_1-\Soll{3} + \Soll{3} \cstV_1(\Soll{3}\cstV_1-E)(Z_1+\Soll{3})\Big)\frac{1}{Z_1-Z_2}(Z_2+\Soll{3}) -\Soll{3}
    \,.
\end{equation}
Using that $(Z_1-Z_2)^{-1} = - \frac{ \cstT^2 z}{2}(I_2+\cstV_3) +\frac{1}{8}(I_2-\cstV_3)$ and performing the  matrix products we see that the entries of $\Delta \in \CC^{2 \times 2}$ are polynomials in the entries of $\Soll{3}$, $E$ and $z$. 

\emph{Step 1: Expansion around $z=0$}.
We will now first construct a solution to \eqref{eq:29} in a vicinity of $z=0$ by asymptotic expansion.
Later, in \emph{Step 3}, we will show that the constructed solution coincides with $\Soll{3}$ defined in \eqref{eq:20} and \eqref{eq:27}.
For this purpose we write $t = z^{1/3}$ with an analytic cubic root on $\CC \setminus (\imu (-\infty,0]))$ such that $(-1)^{1/3} =-1$.
Then we make the ansatz 
\begin{equation}
  \label{eq:106}
  \widetilde{\Sol}_3
  =
  \left(
    \begin{array}{cc}
        \xi_{11} t^{-1}& 4  \cstT^2 \xi_{12} t
        \\
        4  \cstT^2 \xi_{21} t  & 4  \cstT^2 \xi_{22} t^2 
      \end{array}
    \right)
    \,,
\end{equation}
where we will determine the unknown functions $\Xi_t:=(\xi_{ij}(t))_{i,j=1}^2$.
Plugging \eqref{eq:106} into \eqref{eq:105} reveals that
\begin{equation}
  \label{eq:34}
  \Delta
  = 
  \left(
    \begin{array}{cc}
      4  \cstT^2  ( q_{11} + t p_{11}) & 4  \cstT^2t(q_{12} + t p_{12})
      \\
      -4  \cstT^2t(q_{21} + t p_{21}) & -16 \cstT^4 t^3 (q_{22} + t p_{22})
    \end{array}
  \right)
  \,,
\end{equation}
where the entries of  $P:=(p_{ij})_{i,j=1}^2$ are polynomials in $t,E,\xi_{ij}$  and where $Q:=(q_{ij})_{i,j=1}^2$ is given by
\begin{equation}
  \label{eq:36}
  Q
  =
  \left(
    \begin{array}{cc}
      \frac{1}{16  \cstT^4} + \xi_{11}^2 \xi_{22} -  E \xi_{11} \xi_{12} & \xi_{12}  +E \xi_{11} \xi_{22}
      \\
      \xi_{21} + E \xi_{11} \xi_{22} & -\frac{1}{16 \cstT^4}- \xi_{11} \xi_{22}^2+E \xi_{21} \xi_{22}
    \end{array}
  \right)
  \,.  
\end{equation}
Thus the equation $\Delta =0$ is equivalent to $Q+t P =0$.
For $t=0$ the three possible solutions are
\begin{equation}
  \label{eq:107}
  \begin{split}
    \Xi_0 =
    \left(
      \begin{array}{cc}
        \zeta  & -E \zeta^2
        \\
        -E \zeta^2& \zeta
      \end{array}
    \right)
    \,, \quad
    \text{where}
    \quad
    (1+E^2) \zeta^3
    =
    -\frac{1}{16  \cstT^4}\,.
  \end{split}
\end{equation}
Since $\det\nabla_{\Xi} Q|_{t=0} = -3 \zeta^4(1+E^2)^2$ the equation $Q+tP = 0$ is linearly stable at $t=0$ and can thus be solved for $\Xi_t$ in a neighborhood of $\Xi_0$ when $t$ is sufficiently small.
Since the equation is polynomial, the solution $\Xi_t$ admits a power series expansion in $t$.
Now we define the analytic function $\widetilde{\Sol}_3(z)$ through \eqref{eq:106} on a neighborhood of $z=0$ in $\CC \setminus (\imu (-\infty,0]))$ with $\Xi$ the solution to $Q+tP = 0$ and the choice $\zeta <0$ in \eqref{eq:107}.  

We will now check that for any phase $e^{\imu \psi} \ne -1$ in the complex upper half-plane the imaginary part of $\widetilde{\Sol}_3(\theta e^{\imu \psi})$ is positive definite in the sense that for any fixed $\varepsilon>0$ we have 
\begin{equation}
  \label{eq:37}
  \inf_{\psi \in [-\pi +\varepsilon,0]}\inf_{\| u \|_{2}=1}\Im \la u, \widetilde{\Sol}_3(\theta e^{\imu \psi}) u \ra
  >
  0
  \,,  
\end{equation}
for sufficiently small $\theta >0$.
Indeed, for any vector $u=(u_1,u_2) \in \CC^2$ we find
\begin{multline}
  \label{eq:38}
  \Im \la u, \widetilde{\Sol}_3 u \ra
  \ge
  \cstE \Im (-t^{-1} |u_1|^2 - t^2 |u_2|^2 ) - C |u_1||u_2| |t|
  \\
  \ge
  \cstE_\varepsilon \bigg(\frac{1}{|t|} |u_1|^2 + |t|^2 |u_2|^2\bigg) - \frac{C}{R}\frac{1}{|t|}|u_1|^2 - C R|u_2| |t|^3
\end{multline}
for some constants $\cstE,C>0$, small $t$, an $\varepsilon$-dependent constant $\cstE_\varepsilon$ and any $R>0$.
For $R=2C/\cstE_\varepsilon$ and sufficiently small $t$ this is still positive. 

\emph{Step 2: Expansion around $z=1$.}
For the expansion around $z=1$ we proceed similarly to the discussion at $z=0$.
We set $t = \sqrt{z-1}$, where the square root has a branch cut at $\imu (-\infty,0]$ and  $\sqrt{1}=1$.
Here we make an ansatz with a reduced number of unknown functions by exploiting the identities \eqref{eq:31} and \eqref{eq:169}, namely
\begin{equation}
  \label{eq:108}
    \widetilde{\Sol}_3 =
    \left(
      \begin{array}{cc}
        \frac{\xi}{4 \cstT^2 t} & \frac{E}{2}- \frac{\imu \xi}{2 \cstT } (t^{-1}+t) + \frac{\nu t}{4 \cstT }
        \\
        \frac{E}{2} + \frac{\imu \xi}{2 \cstT } (t^{-1}+t) + \frac{\nu t}{4 \cstT } & \xi (t^{-1}+t)
      \end{array}
    \right)
    \,.
\end{equation}
We will determine the unknown functions $\xi$ and $\nu$.
Plugging \eqref{eq:108} into \eqref{eq:105}, multiplying out everything and simplifying afterwards reveals
\begin{equation}
  \label{eq:40}
  \Delta 
  =
  \left(
    \begin{array}{cc}
      -\frac{1}{64  \cstT^4 }(q_1+t p_1) & -\frac{1}{32  \cstT^3} (q_2 + t p_2)
      \\
      \frac{1}{32  \cstT^3} (q_2 + t p_2) & \frac{1+t^2}{16  \cstT^2}(q_1+t p_1 -2 \imu (q_2 + t p_2))
    \end{array}
  \right)
  \,,  
\end{equation}
where $p_1,p_2$ are polynomials in $t,E,\xi,\nu$ and where
\begin{equation}
  \label{eq:109}
  \begin{split}
    q_1
    =
    & \xi^4 +  (2  E^2  \cstT^2   +  4 + 4 \cstT^2) \xi^2 +   \cstT^2(  \cstT^2 E^4 + 4 E^2 (1 +  \cstT^2)- 16 )
    \\ & 
    - 16 \imu  \cstT^3 E + \imu E^2  \cstT^2 \xi \nu  +  4  \imu   \xi \nu  +  \imu \xi^3 \nu  \,,
    \\
    q_2 = & \xi^3 \nu +  (4 +   \cstT^2 E^{2}) \xi \nu - 16  \cstT^{3} E   \, .
  \end{split}
\end{equation}
The solution to the system $(q_1,q_2)=0$ at $t=0$ has the form
\begin{equation}
  \label{eq:110}
    \nu_0 = \frac{16  \cstT^3 E}{\xi (4+ \cstT^2 E^2+ \xi^2 )}\,, \quad  r:= \cstT^2E^4  +  \frac{1}{ \cstT^2} \xi^4  +  \frac{4}{  \cstT^2}  (1 +  \cstT^2) \xi^2  +  4 E^2 (1 +  \cstT^2 + \frac{1}{2} \xi^2 )-16 =0
\end{equation}
In $r=0$ from \eqref{eq:110} we choose the unique negative solution (as long as the expression inside the square root is positive)
\begin{equation}
  \label{eq:170}
  \xi_0 = -\sqrt{2 \sqrt{1 + 6  \cstT^2 +  \cstT^4}-2-2  \cstT^2 - \cstT^2 E^2} 
  .
\end{equation}
We compute the Jacobian
\begin{equation}
  \label{eq:111}
  \begin{split}
    \det \nabla_{\xi,\nu}(q_1,q_2)
    =  4 \xi^2 (8 + 2 \cstT^2 (12 + E^2)  +  \cstT^4( \frac{2}{ \cstT^4} \xi^2 - 2 E^2)  -  2  \cstT^2 \xi^2  + \cstT^2 r)\,, 
  \end{split}
\end{equation}
and evaluate at $t=0$ with $\xi=\xi_0$ to find
\begin{equation}
  \label{eq:171}
\det \nabla_{\xi_0,\nu_0}(q_1,q_2) = 16 \xi_0^2 (1 +  \cstT^4 + \sqrt{
   1 + 6  \cstT^2 +  \cstT^4} -  \cstT^2 (-6 + \sqrt{
      1 + 6  \cstT^2 +  \cstT^4}))<0\,.  
\end{equation}
In particular, $\xi$ and $\nu$ admit a power series expansion in $t$. 

\emph{Step 3: Coincidence of $\widetilde{\Sol}_3$ and $\Soll{3}$, asymptotic behavior of $\Sol_{z,E}(1,1)$.}
Here we show that the asymptotic expansion $\widetilde{\Sol}_3$ around $z=0,1$ coincides with $\Soll{3}$ defined in \eqref{eq:20} and \eqref{eq:27}, the solution to \emph{DEL} \eqref{eq:29} constructed from free probability.
For this purpose, for any small $\delta >0$, we construct a modification $\Sol_{z,E}^{\delta}$ of the explicit solution $\Sol_{z,E}$ (given by \eqref{eq:33}) as follows
\begin{equation}
  \label{eq:172}
  \Sol^{\delta}_{z,E}
  :=
  (\id_{8} \otimes \,\tau_{\Sc})(\Hb_E^{(\mathrm{sc}),\delta}- z J_{8} \otimes \scone)^{-1} \,, 
\end{equation}
with
\begin{equation}
  \label{eq:177}
  \qquad 
\Hb_E^{(\mathrm{sc}),\delta}
:=
  \left(
    \begin{array}{ccc}
      \cstU_1 \otimes \scone & 0 &\cstU_5^t \otimes c_2^*
      \\
      0& \cstU_2 \otimes \scone & \cstU_4^t \otimes c_1^*
      \\
      \cstU_5 \otimes c_2 & \cstU_4 \otimes c_1 &  \cstV_1 \otimes (E \cdot \scone -(1-\delta)\cdot \semic ) -\delta \cstV_1 \widetilde{\Sol}_3 \cstV_1\otimes \scone 
    \end{array}
  \right)
  \,,
\end{equation}
where $c_i$ are circular and $\semic$ semicircular elements.
Since the original generalised resolvent has the bound $\|\Sol_{z,E}\| \leq C (1+\frac{1}{\Im z})$, we also get 
\begin{equation}
  \label{eq:112}
  \|\Sol_{z,E}^{\delta}-\Sol_{z,E}\|
  \leq \cstF \delta\,.
\end{equation}
Now, similarly as in \eqref{eq:28}-\eqref{eq:29}, we derive the \emph{DEL} corresponding to the generalized resolvent of $\Hb_E^{(\mathrm{sc}),\delta}$ and find
\begin{align}
  \label{eq:199}
  -\frac{1}{\Sol_{1}^{\delta}}
  =
  z J_{3} - \cstU_{1} + (1-\delta) \cstU_{5}^{t} \Sol_{3}^{\delta} \cstU_{5}
  ,\quad
  -\frac{1}{\Sol_{2}^{\delta}}
  =
  - \cstU_{2} + (1-\delta) \cstU_{4}^{t} \Sol_{3}^{\delta} \cstU_{4}
  ,
  \\   \label{eq:173}
  -\frac{1}{\Sol_{3}^{\delta}} = -E \cstV_1 - \frac{1}{Z_1 +\Sol_{3}^{\delta}} - \frac{1}{Z_2 +\Sol_{3}^{\delta}} + (1-\delta) \cstV_1 \Sol_{3}^{\delta}\cstV_1 + \delta \cstV_1 \widetilde{\Sol}_{3}\cstV_1
  \,,
\end{align}
where we exploited the block structure of $\Sol_{z,E}^{\delta}$ and denoted
\begin{equation}
  \label{eq:200}
  \Sol_{z,E}^{\delta}
  =
  \left(
    \begin{array}{ccc}
      \Sol_{1}^{\delta}& 0 & 0
      \\
      0 & \Sol_{2}^{\delta} & 0
      \\
      0 & 0 & \Sol_{3}^{\delta}
    \end{array}
  \right)
  .
\end{equation}
Notice that $\Im \big(\delta \cstV_1 \widetilde{\Sol}_{3}\cstV_1\big)>0$ and $R \mapsto -E \cstV_1 - (Z_1 +R)^{-1} - (Z_2 +R)^{-1} + (1-\delta) \cstV_1 R \cstV_1$ is a positivity preserving analytic mapping.
The existence and uniqueness of the solution to \eqref{eq:173} now follows from \cite[Theorem~2.1]{HeltRaFaSpei07}.
Inverting both sides of \eqref{eq:173}, we obtain a fixed point equation $\Sol_{3}^{\delta}= \Phi_{\delta}(\Sol_{3}^{\delta})$ for $\Sol_{3}^{\delta}$, where the map $\Phi_{\delta}$  can be read off from the inverse of the right-hand side of \eqref{eq:173}. Clearly, $\Phi_{\delta}$ with any $\delta>0$ is a contraction with respect to the Carath\'eodory metric mapping the set of $2 \times 2$ matrices with (strictly) positive definite imaginary parts strictly into itself (for more details see \cite{HeltRaFaSpei07}).
In particular, there is a unique solution with positive semidefinite imaginary part and thus  $\Sol_{3}^{\delta} = \widetilde{\Sol}_{3}$ for $\delta>0$.
Together with  \eqref{eq:112} and taking the limit $\delta \downarrow 0$ we conclude $\Soll{3} = \widetilde{\Soll{3}}$. 
In particular, $\Soll{3}$ has a power law expansion around the singularities at $z=0,1$.

Finally, plugging the expansions \eqref{eq:106} and \eqref{eq:108} into \eqref{eq:79} yields the asymptotics \eqref{eq:175}-\eqref{eq:176}.
\end{proof}

\subsection{ Proof of (i) and (ii) of Theorem~\ref{thm:main} for $\cstS = 1$ }
\label{sec:proof-ii-iii}

We will need the following classical result from the theory of Herglotz functions (see, e.g., \cite[Theorem~2.3]{GeszTsek00})
\begin{lem} \label{lem:herglotz}
  Let $m:\CC_{+}\rightarrow \CC_{+}$ be a Herglotz function with representation
  \begin{equation}
    \label{eq:174}
    m(z)
    =
    m^{\infty} + \int_{\RR} \frac{1}{\lambda - z} v (d\lambda)
  \end{equation}
  for $m^{\infty}\in \RR$ and a Borel measure $v(d\lambda)$ on $\RR$.
  If for some $p>1$ and interval $I\subset \RR$
  \begin{equation}
    \label{eq:178}
    \sup_{0< \eta < 1}\int_{I}|\Im m(\lambda + \imu \eta)|^{p} d\lambda
    <
    \infty
    ,
  \end{equation}
 then the measure $v(d\lambda)$ is absolutely continuous on $I$.
\end{lem}

\begin{proof}[ Proof of (i) and (ii) of Theorem~\ref{thm:main} for $\cstS = 1$ ]
     The weak limits in the part \emph{(i)} of Theorem~\ref{thm:main} have been established in Lemma~\ref{lem:global-law}.
  
  We know from \eqref{eq:22} that $\Sol_{z,E}(1,1)$ is a Herglotz function admitting the representation \eqref{eq:174}  with $m^{\infty} = \Sol_{E}^{\infty}(1,1) $ and $v(d\lambda) = \la e_1, V_{E}(d\lambda) \, e_1\ra = \rho_{E}(d\lambda)$.
Therefore, the absolute continuity of $\rho_{E}(d\lambda) = \rho_{E}(\lambda) d\lambda$ is a direct consequence of Lemma~\ref{lem:herglotz} and Lemmas \ref{lem:boundedness} and \ref{lem:sing-sol}  establishing together the integrability of $\lambda \mapsto \Im \Sol_{\lambda + \imu \eta ,E}(1,1)$ in the form \eqref{eq:178} on the whole real axis for any $p<3/2$.

Finally, the asymptotic behavior of $\rho_{E}(\lambda)$ \eqref{eq:5}-\eqref{eq:179} near its singularities at $0$ and $1$ follows from Lemma~\ref{lem:sing-sol} and the inverse Stieltjes transform formula \eqref{eq:24}.
This, together with the global law established in Lemma~\ref{lem:global-law}, finishes the proof of Theorem~\ref{thm:main}  for $\cstS = 1$ .  
\end{proof}

\section{Proof of Theorem~\ref{thm:main} for  general rational  $\cstS \in (0,1)$ }
\label{sec:proof-theor-refthm:m}

In this section we explain how the techniques described in Section~\ref{sec:proofs} can be used to prove  Theorems \ref{thr:Global law} and  \ref{thm:main} in the case when $\cstS = k/l\in (0,1)$  for fixed $k,l\in \NN$, i.e. when $M=ln$ and $N=kn$ for some integer $n$ tending to infinity.
In addition to the steps used in Section~\ref{sec:proofs}, we 
need to tensorize the setup to accommodate for rectangular matrices. For example,
the $M\times N = ln\times kn$ matrices $W_1$ and $W_2$  will be viewed  as  $l\times k$ rectangular block matrices with blocks of dimension $n\times n$.
As we will see later, this allows us to treat the linearization matrix as a Kronecker random matrix in independent Wigner and i.i.d. matrices, which in turn makes various probabilistic estimates of the error terms in the corresponding \emph{DEL} readily accessible from \cite{AltErdoKrugNemi_Kronecker}. The restriction $\cstS = k/l$ makes the presentation conceptually easier.

Note that different values of $\cstS = k/l \in (0,1)$ require slightly different approaches.
For $\cstS \in (1/2,1)$ the matrix $\Tb_{E,\cstS} \in \CC^{kn \times kn}$ given by \eqref{eq:3} is well defined and bounded with very high probability.
Indeed, in this case the product  $W W^*$ is a sample covariance matrix with concentration ratio $\frac{1}{2\cstS} \in (0,1)$,  therefore similarly as in Section~\ref{sec:linearization-trick}, one can define a sequence of events $\Theta_{\cstS,N}$ holding \emph{a.w.o.p.} (see \eqref{eq:12} and \eqref{eq:219} below) such that  for any small  $\delta>0$ and big enough $n\geq n_0( \delta )$, the spectrum of $W W^{*}$ is contained inside the interval $[(1-\frac{1}{\sqrt{2 \cstS}})^{2}(1-  \delta) ,(1+\frac{1}{\sqrt{2 \cstS}})^{2}(1+  \delta) ]$ when restricted to $\Theta_{\cstS, N}$ (see, e.g., \cite[Section~5]{BaiSilvBook}).
Thus, as $n$ tends to infinity, the  matrices in the denominator of \eqref{eq:3} have positive imaginary parts and therefore bounded inverses with very high probability. 

On the other hand, for $\cstS \in (0,1/2]$ we need to proceed via regularization by replacing $\imu  \cstT W W^*$ with $\imu  \cstT W W^{*} + \imu \cstX \cstS  I_{M}$ for some small $\cstX > 0$ to ensure the invertibility of the  denominators in \eqref{eq:3}.
This requires a more careful analysis of the $\cstX$-dependence of various bounds and identities before we can take $\cstX\rightarrow 0$.

Note also that for $\cstS > 1$, $\mathrm{Rank}(\Tb_{E,\cstS})\leq M$ and the spectral measure of the matrix $\Tb_{E,\cstS}$ has an atom of mass $1-1/2\cstS$ at zero, while, as we will show, for $\cstS <1$ the spectral measure $\mu_{\Tb_{E,\cstS}}$ does not have the pure point component.
The regime $\cstS=1$ studied in Section~\ref{sec:proofs} is borderline: the limiting spectral measure of $\Tb_{E,\cstS}$ does not have atom at $0$, but its behavior near the origin is more singular than in the case $\cstS<1$.

\subsection{Linearization trick and the Dyson equation for linearization}
\label{sec:line-trick-dyson-1}

In order to apply the linearization trick for $\cstS = k/l\in (0,1)$, we split $H$, $W_{1}$ and $W_{2}$ into blocks of size $n\times n$, so that
\begin{equation}
\label{eq:51}
  H
  =
  (\widehat{H}_{ij})_{\substack{i=1 \ldots l \\ j=1 \ldots l}}
  ,\quad
  W_{1}
  =
  (\widehat{W}_{1,ij})_{\substack{i=1 \ldots l \\ j=1 \ldots k}}
  ,\quad
  W_{2}
  =
  (\widehat{W}_{2,ij})_{\substack{i=1 \ldots l \\ j=1 \ldots k}}
  ,
\end{equation}
with $\widehat{H}_{ij}, \widehat{W}_{1,ij}, \widehat{W}_{2,ij} \in \CC^{n\times n}$.
Note that for any $i\in \{1,\ldots,l\}$, $\sqrt{l}\cdot\widehat{H}_{ii}$ is a (normalized) Wigner matrix of size $n$, and for any $i\neq j$, $\sqrt{l} \cdot \widehat{H}_{ij}$ is a (normalized) \emph{i.i.d.} matrix of size $n$.
Similarly, all the matrices $\sqrt{l} \cdot \widehat{W}_{1,ij}$ and $\sqrt{l} \cdot \widehat{W}_{2,ij}$ are also \emph{i.i.d.} matrices of size $n$.
All these matrices are independent apart from the natural constraint $\widehat{H}_{ij} = \widehat{H}_{ji}^{*}$.

 Define the events
\begin{equation}
  \label{eq:219}
  \Theta_{\cstS, N}
  :=
  \left\{
    \begin{array}{ll}
      \{\, \|H\| \leq 3,\, \| W_{1}\| \leq 3, \, \|W_2\| \leq 3, \,\|(W W^*)^{-1}\| \leq  \frac{1}{(1-\frac{1}{\sqrt{2 \cstS}})^{2}(1-\delta)}  \, \},  & \quad \cstS \in (1/2,1),
      \\
      \{\, \|H\| \leq 3,\, \| W_{1}\| \leq 3, \, \|W_2\| \leq 3 \, \},  & \quad \cstS \in (0,1/2],
    \end{array}
    \right.
  \end{equation}
for some $\delta >0$.
Similarly as in Section~\ref{sec:linearization-trick}, one can show that the events $\Theta_{\cstS,N}$ hold \emph{a.w.o.p.} and the random matrix models \eqref{eq:73} and \eqref{eq:3} for $\cstS \in (1/2,1)$ and $\cstY \in \CC_{+}\cup \RR$ restricted to $\Theta_{\cstS,N}$ are well defined (see Remark~\ref{rem:matrix-valued-ext} and Lemma~\ref{lem:norm-bounds-random}). 
At the same time, in order to deal with $\cstS\in (0,1/2]$, we consider the regularized models \eqref{eq:73} and \eqref{eq:3} with $\cstY \in \CC_{+} $, i.e., $\Im \cstY >0$ strictly positive, which guarantees the invertibility of $\cstY-H + \imu \cstT W W^{*}$ without any additional restrictions on $H$ and $W$.

The linearization matrix $\Hb_{\cstY,\cstS}$ for \eqref{eq:3} is defined as in \eqref{eq:6}.
This is a \textit{Kronecker } random matrix consisting of $(6k+2l)\times (6k+2l)$ blocks of size $n$.
We now introduce a tensorized version of the generalized resolvent that takes into account the additional structure coming from \eqref{eq:51}.
\begin{defn}[Generalized resolvent]
  Let $(\cstS, \cstY) \in ((0,1/2]\times \CC_{+})\cup ((1/2,1)\times (\CC_{+}\cup \RR))$ with $\cstS = k/l \in \QQ$.
  We call the matrix-valued function $\CC_{+}\ni z \mapsto (\Hb_{\cstY,\cstS}-z\Jb_{k}\otimes I_{n})^{-1}$ the \emph{generalized resolvent} of $\Hb_{\cstY,\cstS}$.
  Here we denote $\Jb_{k}:= \sum_{i=1}^{k} E_{ii}  \in \CC^{(6k+2l)\times (6k+2l)}$ with $\{E_{ij}\}$ being the standard basis of $\CC^{(6k+2l)\times (6k+2l)}$.
\end{defn}

\begin{lem}[Basic properties of the generalized resolvent] \leavevmode
  \label{lem:trivial-gen-res2}
\begin{itemize}
  \item[(i)] For any $\cstT>0$ and $\cstS \in (1/2,1)\cap \QQ$ there exists $C_{\cstT,\cstS}>0$ such that \emph{a.w.o.p.}
    \begin{equation}
      \label{eq:230}
       \big\| (\Hb_{\cstY,\cstS} - z\Jb_{k}\otimes I_{n})^{-1}\big\|
      \leq
      C_{\cstT,\cstS}\bigg(1 + \frac{1}{\Im z}\bigg)
    \end{equation}
    uniformly for all $z\in\CC_{+}$ and $\cstY \in \CC_{+}\cup \RR$.
  \item[(ii)] For any $\cstT>0$, $\cstS \in (0,1/2]\cap \QQ$ and $\cstY \in \CC_{+}$ there exists $C_{\cstT,\cstS,\cstY}>0$ such that \emph{a.w.o.p.}
    \begin{equation}
      \label{eq:246}
      \big\| (\Hb_{\cstY,\cstS} - z\Jb_{k}\otimes I_{n})^{-1}\big\|
      \leq
      C_{\cstT,\cstS,\cstY}\bigg(1 + \frac{1}{\Im z}\bigg)
    \end{equation}
    for all $z\in \CC_{+}$.
  \item[(iii)] For all $(\cstS, \cstY) \in \big(((0,1/2]\cap \QQ)\times \CC_{+}\big)\cup \big(((1/2,1)\cap \QQ)\times (\CC_{+}\cup \RR)\big)$ and $1\leq i,j\leq N$
    \begin{equation}
      \label{eq:231}
      \Big[(\Hb_{\cstY,\cstS}-z\Jb_{k}\otimes I_{n})^{-1}\Big]_{ij}
      =
      \Big[(\Tb_{\cstY,\cstS}- zI_N)^{-1}\Big]_{ij}
      ,\quad
      1\leq i,j \leq N
      .
    \end{equation}
  \end{itemize} 
\end{lem}
\begin{proof}
  Denote by $\Dc_{\qb_{0},\{q_{1,\cstY}\};C}$ the effective domain related to the noncommutative rational function  $q_{1,\cstY}(x,y_1,y_2,y_1^*,y_2^*)= \cstY-x + \imu \cstT (y_1y_1^* + y_2 y_2^*)$.
  More precisely (see \eqref{eq:116} below), the set $\Dc_{\qb_{0},\{q_{1,\cstY}\};C}$ consists of all triples of elements $(\hat{x},\hat{y}_1,\hat{y}_2)$ such that
  \begin{equation}
    \label{eq:247}
    \left\| \frac{1}{q_{1,\cstY}(\hat{x},\hat{y}_1,\hat{y}_2,\hat{y}_1^*,\hat{y}_2^*)}\right\|
    \leq
    C
    .
  \end{equation}
  
  As explained at the beginning of this section, for $\cstS = k/l \in (1/2,1)$ \emph{a.w.o.p.} the random matrices $H$, $W_1$ and $W_2$ defined in \eqref{eq:51} belong to the domain  $\Dc_{\qb_{0},\{q_{1,\cstY}\};C}$ with constant (see \eqref{eq:219} with $\delta = 1/2$) $C=\frac{1}{2 \cstT (1-\frac{1}{\sqrt{2\cstS}})}$
  depending only on $\cstT$ and $\cstS$.
  
  For $\cstS = k/l \in (0,1/2]$ on the other hand, the evaluation of the function $q_{1,\cstY}$ on random matrices $H$, $W_1$ and $W_2$ is invertible only for $\Im \cstY >0$, in which case the norm of the inverse of $q_{1,\cstY}$ evaluated on any triple $(\hat{x},\hat{y}_1,\hat{y}_2)$ is bounded by $\frac{1}{\cstT \Im \cstY}$.
  
 With the above choice of the constant $C$ in $\Dc_{\qb_{0},\{q_{1,\cstY}\};C}$ depending on the value of $\cstS$, the proof of (i)-(iii) can be obtained from the same argument as in Lemma~\ref{lem:trivial-gen-res} by restricting $\Hb_{\cstY,\cstS}$ to the events $\Theta_{\cstS, N}$ defined in \eqref{eq:219} and taking into account the dimensions of $H$, $W_{1}$ and $W_{2}$ and the relation $N=kn$ (see Remark~\ref{rem:matrix-valued-ext}).
\end{proof}
From the structure of the linearization, we derive the DEL corresponding to  $\Hb_{\cstY,\cstS}$
\begin{equation}
  \label{eq:52}
  -\frac{1}{\Sol}
  =
  z \Jb_{k} - K_0(\cstY) + \SuOp_{\cstS}[\Sol]
\end{equation}
for an unknown matrix-valued function $M$ depending on $z$, $\cstY$ and $\cstS$, having the following components:
\begin{itemize}
\item[(i)] the expectation matrix is given by
\begin{equation}
  \label{eq:54}
  K_0(\cstY)
  :=
  \left(
    \begin{array}{ccc|ccc|ccc}
      &\cstU_{1}\otimes I_{k}& &&& &&&
      \\
      &&&&&&&&
      \\ \hline
      &&&&&&&&
      \\
      &&          & &\cstU_{2} \otimes I_{k}& &&&
      \\
      &&&&&&&&
      \\ \hline
      &&&&&&&&
      \\
      &&&&& & &\cstU_{3}(\cstY) \otimes I_{l}&
    \end{array}
  \right)
\end{equation}
with matrices $\cstU_{1}, \cstU_{2},\cstU_{3}(\cstY)$ defined as in \eqref{eq:13}-\eqref{eq:14}; 
\item[(ii)] the operator $\SuOp_{\cstS} : \CC^{(6k+2l)\times(6k+2l)} \rightarrow \CC^{(6k+2l)\times(6k+2l)}$ maps an arbitrary matrix
  \begin{equation}
    \label{eq:66}
    R
    =
    \left(
      \begin{array}{ccc}
        R_{11}& R_{12}& R_{13}                        
        \\
        R_{21}& R_{22}& R_{23}
        \\
        R_{31}& R_{32}& R_{33}                        
      \end{array}
    \right)
    \in \CC^{(6k+2l)\times (6k+2l)}
  \end{equation}
  with $R_{11}, R_{22} \in \CC^{3k\times 3k}$, $R_{33} \in \CC^{2l\times 2l}$ into a block diagonal matrix with the first $3k\times 3k$ block equal to
  \begin{equation}
    \label{eq:68}
    \cstU_{5}^{t}\bigg(\Big(\id_{2}\otimes \frac{1}{l}\Tr_{l}\Big) R_{33}\bigg)\cstU_{5}\otimes I_{k}
    ,
  \end{equation}
  the second $3k\times 3k$ diagonal block equal to
  \begin{equation}
    \label{eq:69}
    \cstU_{4}^{t}\bigg(\Big(\id_{2}\otimes \frac{1}{l}\Tr_{l}\Big) R_{33}\bigg)\cstU_{4}\otimes I_{k}
    ,
  \end{equation}
  and the lower-right $2l\times 2l$ block equal to 
  \begin{equation}
    \label{eq:70}
    \cstU_{5}\bigg(\Big(\id_{3}\otimes \frac{1}{l}\Tr_{k}\Big) R_{11}\bigg)\cstU_{5}^{t} \otimes I_{l} + \cstU_{4}\bigg(\Big(\id_{3}\otimes \frac{1}{l}\Tr_{k}\Big) R_{22}\bigg)\cstU_{4}^{t} \otimes I_{l} +  \cstV_{1} \bigg( \Big(\id_{2} \otimes \frac{1}{l}\Tr_{l}\Big) R_{33}\bigg) \cstV_{1} \otimes I_{l}
    .
  \end{equation}
  Here, for $n\in \NN$, we denote by $\Tr_{n}$ the trace of an $n\times n$ matrix  $\cstU_{4},\cstU_{5}$ are defined as in \eqref{eq:14} and $\cstV_{1}$ is a standard Pauli matrix. 
  The operator $\SuOp_{\cstS}$ is the tensorized analogue of $\SuOp$ from \eqref{eq:26}.
\end{itemize}

Now we can proceed similarly as for $\cstS=1$ in Section~\ref{sec:proofs} just the $(3+3+2)\times (3+3+2)$ structure of the linearized matrices is replaced by larger block matrices structured as $(3k+3k+2l)\times (3k+3k+2l)$.
\begin{lem}[Existence and basic properties of the solution to the DEL \eqref{eq:52}] \label{lem:exist-uniq-sol-2} \leavevmode
  For any $\cstT>0$, $(\cstS, \cstY) \in (((0,1/2]\cap \QQ)\times \CC_{+})\cup (((1/2,1)\cap \QQ)\times (\CC_{+}\cup \RR))$ and $z\in \CC_{+}$ define $\Sol_{z, \cstY } \in \CC^{\,(6k+2l)\times (6k+2l)}$ as
\begin{align}
  \Sol_{z,\cstY}
  :=
  (\id_{6k+2l} \otimes \,\tau_{\Sc}) \bigg[ \Big((K_0(\cstY) - & z \Jb_{k}) \otimes \scone + K_{1}\otimes H^{\mathrm{(sc)}}  + L_1\otimes W_{1}^{\mathrm{(sc)}} \nonumber
  \\ \label{eq:168}
  & +L_1^{*}\otimes \big(W_{1}^{\mathrm{(sc)}}\big)^* + L_2 \otimes W_{2}^{\mathrm{(sc)}} + L_2^{*} \otimes \big(W_{2}^{\mathrm{(sc)}}\big)^{*}\Big)^{-1}\bigg]
  ,
\end{align}
where $W_{1}^{(\mathrm{sc})}$ and $W_{2}^{(\mathrm{sc})}$ are $l\times k$ matrices consisting of freely independent circular  elements multiplied by $1/\sqrt{l}$, and $H^{(\mathrm{sc})}$ is an $l\times l$ self-adjoint matrix with freely independent semicircular elements multiplied by $1/\sqrt{l}$ on the diagonal and freely independent circulars multiplied by $1/\sqrt{l}$ above the diagonal.

Then
  \begin{itemize}
  \item[(i)]  For any $\cstT>0$ and  $\cstS \in (1/2,1)\cap \QQ$ there exists $C_{\cstT,\cstS}>0$ such that $\Sol_{z,\cstY}$ satisfies the a priori 
   bound 
\begin{equation}
  \label{eq:86}
  \| \Sol_{z,\cstY} \|
  \leq
  C_{\cstT,\cstS} \Big(1 + \frac{1}{\Im z}\Big)
\end{equation}
uniformly for all $\cstY \in \CC_{+}\cup \RR$ and $z\in \CC_{+}$.
\item[(ii)] For any $\cstT>0$, $\cstS \in (0,1/2]\cap \QQ$ and $\cstY \in \CC_{+}$ there exists $C_{\cstT,\cstS,\cstY}>0$ such that function $\Sol_{z,\cstY}$ satisfies the a priori bound
  \begin{equation}
    \label{eq:244}
      \| \Sol_{z,\cstY} \|
  \leq
  C_{\cstT,\cstS,\cstY} \Big(1 + \frac{1}{\Im z}\Big)
  .
  \end{equation}
  \item[(iii)] For any $\cstT>0$, $(\cstS, \cstY) \in (((0,1/2]\cap \QQ)\times \CC_{+})\cup (((1/2,1)\cap \QQ)\times (\CC_{+}\cup \RR))$ and $z\in \CC_{+}$, matrix $\Sol_{z,\cstY}$ satisfies the DEL \eqref{eq:52} and has positive semidefinite imaginary part, $\Im  \Sol_{z,\cstY} \geq 0$. Moreover, for all $\cstT>0$, $(\cstS, \cstY) \in (((0,1/2]\cap \QQ)\times \CC_{+})\cup (((1/2,1)\cap \QQ)\times (\CC_{+}\cup \RR))$, the matrix-valued function $z\mapsto \Sol_{z,\cstY}$ is analytic on $\CC_{+}$.
  \item[(iv)] For any $\cstT>0$ and $(\cstS, \cstY) \in (((0,1/2]\cap \QQ)\times \CC_{+})\cup (((1/2,1)\cap \QQ)\times (\CC_{+}\cup \RR))$ function $z\mapsto \Sol_{z,\cstY}$  admits the representation
\begin{equation}
  \label{eq:102}
  \Sol_{z,\cstY}
  =
  \Sol^{\infty}_{\cstY,\cstS} + \int_{\RR} \frac{V_{\cstY,\cstS}(d\lambda)}{\lambda - z}
  ,
\end{equation}
where $\Sol_{\cstY,\cstS}^{\infty} \in \CC^{(6k + 2l)\times (6k+2l)}$ is a self-adjoint matrix,  and $V_{\cstY,\cstS}(d\lambda)$ is a positive-semidefinite matrix-valued measure on $\RR$ with compact support.
  \end{itemize}
\end{lem}
\begin{proof}
  The proof follows from parts (i)-(v) of Lemma~\ref{lem:exist-uniq} (see also Remark~\ref{rem:matrix-valued-ext}).
  Similarly as in the proof of Lemma~\ref{lem:trivial-gen-res2}, notice that the noncommutative rational expression $q_{1,\cstY}= \cstY-x + \imu \cstT (y_1y_1^* + y_2 y_2^*)$ evaluated on  matrices $ x= H^{\mathrm{(sc)}}$, $y_1 = W_1^{\mathrm{(sc)}}$ and $y_2 = W_2^{\mathrm{(sc)}}$ expresses different behavior in variable $\cstY$ for $\cstS \in (0,1/2) $ and $\cstS \in [1/2,1)$.
  
  In the first case, $\cstS \in (1/2,1)$, the invertibility of $q_{1,\cstY}$ evaluated on $(s,c_1,c_2)$ does not depend on $\cstY$ and thus $(s,c_1,c_2) \in \Dc_{\qb_{0},\{q_{1,\cstY}\};C}$ with $C=C(\cstT,\cstS)$.
  In the case $\cstS \in (0,1/2]$, $q_{1,\cstY}$ is invertible if and only if $\Im \cstY >0$, and the norm of $(q_{1,\cstY})^{-1}$ depends on $\Im \cstY$.
  This, in particular, means that $(s,c_1,c_2) \in \Dc_{\qb_{0},\{q_{1,\cstY}\};C}$ with a $\cstY$-dependent constant $C(\cstT,\cstS,\cstY)$.
  
  This leads to two different a priori estimates: a bound~\eqref{eq:86}  uniform in $\cstY$  for $\cstS \in (1/2,1)$ and a $\cstY$-dependent bound for $\cstS \in (0,1/2]$.
  The rest of the proof follows directly from Lemma~\ref{lem:exist-uniq}.
\end{proof} 
We omit the dependence of $\Sol_{z,\cstY}$ on $\cstS$ for brevity. 
With these notations we have the following global law establishing Theorem~\ref{thr:Global law} and partially \emph{(i)} of Theorem~\ref{thm:main} for $\cstS \in (0,1)$.
The proof of the weak limit \eqref{eq:233} for $\cstS \in (0,1/2)$ is postponed to Section~\ref{sec:proof-ii-iii-1}.
\begin{lem}[Global law for $\Tb_{\cstY,\cstS}$, $\cstS \in (0,1)$]
  \label{lem:global-law-gen}
  For $(\cstS, \cstY) \in ((0,1/2]\times \CC_{+})\cup ((1/2,1)\times (\CC_{+}\cup \RR))$, $\cstS = k/l$, the empirical spectral measure $\mu_{\Tb_{\cstY,\cstS}}(d\lambda)$ converges weakly in probability (and almost surely) to $\rho_{\cstY, \cstS}(d\lambda)$, where
  \begin{equation}
    \label{eq:232}
    \rho_{\cstY,\cstS}(d\lambda)
    :=
    \frac{1}{k}\Tr ( \Jb_{k} \, V_{\cstY,\cstS}(d\lambda)  )
  \end{equation}
  is the normalized trace of the upper-left $k\times k$ submatrix of the matrix-valued measure $V_{\cstY,\cstS}(d\lambda)$ from \eqref{eq:102}.
  The support of the measure $\rho_{\cstY, \cstS}(d\lambda)$ is a subset of the interval $[0,1]$.
  Moreover, for any $\cstS\in (1/2,1)$ and $E\in \RR$, the measure $\rho_{\cstY, \cstS}(d\lambda)$ converges weakly to $\rho_{ E , \cstS}(d\lambda)$ as $\cstY \in \CC_+$ tends to $E\in \RR$.
\end{lem}
\begin{proof}
 The proofs are similar to the case $\cstS=1$  (see Lemma~\ref{lem:global-law} and Remark~\ref{rem:matrix-valued-ext}) after taking into account the dimensions of the matrices $H$, $W_1$ and $W_2$ and the additional structure \eqref{eq:51}.
\end{proof}
\begin{defn}[Self-consistent density of states]
  We call the function
  \begin{equation}
  \label{eq:103}
    \rho_{\cstY,\cstS}(\lambda)
  :=
  \lim_{\eta\downarrow 0} \frac{1}{\pi k}  \Im \Tr (\Jb_{k} \Sol_{\lambda+\imu \eta,\cstY})
\end{equation}
that gives the absolutely continuous part of $\rho_{\cstY,\cstS}(d\lambda)$, the \emph{self-consistent density of states} of the model \eqref{eq:3}.
\end{defn}
Since $\mathrm{supp} (\rho_{\cstY,\cstS})\subset [0,1]$ by unitarity of $S(\cstY)$, part \emph{(iii)} of Theorem~\ref{thm:main} can be established by proving the boundedness of the upper-left $k\times k$ minor of $\Sol_{z,\cstY}$ for the spectral parameter $z$ bounded away from $0$ and $1$ (Section~\ref{sec:boundedness-sol_z-e1} below), and analyzing the asymptotic behavior of this upper-left submatrix in the vicinity of the special points $z=0$ and $z=1$ (Section~\ref{sec:singularities} below).
    The study of $\Sol_{z,\cstY}$ is simplified by the particular form of $K_{0}(\cstY)$ and $\SuOp_{\cstS}$, which implies that
    \begin{equation}
      \label{eq:72}
      \Sol_{z,\cstY}
      =
      \left(
        \begin{array}{ccc}
          \Sol_1\otimes I_{k}&&
          \\
                          & \Sol_{2} \otimes I_{k} &
          \\
          && \Sol_{3} \otimes I_{l}
        \end{array}
        \right)
      \end{equation}
      with $\Sol_1,\Sol_2\in \CC^{\,3\times 3}$ and $\Sol_{3}\in \CC^{\,2\times 2}$ satisfying
      \begin{equation}
        \label{eq:74}
        -\frac{1}{\Sol_1}
        =
        zJ_{3}-\cstU_{1} + \cstU_{5}^{t} \Sol_{3} \cstU_{5}
        ,\quad
        -\frac{1}{\Sol_{2}}
        =
        -\cstU_{2} + \cstU_{4}^{t} \Sol_{3} \cstU_{4}        
      \end{equation}
      and
      \begin{equation}
        \label{eq:76}
        -\frac{1}{\Sol_{3}}
        =
        -  \cstU_{3}(\cstY)  
        + \cstS \cstU_{5} \Sol_{1} \cstU_{5}^{t} + \cstS \cstU_{4} \Sol_{2} \cstU_{4}^{t} +  \cstV_{1} \Sol_{3}  \cstV_{1} 
        .
      \end{equation}
      Similarly as in the case $\cstS=1$, plugging \eqref{eq:74} into \eqref{eq:76} leads to the following self-consistent equation for $\Sol_{3}$
      \begin{equation}
        \label{eq:77}
        -\frac{1}{\Sol_{3}}
        =
        - \cstU_{3}(\cstY) 
        -  \frac{ \cstS }{-\frac{1}{2\cstT^2 z}(I_{2}+ \sigma_{3})-\frac{1}{\cstT}\sigma_{2}+\Sol_{3}} -  \frac{\cstS }{-2(I_{2}- \sigma_{3})-\frac{1}{\cstT}\sigma_{2}+\Sol_{3}} + \sigma_{1} \Sol_{3} \sigma_{1}
        ,
      \end{equation}
      which is the analogue of \eqref{eq:29} for $\cstS \neq 1$.

\subsection{Useful identities}
\label{sec:useful-identities}
Below we prove that identities similar to \eqref{eq:41}, \eqref{eq:31} and \eqref{eq:169} hold for $\cstS \in (0,1)$.
\begin{lem}
  \label{lem:useful-identities-2}
  For all  $(\cstS, \cstY) \in (((0,1/2] \cap \QQ)\times \CC_{+})\cup (((1/2,1)\cap \QQ)\times (\CC_{+}\cup \RR))$, $\cstT>0$  and $z\in \CC_{+}$
  \begin{itemize}
  \item[(i)]   $\Sol_{z,\cstY}(i,i) = \Sol_{z,-\overline{\cstY}}(i,i)$ for all $1\leq i \leq 6k+2l$;
  \item[(ii)]  $  \Sol_{z,\cstY}(6k+l+1,6k+l+1) = 4  \cstT^2  z \Sol_{z,\cstY}(6k+1,6k+1)$;
  \item[(iii)]  $   \Sol_{z,\cstY}(6k+l+1,6k+1) - \Sol_{z,\cstY}(6k+1,6k+l+1)  =  \frac{\imu}{ \cstT }\Sol_{z,\cstY}(6k+l+1,6k+l+1)\Big(1-\frac{ \cstT \Im \cstY  }{2 \cstS} \det T_1\Big)$,
    where, similarly as in \eqref{eq:56}, we denoted
    \begin{equation}
      \label{eq:149}
      T_1
      =
        \left(
    \begin{array}{cc}
      -\frac{1}{ \cstT^2 z} & \frac{\imu}{ \cstT }
      \\
      -\frac{\imu}{ \cstT } & 0
    \end{array}
  \right)
  +
  \Soll{3} 
    \end{equation}
  \end{itemize}
\end{lem}
 Denote the entries of $\Sol_{3}$ in \eqref{eq:72} by $m_{ij}$, $1\leq i,j\leq 2$.
Then the parts $(ii)$ and $(iii)$ of the above lemma can be rewritten as
\begin{align}
  \label{eq:44}
  m_{22} &= 4 \cstT^2 z m_{11}
  ,\\
  \label{eq:45}
  m_{21} - m_{12}
 & =
  \frac{\imu}{ \cstT }m_{22} \Big(1-\frac{ \cstT \Im \cstY }{2 \cstS } \det T_1\Big)
  .
\end{align}

\begin{proof}
In order to establish Lemma~\ref{lem:useful-identities-2}, we can follow the proofs of Lemmas~\ref{lem:useful-identities-3}-\ref{lem:useful-identities-5} and  apply them to the matrix
\begin{equation}
  \label{eq:48}
  (K_0(\cstY) - z \Jb_{k}) \otimes \scone + K_{1}\otimes H^{\mathrm{(sc)}}  + L_1\otimes W_{1}^{\mathrm{(sc)}} +L_1^{*}\otimes \big(W_{1}^{\mathrm{(sc)}}\big)^* + L_2 \otimes W_{2}^{\mathrm{(sc)}} + L_2^{*} \otimes \big(W_{2}^{\mathrm{(sc)}}\big)^{*}
  .
\end{equation}
Note that the above matrix \eqref{eq:48} is obtained from \eqref{eq:32} by substituting $c_{1}$, $c_{2}$ and $\semic$ with matrices $W_{1}^{(\mathrm{sc})}$, $W_{2}^{(\mathrm{sc})}$ and $H^{(\mathrm{sc})}$ correspondingly, and taking into account the dimensions of these matrices.
For example, if we replace each diagonal entry of the matrix $Q^{-}$ from the proof of Lemma~\ref{lem:useful-identities-3} by the tensor product of this entry and a corresponding identity matrix ($I_{k}$ or $I_{l}$), we obtain that the diagonal blocks of $\Sol_{z,\cstY}$ and $\Sol_{z, -\overline{\cstY}}$ coincide.

In order to prove \eqref{eq:44}, similarly as in the proof of Lemma~\ref{lem:useful-identities-4}, use the Schur complement formula with respect to the invertible upper-left $6k\times 6k$ submatrix of \eqref{eq:48} to write the $2l\times 2l$ lower-right submatrix of its inverse as
\begin{equation}
  \label{eq:55}
  \left(
    \begin{array}{cc}
      4  \cstT^2 W_{1}^{(\mathrm{sc})} (W_{1}^{(\mathrm{sc})})^* & \cstY-H^{(\mathrm{sc})}+ \imu  \cstT \Big(W_{1}^{(\mathrm{sc})} (W_{1}^{(\mathrm{sc})})^* + W_{2}^{(\mathrm{sc})} (W_{2}^{(\mathrm{sc})})^*\Big)
                              \\
      \overline{\cstY}-H^{(\mathrm{sc})}- \imu  \cstT \Big(W_{1}^{(\mathrm{sc})} (W_{1}^{(\mathrm{sc})})^* + W_{2}^{(\mathrm{sc})} (W_{2}^{(\mathrm{sc})})^*\Big) & \frac{1}{z} W_{2}^{(\mathrm{sc})} (W_{2}^{(\mathrm{sc})})^*
    \end{array}
  \right)^{-1}
  .
\end{equation}
Then we can switch the blocks of \eqref{eq:55} as in \eqref{eq:46} and apply the Schur complement formula with respect to
\begin{equation}
  \label{eq:60}
  \cstY - H^{(\mathrm{sc})} + \imu  \cstT \Big(W_{1}^{(\mathrm{sc})} \big(W_{1}^{(\mathrm{sc})}\big)^* + W_{2}^{(\mathrm{sc})} \big(W_{2}^{(\mathrm{sc})}\big)^*\Big)
  .
\end{equation}
 The expression in \eqref{eq:60} is invertible: for $(\cstS, \cstY) \in (((1/2,1)\cap \QQ)\times (\CC_{+}\cup \RR))$ the spectrum of  $W_{1}^{(\mathrm{sc})} (W_{1}^{(\mathrm{sc})})^* + W_{2}^{(\mathrm{sc})} (W_{2}^{(\mathrm{sc})})^*$ follows the free Poisson distribution with rate  $2\cstS \in (1,2)$ and is therefore bounded away from zero, and for  $(\cstS, \cstY) \in (((0,1/2]\cap \QQ)\times \CC_{+})$ the invertibility is guaranteed by $\Im \cstY$ being strictly positive.
Due to the properties of freely independent circular and semicircular elements, switching the labels of the pair $(W_{1}^{(\mathrm{sc})},W_{2}^{(\mathrm{sc})})$ or changing the sign of $H^{(\mathrm{sc})}$ does not change the value of an expression involving these matrices after applying $\id_{6k+2l}\otimes \,\tau_{\Sc}$.
Therefore, by proceeding as in \eqref{eq:49}-\eqref{eq:53} with $\Sol_{z,\cstY}(7,7)$ and $\Sol_{z, \cstY}(8,8)$ replaced by the corresponding $l\times l$ blocks of $\Sol_{z, \cstY}$, and using the diagonal structure of these blocks \eqref{eq:72} and the part \textit{(i)} of this lemma, we obtain \eqref{eq:44}.

Now it is straightforward to check by plugging \eqref{eq:44} into \eqref{eq:77} and following the proof of Lemma~\ref{lem:useful-identities-5}, that \eqref{eq:45} holds.
This proves Lemma~\ref{lem:useful-identities-2}.
  
\end{proof}

\subsection{Boundedness of $\Sol_{z,\cstY}$ away from $z=0$ and $z=1$  for $\cstS\in (0,1)$}

 The goal of this section is to establish a uniform bound on $\|\Sol_{z,\cstY}\|$ for parameter $\cstY$ close to the real line and parameter $z$ bounded away from $0$ and $1$.
This will be used later in the proof of Theorem~\ref{thm:main}, in particular to show the absolute continuity of the measure $\rho_{E,\cstS}(d \lambda)$.
In the case $\cstS \in (1/2,1)$ we can set $\cstY = E \in \RR$ and  work directly with $\Sol_{z,E}$.
For $\cstS \in (0,1/2]$ we prove the uniform bound for $\|\Sol_{z,\cstY}\|$ with $\cstY = E + \imu \cstS \, \cstX$ and small $\cstX>0$, which will allow taking the limit $\cstX \to 0$ in Section~\ref{sec:singularities}.

\label{sec:boundedness-sol_z-e1}
\begin{lem}[Boundedness of $\Sol_{z,\cstY}$]\label{lem:boundedness-2} \leavevmode
  \begin{itemize}
  \item[(i)] Case  $\cstS \in (1/2,1]$:  For any $\cstT >0$ and small $\epsA>0$ there exists $\cstCC>0$ such that 
  \begin{equation}
    \label{eq:161}
    \sup \Big\{\,\|\Sol_{z,E}\|  \, : \,  \cstS \in (1/2,1)\cap\QQ,\, | z |\geq \epsA,  \, |1-  z | \geq \epsA, \, \Im z > 0 ,\, |E|\leq \frac{1}{\epsA}\, \Big\} 
    \leq
    \cstCC
    ;
  \end{equation}
  \item[(ii)] Case $\cstS \in (0, 1)$:  Let $\cstY = E + \imu \cstS \cstX$, $\cstX >0$.
For any $\cstT>0$, small $\epsA>0$, $\cstS_0\in (0,1/2)$ and $\cstX_0>0$ small enough there exists $\cstCCC>0$ such that 
  \begin{align}
    \label{eq:162}
    \|\Sol_{z,\cstY}\| 
    \leq
    \cstCCC
  \end{align}
  uniformly on the set
  \begin{equation}
    \label{eq:234}
    \Big\{ \cstS \in [\cstS_0,1-\cstS_0]\cap\QQ, |z|\geq \epsA,   |1-z| \geq \epsA,  \Im z > 0, |E|\leq \frac{1}{\epsA}\,, \cstX \in (0, \cstX_0]\,   \Big\}
    .
  \end{equation}
  \end{itemize}
\end{lem}
\begin{proof}
  Consider first \eqref{eq:161} for which $\cstS \in (1/2,1)$.
  By setting $\cstX = 0$ in Lemma~\ref{lem:useful-identities-2} we can proceed by establishing \eqref{eq:161} in the same manner as in the proof of Lemma~\ref{lem:boundedness}.
  To guarantee a uniform bound in parameters $z$, $E$ and $\cstS$, instead of a sequence $(z_n, E_{n})_{n=1}^{\infty}$ as in the proof of Lemma~\ref{lem:boundedness},  we assume the existence of a sequence $((z_n, E_n, \cstS_n))_{n=1}^{\infty}$, $z_n \in \CC_+$, $|E|\leq 1/\epsA$, $\cstS_n \in (1/2,1)$, on which $|m_{11}^{(n)}|\rightarrow \infty$ or $|\det T_{1}^{(n)}|\rightarrow 0$.
  Note that for $\cstS_n\in (1/2,1)$ the leading terms in the analogues of the mutually contradicting pairs of statements \eqref{eq:88}/\eqref{eq:91}, \eqref{eq:92}/\eqref{eq:95} and \eqref{eq:100}/\eqref{eq:101} do not depend on $\cstS_n$ (with the exception of \eqref{eq:92} where the constant $-1$ is replaced by $\cstS_n - 2$).

  For \eqref{eq:162}, i.e. $\cstS \in [\cstS_0,1- \cstS_0]$, the above argument has to be slightly adjusted to ensure a uniform bound for small $\cstX>0$.
  Instead of a sequence $((z_n, E_n, \cstS_n))_{n=1}^{\infty}$ as in the first part of the proof, we now assume the existence of a sequence $((z_n, E_n ,\cstS_n, \cstX_n))_{n=1}^{\infty}$, $z_n \in \CC_+$, $|E|\leq 1/\epsA$, $\cstS_n \in [\cstS_0,1-\cstS_0]$, $\cstX_n>0$, on which $|m_{11}^{(n)}|\rightarrow \infty$ or $|\det T_{1}^{(n)}|\rightarrow 0$.

  The analogues of \eqref{eq:88}/\eqref{eq:91} in this case are
  \begin{align}
    \label{eq:163}
      \det \Soll{3}^{(n)}
  &=
  -4(z_{n}-1)m_{11}^{(n)} + \OO{1}
    ,\\
       \det \Soll{3}^{(n)}
    &=
      4  \cstT^2  z_n (1-z_n(1-\cstX_n\cstS_n  \cstT )^2)(m_{11}^{(n)})^2 + \OO{m_{11}^{(n)}}
      ,
  \end{align}
  which contradict each other if $|m_{11}^{(n)}|\rightarrow \infty$, $\cstX_n\leq \cstX_0$ is small enough and $z_n$ is bounded away from $0$ and $1$.

  Instead of the pair \eqref{eq:92}/\eqref{eq:95}, we take the analogues of \eqref{eq:92} and \eqref{eq:93} for $\cstS_n \in [\phi_0, 1-\phi_0]$ and $\cstX_n>0$
  \begin{align}
    \label{eq:164}
      \det \Soll{3}^{(n)}
  &=
  \cstS_n -2   +\OO{\frac{1}{|m_{11}^{(n)}|}} 
  ,\\
  \det T_{1}^{(n)}
  &=
    4(z_{n}-1)m_{11}^{(n)} + \OO{1}
    ,
  \end{align}
  and observe that since $1 + 1/(\cstS_n-2)\geq \cstS_0/(1+\cstS_0)>0$ for $\cstS_n \in [\cstS_0,1-\cstS_0]$, the above equations contradict to
  \begin{equation}
    \label{eq:165}
       \Big(\frac{1}{\det \Soll{3}} -\frac{2 \cstS_n}{\det T_{1}} + 1 \Big) m_{11}^{(n)} = -\frac{\cstS_n}{ \cstT^2  z\det T_1},
  \end{equation}
  the analogue of \eqref{eq:62}, in the regime $|m_{11}^{(n)}|\rightarrow \infty$.

  For the last pair \eqref{eq:100}/\eqref{eq:101} note that in the analogue of \eqref{eq:96}
  \begin{equation}
    \label{eq:166}
      \Big(\frac{\cstS_n}{  \cstT^2  z_{n} m_{11}^{(n)}   }+\OO{\det T_{1}^{(n)}}\Big)m_{12}^{(n)}
  =
  -\frac{2 \cstS_n \imu}{  \cstT } + \OO{\det{T_{1}}^{(n)}}
\end{equation}
with $|\det{T_{1}}^{(n)}|\rightarrow 0$, the parameter $\cstS_n$ disappears after dividing \eqref{eq:166} by $\cstS_n$, therefore we can proceed exactly as in the proof of Lemma~\ref{lem:boundedness}.
We conclude, similarly as in Lemma~\ref{lem:boundedness}, that $\|\Sol_{z,\cstY}\|$ is uniformly bounded provided that $\min\{|z|, |1-z|\} \geq \epsA$, $|E|\leq 1/\epsA$, $\cstS_0 \leq \cstS\leq 1-\cstS_0$ and $\cstX \leq \cstX_0$ for some $\cstS_0 >0$ and  $\epsA,\cstX_0>0$ small enough.
\end{proof}

\subsection{Singularities of $\Sol_{z,E}(1,1)$}
\label{sec:singularities}

\newcommand{\ii}{\mathrm{i}} 
\newcommand{\1} {\mspace{1 mu}}
\newcommand{\2} {\mspace{2 mu}}
\newcommand{\msp}[1] {\mspace{#1 mu}}
 For $\cstS\in (1/2,1)$, 
 the solution matrix $\Sol_{z,E}$ with $E\in \RR$ is given directly  via \eqref{eq:168}.
 For $\cstS \in (0,1/2]$ this formula cannot be directly applied when $\cstY  = E \in \RR $ is real; 
 we need  an additional regularization argument.   Nevertheless, 
  in the next lemma we   show the existence of the solution to the Dyson equation \eqref{eq:74}-\eqref{eq:77} for $\cstS \in (0,1/2]$ and $\cstY  = E \in \RR $ and we 
establish the asymptotic behavior of $\Sol_{z, E}(1,1)$ near $z=0$ and $z=1$ for $\cstS\in (0,1)$.
We start by constructing an expansion of the solution $\Sol_{z, \cstY}$ in the vicinity of $z=0$ for $\cstS\in (0,1)$ and $\cstY = E + \imu \cstS \cstX$ with $\cstX>0$ sufficiently small.
We then use this expansion to extend $\Sol_{z,\cstY}$ to $\cstY = E \in \RR$ for $\cstS \in (0,1/2]$ by taking $\cstX \downarrow 0$ and to study the asymptotic behavior of $\Sol_{z,E}$ at special points $z=0$ and $z=1$. 

Recall that  the solution to the Dyson equation  has the block structure \eqref{eq:72}, where $\Sol_{1},\Sol_{2}$ are determined by $\Sol_{3}$ through \eqref{eq:74}  and $\Sol_{3}$ satisfies \eqref{eq:77}.

\begin{lem}[Existence of $\Sol_{z,E}$ for $\cstS\in (0,1/2{]}$ and singularities of $\Sol_{z, E}(1,1)$ for $\cstS \in (0,1)$] \label{lem:sinularities-2}\leavevmode
\begin{itemize}
\item[(a)]  For $\cstT >0$, let $\Sol_{z,\cstY, \cstS}$ denote the function $\Sol_{z,\cstY}$  defined in \eqref{eq:168} evaluated at a point $(z,\cstY,\cstS)$ with
  \begin{equation}
    \label{eq:167}
    z\in \CC_{+}
    \quad \mbox{and} \quad
   (\cstY,\cstS)\in \big(\CC_{+}\times \big((0,1/2]\cap \QQ \big)\big)\cup \big((\CC_{+}\cup \RR) \times \big((1/2,1)\cap \QQ\big)\big). 
  \end{equation}
  Then $\Sol_{z,\cstY,\cstS}$ can be continuously extended to the set
  \begin{equation}
    \label{eq:243}
    z\in \CC_{+}
    \quad \mbox{and} \quad
    (\cstY,\cstS)\in \RR \times (0,1) 
  \end{equation}
  in the following sense: for any $(\cstY,\cstS) \in \RR \times (0,1)$ and any sequence $\{(\cstY_{n},\cstS_{n}) \}_{n\geq 1} \subset \big(\CC_{+}\times \big((0,1/2]\cap \QQ \big)\big)\cup \big((\CC_{+}\cup \RR) \times \big((1/2,1)\cap \QQ\big)\big)$ with $(\cstY_{n},\cstS_{n})\rightarrow (\cstY,\cstS)$ as $n\rightarrow \infty$, there exists an analytic 
   matrix-valued function $\Sol_{z,\cstY,\cstS}:\CC_+ \rightarrow \CC^{(6k + 2l)\times(6k + 2l)}$, $ z\mapsto \Sol_{z,\cstY,\cstS}$, such that
  \begin{equation}
    \label{eq:80}
    \Sol_{z,\cstY_{n},\cstS_{n}}
    \rightarrow
    \Sol_{z,\cstY,\cstS}
  \end{equation}
uniformly on compact $z$-subsets of $\CC_+$ as $n\rightarrow \infty$.
\end{itemize}
For $\cstS \in (0,1)$ and $\cstY = E \in \RR$, denote by $\Sol_{z,E}:= \Sol_{z,E,\cstS}$ the function defined in \eqref{eq:80} omitting explicitly the dependence on $\cstS$. 
\begin{itemize}
  \item[(b)] For all $\cstS \in (0,1)$, $\cstT >0$ and $E\in \RR$
  \begin{equation}
\label{eq:75}
    \Sol_{z,E}(1,1)
    =
    \imu \frac{4+  \cstT^2 \nu_0^2+  \cstT^2 \2\xi_0^2}{4  \cstT \2\xi_0\2}\,z^{-1/2} + O(1)\,
    \mbox{ as }
      z\rightarrow 0
    ,
  \end{equation}
  with constants $\xi_{0} := \xi_{0} ( \cstS, \cstT) $ and $\nu_{0} := \nu_{0} ( \cstS, \cstT) $ given in \eqref{eq:42} in the proof below.
\item[(c)] For all $\cstS \in (0,1)$, $\cstT>0$ and  $|E| \leq E_{0}:=E_{0}(\cstS,\cstT)$
    \begin{equation}
\label{eq:140}
    \Sol_{z,E}(1,1)
    =
    \frac{4\2\xi_0}{(\xi_0^2 +  \cstT^2 E^2+4  )}(z-1)^{-1/2}  + \OO{1}
    \mbox{ as }
    z\rightarrow 1
    ,
  \end{equation}
  with constants $E_{0}$ and $ \xi_{0}:= \xi_{0} ( \cstS, \cstT) <0$ given in \eqref{eq:146} and \eqref{eq:141} correspondingly.
  \end{itemize}
  The branch of the square root is chosen to be continuous on $\CC\setminus (\ii (-\infty,0])$ such that $\sqrt{1}=1$.
\end{lem}
\begin{proof}
  The analysis of \eqref{eq:77} for $\cstS\in (0,1)$ will follow similar steps as the analysis of \eqref{eq:29} for the $\cstS=1$ case as performed  in Section~\ref{sec:proof-edge-asympt} and we will omit  the details of some  straightforward albeit tedious calculations.
   In the first step below we analyze the solution to \eqref{eq:77} for rational $\cstS\in (0,1)$ and $\cstY = E + \imu \cstS\, \cstX$ with $E\in \RR$ and $\cstX>0$ small enough. 
Using the same procedure that led to \eqref{eq:105} we rewrite this equation as $\Delta =0$ with
\begin{equation}
  \label{Delta equation}
    \Delta
    :=
    \Big(Z_1+(1-2\cstS)\Soll{3} + \Soll{3} \cstV_1(\Soll{3}\cstV_1-E+\ii\1\cstX \cstS \sigma_3 )(Z_1+\Soll{3})\Big)\frac{1}{Z_1-Z_2}(Z_2+\Soll{3}) -\cstS\Soll{3}
    \,,
\end{equation}
and the two $2 \times 2$-matrices
\begin{equation}
  Z_1
  :=
  -\frac{1}{2  \cstT^2 z}(I_2+\cstV_3) -\frac{1}{ \cstT }\cstV_2
  \,, \qquad
  Z_2
  :=
  -2(I_2-\cstV_3) -\frac{1}{  \cstT }\cstV_2
  \,.  
\end{equation} 

\emph{Step 1: Expansion around $z=0$.} We construct an expansion of $M_3$ as a power series in $t=\sqrt{z}$ in a neighborhood of $z=0$.
We  make an ansatz compatible with the symmetries 
\eqref{eq:44} and \eqref{eq:45}, namely
\begin{equation}
\label{ansatz with tau at z=0}
  \widetilde{\Sol}_3(t)
  =
  \left(
    \begin{array}{cc}
        \frac{\ii \2\xi}{4  \cstT t} &  \frac{\nu}{2} + \frac{\xi\2t }{2}(1+\frac{\cstX }{  \cstT } f )+\cstX\2\frac{\ii\2 \xi^2}{4}
        \\
          \frac{\nu}{2} - \frac{\xi\2t }{2} (1+\frac{\cstX }{ \cstT } f)-\cstX\2\frac{\ii\2 \xi^2}{4}  &  \ii  \cstT \2\xi \2t
      \end{array}
    \right)
    \,,
\end{equation}
for the two functions $\xi=\xi(t),\nu=\nu(t)$ of $t$ to be determined and where the parameters $E,\cstT, \cstS,\cstX$ are considered fixed. Later we will show that with the right choice of functions $\xi, \nu$ this ansatz coincides with the solution of the Dyson equation, i.e. that $ \widetilde{\Sol}_3(t)= {\Sol}_3(z)$. Here $f$ solves the  equation $q_0=0$ for sufficiently small $\cstX\ge 0$,  where 
\begin{equation} \label{eq:def of f eps}
q_0 = f+  \frac{   \cstT^2 }{2}\det \Bigg[
  \left(
    \begin{array}{cc}
    -\frac{1}{ \cstT^2 t^2} & \frac{\ii}{ \cstT }
        \\
    -\frac{\ii}{ \cstT } &0
    \end{array}
  \right) + \widetilde{\Sol}_3\Bigg]+\frac{\ii\2  \cstT \2 \xi}{2\2t}
\end{equation}
 is a polynomial in all variables $t,E, \cstT,\xi,\nu,\cstX$ and $ f$, in which it is quadratic. The choice \eqref{eq:def of f eps} for $f$ ensures that the symmetry condition  \eqref{eq:45} is satisfied.

We plug \eqref{ansatz with tau at z=0} into \eqref{Delta equation} and, after a mechanical but very long calculation that uses $q_0=0$ in the $(2,1)$- and $(2,2)$- entries of $\Delta$, find
\begin{equation}
  \Delta 
  =
  \left(
    \begin{array}{cc}
   -\frac{\ii}{16 \cstT t}q_{1,\cstX,t}  & \frac{1}{4}q_{2,\cstX,t}
      \\
      -\frac{1}{4}q_{2,\cstX,t} & \frac{ \ii  \cstT  \2t}{4}q_{1,\cstX,t}  + (\ii  \cstT t^2+\ii   t \cstX  f\2t-   \frac{  \cstT \1\cstX\2\xi\2t}{2})q_{2,\cstX,t}
    \end{array}
  \right)
  \,,  
\end{equation}
where 
\[
q_{1,\cstX,t} =q_1+\cstX\2 \widetilde{q}_1 + \frac{t}{4 \cstT^3 } p_{1}\,, \qquad q_{2,\cstX,t} =q_2 +  \frac{t}{ \cstT } p_{2}\,,
\]
and $p_1,p_2$ are polynomials in $t,E, \cstT,\cstS,\xi,\nu, \cstX$ and $q_1=q_{1,0,0},q_2=q_{2,0,0},\widetilde{q}_1$ are the following explicitly defined functions of the unknowns $(\xi,\nu)$:
\begin{equation}
\label{q_i for tau ne 0 at z=0}
q_1 = \xi\2  (\xi^2 - \nu^2+2E\2\nu+4(\cstS-1))\,,
\quad
q_2=    \xi^2 (\nu-E)+2\nu \cstS\,,
\quad
\widetilde{q}_1=-\imu\xi q_2+\cstX\, \Big(\frac{\xi^{5}}{4}+\xi^3\cstS \Big)\, .
\end{equation}
In particular, the equation $\Delta=0$ is equivalent to $(q_{1,\cstX,t},q_{2,\cstX,t})=0$ and thus to  $(q_1,q_2) = 0$ in the limit $t \to 0$ and $\cstX \to 0$. This, in turn, fixes the 
values for $\nu_0=\nu|_{t=0}$ and  $\xi_0=\xi|_{t=0}$ through
\begin{equation}
  \label{eq:42}
\nu_0 = \frac{E\2 \xi_0^2 }{ \xi_0^2+2\cstS}\,, \quad r := \xi_0^6 +  (E^2+8\cstS-4) \2\xi_0^4+ 4 \cstS (E^2+ 5\cstS-4) \2\xi_0^2+16(\cstS-1)\cstS^2=0\,,  
\end{equation}
where we choose the positive solution $\xi_0$ for $r=0$. The fact that $r=0$ has a unique positive solution  $\xi_0>0$ is an explicit elementary calculation.  The positivity of $\xi_0$ will ensure the positive definiteness of $\Im \widetilde{M}_3$ for  $z \in \CC_+$.
  
We compute the Jacobian  of the function $(q_1, q_2)$ from \eqref{q_i for tau ne 0 at z=0} as
\begin{equation}
  \label{eq:147}
J(\xi,\nu):=\det \nabla_{\xi,\nu}(q_1,q_2) = 3\2\xi^4+ ((3 \nu^2-6 \nu E+4 E^2)+10\cstS-4)\2\xi^2+2\cstS(4\cstS-4+2E\nu-\nu^2)\,.  
\end{equation}
Using that at $t=0$ we have $q_1=0$ and $ \xi_0\ne 0$,  we can eliminate the quadratic terms in $\nu$ and obtain
\[
J(\xi_0,\nu_0) = 2\2\xi_0^2\2 (3 \2\xi_0^2 +2E^2   + 10\cstS-8)\,.
\]
Again an elementary calculation using the defining equation $r=0$
for $\xi_0$ shows that $J(\xi_0, \nu_0)$ never vanishes.
Thus, the ansatz \eqref{ansatz with tau at z=0} solves the Dyson equation in a small neighbourhood of $z=t^2=0$. Furthermore, for $t =  u (1+\ii  u )$ with sufficiently small $ u >0$ it is easy to see that its imaginary part is positive definite. By using a regularization argument analogous to the one from Step~3 of Section~\ref{sec:proof-edge-asympt} and combining it with the uniqueness of solutions to the  Dyson equation with positive definite imaginary part, this implies that $\widetilde{M}_3(t) = M_3(z)$ for all $\sqrt{z}=t =  u (1+\ii  u )$, $\cstS \in (0,1)$ and small enough $\cstX >0$. 
Since both functions are analytic this also implies equality for $z=t^2$ in the complex upper half plane intersected with a neighbourhood of $z=0$.

\emph{Step 2: Extending $\Sol_{z,\cstY}$ to $\cstY = E \in \RR$ for $\cstS \in  (0,1)  $.}
Fix $\cstS\in  (0,1)  $, $\cstT >0$ and $E\in \RR$.
For $\cstX_0>0$ and a coordinate pair $(i,j)\in \{1,\ldots, 6k + 2l\}^2$, consider the family of functions $\{\Sol_{z,E+\imu \cstX_n  \cstS_{n} }\,(i,j)\}_{n\geq 1}$ analytic in $z$ with  $\{\cstS_{n},\cstX_n\}_{n\geq 1}\subset ((0,1)\cap \QQ) \times (0,\cstX_0)$ and $(\cstS_{n},\cstX_{n})\rightarrow (\cstS,0)$  as $n\rightarrow \infty$.
  It follows from part \textit{(ii)} of Lemma~\ref{lem:boundedness-2} that for
  small enough $\cstX_0$ and any small $\epsA>0$ the family of functions $\{\Sol_{z,E+\imu \cstX_n  \cstS_{n}}\,(i,j)\}_{n\geq 1}$ is uniformly bounded on the set $\{z \in \CC_+ \, : \, |z| \geq \epsA, |1-z| \geq \epsA\}$, and thus  locally bounded on $\CC_+$.

  It was established in Step 1 above that for any  $\cstT>0$, $E\in\RR$, $\cstS \in (0,1)$ and $\cstX>0$ small enough  the solution to the equation \eqref{eq:77} with positive semidefinite imaginary part can be explicitly given as $\widetilde{\Sol}_3(\sqrt{z})$ (see \eqref{ansatz with tau at z=0}) in the neighborhood of the origin, i.e., on the set $\{|z|\leq \delta, \Im z >0\}$ for $\delta > 0$ sufficiently small.
  Moreover, Step 1 shows that for any $z\in \{|z|\leq \delta, \Im z >0\}$ there exists a well defined limit   $\lim_{(\cstS_{n},\cstX_{n})\rightarrow (\cstS, 0)} \widetilde{\Sol}_3(\sqrt{z})$. 
  Together with \eqref{eq:72} and \eqref{eq:74} this implies that on the set $\{|z|\leq \delta, \Im z >0\}$ the limit $\Sol_{z,E}:= \lim_{(\cstS_{n},\cstX_{n})\rightarrow (\cstS, 0)} \Sol_{z,E+\imu \cstX_n \cstS_{n}}  $ exists as well.

  Combining the above information, we see that for any index pair $(i,j)$,  $\{\Sol_{z,E+\imu \cstX_n \cstS_{n}}\,(i,j)\}_{n\geq 1}$ is a family of analytic functions, locally bounded on $\CC_+$ that converges on $\{|z|\leq \delta, \Im z >0\}$ to $\Sol_{z,E}(i,j)$.
  Applying the Vitali-Porter theorem  (see, e.g., Section 2.4 in \cite{SchiBook}), we conclude that for any $z\in \CC_+$ the limit $\Sol_{z,E}\,(i,j):=\lim_{(\cstS_{n},\cstX_{n})\rightarrow (\cstS, 0) }\Sol_{z,E+\imu \cstX_n \cstS_{ n }}\,(i,j)$ exists, the convergence holds uniformly on the compact subsets of $\CC_+$ and, as a result, the function $z\mapsto \Sol_{z,E}\,(i,j)$ is analytic on $\CC_+$.
  Taking $\Sol_{z,E}\,(i,j)$ as the entries of the matrix-valued function $\Sol_{z,E}$ defines the solution to the Dyson equation \eqref{eq:76} for  $\cstS \in (0,1)$ and $\cstY = E \in \RR$. 

To compute the asymptotic behavior of $\Sol_{z,E}(1,1)$ we use \eqref{ansatz with tau at z=0} with $\cstX=0$ (see \eqref{q_i for tau ne 0 at z=0}), \eqref{eq:74} and $t = \sqrt{z}$ to find $\Sol_1$ and its upper left corner element $\Sol_{z,E}(1,1)$ in the neighborhood of $z=0$
\[
  \Sol_{z,E}(1,1) = \ii \frac{4+  \cstT^2 \nu_0^2+  \cstT^2 \2\xi_0^2}{4 \cstT \2\xi_0} \, z^{-1/2} + O(1)\,.
\]

\emph{Step 3: Expansion around $z=1$.}
We apply exactly the same procedure as in Step~2 of Lemma~\ref{lem:sing-sol},
just we insert the parameters $\cstS$ and $\cstT$ into the identities \eqref{eq:104} and \eqref{eq:105} and follow them through the analysis. Here we record the final result
of this elementary calculation. Our ansatz is
 \begin{equation}
 \label{ansatz with tau at z=1}
   \widetilde{\Sol}_3
   =
   \left(
     \begin{array}{cc}
         \frac{\xi}{4  \cstT^2t}&  \frac{E}{2}+\frac{\nu t}{4 \cstT } - \frac{\ii \2\xi}{2 \cstT }(t +t^{-1})
         \\
          \frac{E}{2}+\frac{\nu t}{4 \cstT } +  \frac{\ii \2\xi}{2 \cstT }(t +t^{-1})   &  \xi\2(t +t^{-1})
       \end{array}
     \right)
     \,.
   \end{equation}
The expansion in $t=\sqrt{z-1}$ gives that for all $|E| \le E_0$ with
\begin{equation}
  \label{eq:146}
E_0 := \frac{\sqrt{2}}{ \cstT }\Big(  \cstT^2 (1-2\cstS) -1+\sqrt{
 1 +   \cstT^4 (1 - 2 \cstS)^2 + 
   2 \cstT^2 (1+2\cstS)}\Big)^{1/2}  
\end{equation}
the upper-left component of $\Sol_{z,E}$ is given by
\begin{equation}
  \label{eq:151}
  M(1,1)
  =
  \frac{4\2\xi_0}{(\xi_0^2 + \cstT^2 E^2+4  )t} + O(1)\,,
\end{equation}
where $\xi_{0}$ is defined by
\begin{equation}
  \label{eq:141}
  \xi_0 = - \cstT \sqrt{E_0^2-E^2}
  .
\end{equation}
This finishes the proof of the lemma.
\end{proof}

\subsection{Explicit solution for $\cstS\to 0$}
\label{sec:expl-solut-phito0}

\begin{lem} \label{lem:phi-0-solution}
  Let $\cstT>0$, $\cstY = E \in \RR$ and $\cstS = 0$.
  Then the Dyson equation \eqref{eq:76}-\eqref{eq:77} admits a solution $\CC_{+}\ni z\mapsto \Sol_{z,E} \in \CC^{8\times 8}$ with $\Im \Sol_{z,E}\geq 0$ and the upper-left entry is explicitly given by
  \begin{equation}
    \label{eq:249}
\Sol_{z,E}(1,1) =
\frac{ \cstT^2  (4-E^2)-(1+ \cstT^2 )^2+\ii  \cstT  (1+ \cstT^2 )}{ \cstT^2  (4-E^2)+((1+ \cstT^2  )^2-(4-E^2) \cstT^2  )z} \sqrt{\frac{4-E^2}{z(1-z)}}\,. 
\end{equation}
Moreover, this solution can be continuously extended to the set $\cstS\in[0,\cstS_{0}]$, $|1-z|\geq \theta$, $\theta \leq |z| \leq \theta^{-1}$ for $E\in \RR$ and sufficiently small $\theta>0$ and $\cstS_0 = \cstS_0(\theta) >0$.
\end{lem}
\begin{proof}
At $\cstX=0$ and setting $\cstS = 0$, the  equation \eqref{eq:77} simplifies to a quadratic matrix equation for $M_3 \sigma_1$, where $M_3$ satisfies the symmetry constraints \eqref{eq:44} and \eqref{eq:45}, i.e. 
\begin{equation}\label{eq:solution for phi=0}
        -\frac{1}{\Sol_{3}}
        =
        -E\,\sigma_{1}   + \sigma_{1} \Sol_{3} \sigma_{1}
        , \qquad 
        M_3 =
        \left(
    \begin{array}{cc}
    \frac{\ii \xi}{4  \cstT } & \frac{\nu }{2}+\frac{ z \xi}{2} 
    \\
    \frac{\nu }{2}-\frac{ z \xi}{2}  & z  \cstT  \ii \xi
     \end{array}  
            \right)\,,
      \end{equation}
      for  two functions $\xi$ and $\nu$ that are easily computed to be $\nu =E$ and
      \begin{equation}
        \label{eq:148}
        \xi = \sqrt{\frac{4-E^2}{z(1-z)}}\,.       
      \end{equation}
The choice of root is determined by $\Re \xi >0$. Inserting $\Sol_{3}$ into \eqref{eq:74} leads to 
\begin{equation}
  \label{eq:229}
\Sol_{z,E}(1,1) =
\frac{ \cstT^2  (4-E^2)-(1+ \cstT^2 )^2+\ii  \cstT  (1+ \cstT^2 ) \xi}{ \cstT^2  (4-E^2)+((1+ \cstT^2  )^2-(4-E^2) \cstT^2  )z}\,.
\end{equation}  

We now construct the solution of \eqref{eq:77} perturbatively for $\cstS\in (0,\cstS_0)$ with some sufficiently small $\cstS_0=\cstS_0(\theta)>0$ with  $|1-z|\ge \theta$ and $\theta\le |z| \le \theta^{-1}$.
Recall that \eqref{eq:77} is equivalent to $\Delta=0$ with $\Delta$ defined as in  \eqref{Delta equation}. In particular,
\begin{equation}\label{eq:Delta at phi=0}
\Delta=\widetilde{\Delta} (Z_1+\Soll{3})\frac{1}{Z_1-Z_2}(Z_2+\Soll{3}) \,,
\end{equation}
implicitly defining $\widetilde{\Delta}$ that satisfies
\begin{equation}\label{eq:wt Delta at phi=0}
\widetilde{\Delta}|_{\phi=0}=1+ \Soll{3}( \cstV_1\Soll{3}\cstV_1- \cstV_1E )\,.
\end{equation}
Similarly as we did in the proof of Lemma~\ref{lem:sinularities-2} instead of considering the equation $\Delta=0$ for a solution $\Soll{3} \in \mathbb{C}^{2\times 2}$ we can equivalently consider it as an equation for the two unknown functions $\xi$ and $\nu$ from the ansatz \eqref{eq:solution for phi=0} for $\Soll{3}$. Since clearly the factors $Z_1+\Soll{3}$ and $Z_2+\Soll{3}$ in \eqref{eq:Delta at phi=0} have bounded inverses when $\Soll{3}$ is the explicit solution 
 at $\phi=0$ from \eqref{eq:solution for phi=0} with \eqref{eq:148}, we can equivalently consider $\widetilde{\Delta}=0$ as the equation for $\xi$ and $\nu$. Plugging the ansatz \eqref{eq:solution for phi=0} into  \eqref{eq:wt Delta at phi=0} we see that $\widetilde{\Delta}_{11}=0$ and $\widetilde{\Delta}_{12}=0$ already imply
  $\widetilde{\Delta}=0$ and that 
\[
\widetilde{\Delta}_{11}|_{\phi=0} = \frac{1}{4} (4 + \nu^2 + 2 z \nu \xi - z \xi^2 + z^2 \xi^2 - 
   2 E (\nu + z \xi))\,, \qquad \widetilde{\Delta}_{12}|_{\phi=0} =-\frac{\ii (E - \nu) \xi}{4 \gamma}\,.
\]
We  compute the determinant of the Jacobian of the function $(\xi,\nu)\to( \widetilde{\Delta}_{11},\widetilde{\Delta}_{12})$ to be 
\[
\det \nabla_{\xi,\nu}(\widetilde{\Delta}_{11},\widetilde{\Delta}_{12})|_{\phi=0,\nu=E} = \frac{\ii\1(z-1)\1z\1 \xi^2}{8\gamma}\,.
\]
Since $\xi$ from \eqref{eq:148} does not vanish we infer that \eqref{eq:77} is linearly stable as an equation for $\xi$, $\nu$ for small enough parameters $\phi$ in a vicinity of the explicit solution $\Soll{3}$ from \eqref{eq:solution for phi=0} with \eqref{eq:148} for $\phi=0$.
\end{proof}

From \eqref{eq:249} we read off the density of transmission eigenvalues 
$ \rho(\lambda)= \rho_{E, \cstS = 0}(\lambda) :=\frac{1}{\pi}\lim_{\eta \downarrow 0}\Im \Sol_{\lambda+ \ii \eta ,E}(1,1)$.
The corresponding Fano factor for $\cstS = 0$ is now computable as
\begin{equation}
  \label{eq:236}
  F(E,\cstT)
  = 1-\frac{\int \lambda^2 \rho(\lambda) \mathrm{d} \lambda}{\int \lambda \rho(\lambda) \mathrm{d} \lambda} = \frac{1+ \cstT^2 }{2(1+ \cstT^2 + \cstT  \sqrt{4-E^2})}\,.  
\end{equation}
For $\gamma=1$ and $E=0$  we recover the density~\eqref{bimodal} and the Fano factor $F=\frac{1}{4}$ 
obtained in~\cite{Been97}, see Section~\ref{sec:ees=0} for more details.

\subsection{ Proof of parts (i), (iii) and (iv) of  Theorem~\ref{thm:main} }
\label{sec:proof-ii-iii-1}

In this section we collect the results established in Sections~\ref{sec:line-trick-dyson-1}-~\ref{sec:expl-solut-phito0} and complete the proof of Theorem~\ref{thm:main}.
Recall that Theorem~\ref{thr:Global law} was proven in Lemma~\ref{lem:global-law} for $\cstS = 1$ and Lemma~\ref{lem:global-law-gen} for $\cstS \in (0,1)$.

\begin{proof}[Proof of part \emph{(i)} of Theorem~\ref{thm:main}.]

  The extension of $\rho_{\cstY,\cstS}(d\lambda)$ to $\cstY = E \in \RR$ for $\cstS\in (0,1/2]\cap \QQ$ as well as the limit \eqref{eq:233} and the extension of $\rho_{E,\cstS}(d\lambda)$ to irrational $\cstS \in (0,1)$ follows from the equivalence between the weak convergence of measures defined by \eqref{eq:102} and the pointwise convergence of $\Sol_{z,w}$ established in Lemma~\ref{lem:global-law-gen} and part (a) of Lemma~\ref{lem:sinularities-2} for the corresponding limits.

  The weak limit $\lim_{\cstS \downarrow 0}\rho_{E,\cstS}(d\lambda)  = \rho_{E,\cstS=0}(\lambda)d\lambda $  
  follows from the 
   continuity of $ \cstS\mapsto \Sol_{z,E}$ at $\cstS = 0$ for all $z\in\CC_{+}$, which was established in Lemma~\ref{lem:sinularities-2} 
   in the regimes  $|z|\leq \theta$,  $|1-z|\leq \theta$ and in Lemma~\ref{lem:phi-0-solution}  in the 
   complementary regimes $|z|\geq \theta$, $\theta\leq |1-z|\leq \theta^{-1}$.
  Since the Stieltjes transform of $\rho_{E,\cstS}$ is given by $\Sol_{z,E}(1,1)$, the exact expression for $\rho_{E,0}$ can be derived as an inverse Stieltjes transfom of $\Sol_{z,E}$ from \eqref{eq:249}.   

  Similarly as for the case $\cstS=1$ in Section~\ref{sec:proof-ii-iii}, Lemma~\ref{lem:herglotz} and the integrability of $\Sol_{z,E}$ that can be deduced from Lemmas~\ref{lem:boundedness-2} and \ref{lem:sinularities-2} yield the absolute continuity of $\rho_{\cstY,\cstS}(d\lambda)=\rho_{\cstY,\cstS}(\lambda) d\lambda$.
\end{proof}
\begin{proof}[Proof of part \emph{(iii)} of Theorem~\ref{thm:main}.]
  Follows from the asymptotic behavior of $\Sol_{z,E}$ near $z=0$ and $z=1$ \eqref{eq:75}-\eqref{eq:140} established in Lemma~\ref{lem:sinularities-2}, block structure of $\Sol_{z,E}$ \eqref{eq:72} and the definition of the self-consistent density of states \eqref{eq:103}.
\end{proof}

\begin{proof}[Proof of part \emph{(iv)} of Theorem~\ref{thm:main}.]
  Follows from the explicit formula for $\Sol_{z,E}$ in the regime $\cstS\to 0$ \eqref{eq:229} established in Section~\ref{sec:expl-solut-phito0}, the block structure of $\Sol_{z,E}$ \eqref{eq:72} and the definition of the self-consistent density of states \eqref{eq:103}.  
\end{proof}

\section{Comparison with the results of Beenakker and Brouwer }
\label{sec:ees=0}

Consider the scattering matrix \eqref{eq:73}
\begin{equation}
\label{eq:9059}
  S(E)
  :=
  I - 2 \cstT \imu W^* (E\cdot I - H + \imu  \cstT W W^*)^{-1} W \in \CC^{N_0\times N_0}
\end{equation}
with $\cstY = E \in \RR$ in the regime $\cstS = N/M \rightarrow 0$ as $M\rightarrow \infty$ with $N_0 = 2N$.
 This model was studied in the case of the Gaussian entries by Beenakker and Brouwer in \cite{Been97,Brou95}, and one of the remarkable results of their theory is that in the experimentally relevant setting of the ideal coupling the limiting transmission eigenvalue density is given by the arcsine law \eqref{bimodal} (see \cite[Eq. (3.12)]{Been11b}.
The ideal coupling assumption is formulated in terms of the matrix $S(E)$ having zero mean \cite[Eq. (3.8)]{Been11b}.
Below we show that in the regime $\cstS \to 0$ the assumption $\EE[S(E)] = 0$ is equivalent to $E=0$ and $\cstT = 1$.
By plugging these values into \eqref{eq:235} and \eqref{eq:236} we recover the arcsine distribution and the corresponding Fano factor $F(0,1) = 1/4$.

 Since the results of the current section do not affect the main outcomes of this paper  and are meant to be of expository nature, we will keep the presentation rather informal, focusing only on the crucial steps and omitting the technical details.

For simplicity we will assume in this section that $H$ is Gaussian.
Note that $W\in \CC^{M\times N_{0}}$, so
\begin{equation}
  \label{eq:9060}
  W
  =
  U \Gamma V^*
\end{equation}
with unitary matrices $U\in \CC^{M\times M}$ and $V\in \CC^{N_{0}\times N_{0}}$ and
\begin{equation}
  \label{eq:9061}
  \Gamma
  =
  \left(
    \begin{array}{c}
      \widetilde{\Gamma}
      \\
      0
    \end{array}
  \right)
  ,
  \quad
  \widetilde{\Gamma}
  =
  \mathrm{diag}(\gamma_{1},\ldots, \gamma_{N_0})
  ,
\end{equation}
where $\gamma_{i}$ are the singular values of $W$ and $N_0 \ll M$.

Note that $N_{0} = 2 N = 2 \cstS M$, and thus the eigenvalues of  $W^*W$ have Marchenko-Pastur distribution with parameter $2\cstS$.
For $\cstS\rightarrow 0$, regime that we are interested in, the eigenvalues of $W^*W$ will be concentrated around point $1$, in the neighborhood of size $O(\sqrt{\cstS})$, so for simplicity of presentation we will omit the asymptotically small term $O(\sqrt{\cstS})$ and  assume throughout these computations that all $\gamma_i=1$, i.e., $\widetilde{\Gamma} = I_{N_0}$.

After applying the singular value decomposition to $W$ and factoring out matrices $U$ and $U^*$ from the inverse, we get
\begin{align}
  \label{eq:9062}
    S(E)
  &=
    I - 2 \cstT \imu V \Gamma^t U^* (E\cdot I - H + \imu  \cstT U \Gamma \Gamma^{t} U^*)^{-1} U \Gamma V^*
  \\ \label{eq:9063}
  &=
    I - 2 \cstT \imu V (\widetilde{\Gamma}^t, 0) \left(E\cdot I - \widetilde{H} + \imu  \cstT \left(
    \begin{array}{cc}
      \widetilde{\Gamma}\widetilde{\Gamma}^{t}& 0
      \\ 
      0 & 0_{(M-N_{0})\times(M-N_0) }
    \end{array} \right)
 \right)^{-1}
         \left( \begin{array}{c}
                  \widetilde{\Gamma}
                  \\
                  0
          \end{array}\right)
  V^*
  \\ \label{eq:9064}
  &=
    I + 2 \cstT \imu V (\widetilde{\Gamma}^t, 0) \left(\widetilde{H} -E\cdot I   - \imu  \cstT \left(
    \begin{array}{cc}
      \widetilde{\Gamma}\widetilde{\Gamma}^{t}& 0
      \\
      0 & 0_{(M-N_{0})\times(M-N_0) }
    \end{array} \right)
 \right)^{-1}
         \left( \begin{array}{c}
                  \widetilde{\Gamma}
                  \\
                  0
          \end{array}\right)
  V^*
  ,
\end{align}
where $\widetilde{H}$ remains a GUE matrix.
Separate the upper-left $N_0 \times N_0$ block of $\widetilde{H}$
\begin{equation}
  \label{eq:9066}
  \widetilde{H}
  =
  \left(
    \begin{array}{cc}
      \widetilde{H}_{1}& \widetilde{H}_{2}
      \\
      \widetilde{H}_{2}^{*} & \widetilde{H}_{3}
    \end{array}
    \right)
  \end{equation}
  and note that $\frac{1}{2\cstS}\widetilde{H}_{1}$ and $ \frac{1}{1-2\cstS}\widetilde{H}_{3}$ are both independent GUE matrices.
  Now the inverse matrix in \eqref{eq:9064} can be rewritten as
  \begin{equation}
    \label{eq:9067}
  \left(
    \begin{array}{cc}
      \widetilde{H}_{1} -E - \imu  \cstT \widetilde{\Gamma}\widetilde{\Gamma}^{t}& \widetilde{H}_{2}
      \\
      \widetilde{H}_{2}^* & \widetilde{H}_{3} -E
    \end{array}
  \right)^{-1}
  .
\end{equation}
Using the Schur complement formula we have that the upper-left $N_{0}\times N_{0}$ block of \eqref{eq:9067}, the only part that does not vanish after sandwiching \eqref{eq:9067} by $(\widetilde{\Gamma}^t, 0)$ and its transpose, is given by
\begin{equation}
  \label{eq:9068}
  \left(
    \widetilde{H}_{1} -E - \imu  \cstT \widetilde{\Gamma}\widetilde{\Gamma}^{t} - \widetilde{H}_{2} \Big(\widetilde{H}_{3} -E\Big)^{-1} \widetilde{H}_{2}^*
  \right)^{-1}
  .
\end{equation}
The semicircular law for the Hermitian (GUE) matrix  $\frac{1}{1-2\cstS}\widetilde{H}_{3}$ implies that
\begin{equation}
  \label{eq:9069}
  (\widetilde{H}_{3} -E)^{-1}
  =
  \frac{1}{1-2\cstS}\Big(\frac{1}{1-2\cstS}\widetilde{H}_{3} -\frac{1}{1-2\cstS}E\Big)^{-1}
  \approx
  \frac{1}{1-2\cstS}m_{sc}\Big(\frac{1}{1-2\cstS}E\Big) I_{M-N_{0}}
  ,
\end{equation}
as  $ M\rightarrow \infty$, where $m_{sc}(z)$ denotes the Stieltjes transform of the semicircular distribution  and ``$\approx$'' denotes that the corresponding equality holds asymptotically with a vanishing additive term and with high probability.
Note that random matrices $\widetilde{H}_{1}$, $\widetilde{H}_{2}$ and $\widetilde{H}_{3}$ are independent. Therefore, by the concentration for quadratic forms (see, e.g.,  \cite[Theorem C.1]{ErdoKnowYauYin13b}) can be approximated by
\begin{equation}
  \label{eq:9070}
  \widetilde{H}_{2} \Big(\widetilde{H}_{3} -E \Big)^{-1} \widetilde{H}_{2}^*
  \approx
  m_{sc}\Big(\frac{1}{1-2\cstS}E\Big) I_{N_{0}}
  .
\end{equation}
From (\ref{eq:9068}) and (\ref{eq:9070}) it remains to check the limiting behavior of 
\begin{equation}
  \label{eq:9071}
    \left(
    \widetilde{H}_{1} -E - \imu  \cstT \widetilde{\Gamma}\widetilde{\Gamma}^{t} - m_{sc}\Big(\frac{1}{1-2\cstS}E\Big)
  \right)^{-1}
.
\end{equation}
Recall that since $\cstS$ is small, we  assumed that $\widetilde{\Gamma}\widetilde{\Gamma}^{t} = I_{N_0}$.
In this case, using again the semicircular law for the Hermitian matrix  $\frac{1}{2\cstS}\widetilde{H}_{1}$, we have
\begin{align}
  \label{eq:9072}
    \left(
    \widetilde{H}_{1} -E - \imu  \cstT \widetilde{\Gamma}\widetilde{\Gamma}^{t} - m_{sc}\Big(\frac{1}{1-2\cstS}E\Big)
  \right)^{-1}
  &\approx
    \left(
    \widetilde{H}_{1} -E - \imu  \cstT  - m_{sc}\Big(\frac{1}{1-2\cstS}E\Big)
  \right)^{-1}
  \\
  & =
   \frac{1}{2\cstS} \left(
    \frac{1}{2\cstS}\widetilde{H}_{1} -\frac{1}{2\cstS}\Big(E + \imu  \cstT  + m_{sc}\Big(\frac{1}{1-2\cstS}E\Big)\Big)
    \right)^{-1}
      \\
  & \approx
   \frac{1}{2\cstS} m_{sc}\left(
    \frac{1}{2\cstS}\Big(E + \imu  \cstT  + m_{sc}\Big(\frac{1}{1-2\cstS}E\Big)\Big)
    \right)
    .
\end{align}
From the asymptotics $z m_{sc}(z) \rightarrow -1, |z|\rightarrow \infty$, we get that
\begin{equation}
  \label{eq:9073}
  \left(
    \widetilde{H}_{1} -E - \imu  \cstT \widetilde{\Gamma}\widetilde{\Gamma}^{t} - m_{sc}\Big(\frac{1}{1-2\cstS}E\Big)
  \right)^{-1}
  \approx
  -\frac{1}{E + \imu  \cstT  + m_{sc}(E)}
  ,\quad
  \cstS \rightarrow 0
  .
\end{equation}
We conclude that as $N_{0}, M \rightarrow \infty$, $\cstS\rightarrow 0$ (see (\ref{eq:9064}) and (\ref{eq:9073}))
\begin{equation}
  \label{eq:9074}
  \EE[S(E)] \approx 1+ 2 \cstT \imu \Big(  -\frac{1}{E + \imu  \cstT  + m_{sc}(E)} \Big)
  .
\end{equation}
Now, from  $m_{sc}(0) = \imu$ we have that
\begin{equation}
  \label{eq:9075}
  \EE[S(0)] \approx 1- \frac{2 \cstT \imu }{\imu  \cstT  + \imu}
  .
\end{equation}
Finally, taking $\cstT = 1$ gives $  \EE[S(0)] \approx 0$.

Note, that from (\ref{eq:9074}) we have that in the limit $\cstS\rightarrow 0$
\begin{equation}
  \label{eq:9076}
  \EE[S(0)] = 0 \quad \Leftrightarrow \quad E + m_{sc}(E) = \imu  \cstT
  ,
\end{equation}
and for $E\in \RR$, the expression $E + m_{sc}(E)$ is purely imaginary if and only if $E=0$
\begin{equation}
  \label{eq:9077}
  E + m_{sc}(E)
  =
  E+\frac{-E + \sqrt{E^2-4}}{2}
  =
  \frac{E + \sqrt{E^2-4}}{2}
  .
\end{equation}
Therefore, for $\cstS\rightarrow 0$, $\EE[S(E)]=0$ if and only if $E=0$ and $\cstT = 1$.

\appendix

\section{Spectral properties for a general class of rational expressions in random matrices}
\label{sec:spectr-prop-gener}

\subsection{Noncommutative (NC) rational expressions and their linearizations}
\label{sec:line-rati-funct}
NC rational expressions are formally defined as expressions obtained by applying the four algebraic operations (including taking inverses) to a tuple of NC variables.
A systematic overview of the abstract theory of NC rational expression and functions (equivalence classes of rational expressions) can be found in \cite{BersReutBook}.
Note that unlike polynomials, rational expressions do not have a canonical representation, which may lead to a situation when, after evaluating on some algebra, two rational expressions represent the same function.
In this paper we leave aside the question of identification of rational functions, and will work instead directly with rational expressions and their evaluations, specifying each time on which domain the evaluation is taking place.
Below we introduce a standard set-up in which we will work and define recursively the classes of rational expressions, denoted by letter $\Qc$, together with corresponding domains of evaluation denoted by $\Dc$.

Let $\Hc$ be a Hilbert space, $\Ac \subseteq \Bc(\Hc)$ be a $C^*$-algebra (of bounded operators on $\Hc$) with norm $\| \cdot \|_{\Ac}$, and let $x_{1},\ldots,x_{\alpha_*},y_{1},\ldots,y_{\beta_*}$ be the NC variables taking values in $\Ac$ with $x_{\alpha}=x_{\alpha}^*$ for $1\leq \alpha \leq \alpha_*$.
Denote by $\Ac_{sa}\subset \Ac$ the set of self-adjoint elements of $\Ac$ and let $\CC\la\xb,\yb,\yb^* \ra$ be the set of polynomials in $\xb:=(x_{1},\ldots,x_{\alpha_*})$, $\yb:=(y_{1},\ldots,y_{\beta_*})$ and $\yb^*:=(y_1^*,\ldots,y_{\beta_*}^*)$.
We define the NC rational expressions recursively on their \emph{height} using the following procedure:
\setlength{\itemsep}{0pt}
\begin{itemize}
\item[(i)] Let $\qb_{0}:= \cstone$.
  The set of rational expressions of height $0$ is defined to be the set of polynomials in $\xb,\yb,\yb^*$ with the domain of definition $\Ac_{sa}^{\alpha_*} \times \Ac ^{\beta_*}$
  \begin{equation}
    \label{eq:113}
    (\Qc_{\qb_0},\Dc_{\qb_0}):=(\CC\la \xb,\yb,\yb^*\ra, \Ac_{sa}^{\alpha_*}\times \Ac^{\beta_*})
    . 
  \end{equation}
\item[(ii)] Let $\qb_{1}:=(q_{1,1},\ldots,q_{1,\ell_{1}}) \in (\Qc_{\qb_0})^{\ell_{1}}$, $\qb_{2}:=(q_{2,1},\ldots,q_{2,\ell_{2}}) \in (\Qc_{\qb_0,\qb_{1}})^{\ell_{2}}$, $\ldots$, $\qb_{n}:=(q_{n,1},\ldots, q_{n,\ell{n}})\in (\Qc_{\qb_0,\ldots,\qb_{n-1}})^{\ell_{n}}$, assuming that $(\Qc_{\qb_{0}},\Dc_{\qb_{0}}),\ldots,(\Qc_{\qb_{0},\ldots,\qb_{n-1}},\Dc_{\qb_{0},\ldots,\qb_{n-1}})$ are defined.
  Then we define
  \begin{align}
    \label{eq:114}
    &\Qc_{\qb_{0},\ldots,\qb_{n-1},\qb_{n}}
      :=
      \CC\Big\la \xb,\yb,\yb^*,\frac{1}{\qb_{1}},\frac{1}{\qb_1^*},\ldots,\frac{1}{\qb_{n}},\frac{1}{\qb_n^*} \Big\ra
      ,
    \\ \label{eq:115}
    &\Dc_{\qb_{0},\ldots,\qb_{n-1},\qb_{n}}
      :=
      \Big\{(\xb,\yb)\in \Dc_{\qb_{0},\qb_{1},\ldots,\qb_{n-1}} : \Big\|\frac{1}{q_{n,j}(\xb,\yb,\yb^*)}\Big\|_{\Ac} < \infty \mbox{ for }  1\leq j \leq \ell_{n}\Big\}
      ,
  \end{align}
  where $\frac{1}{\qb_{i}}:=(q_{i,1}^{-1},\ldots,q_{i,\ell_{i}}^{-1})$ and $\frac{1}{\qb_{i}^*}:=((q_{i,1}^*)^{-1},\ldots,(q_{i,\ell_{i}}^*)^{-1})$.
\end{itemize}
We say that the rational expression $q\in \Qc_{\qb_{0},\ldots,\qb_{n-1},\qb_{n}}$ defined on $\Dc_{\qb_{0},\ldots,\qb_{n-1},\qb_{n}}$ has height $n$.

For any $C>0$ and $\qb_{1}\in (\Qc_{\qb_0})^{\ell_{1}}$, $\qb_{2}\in (\Qc_{\qb_0,\qb_{1}})^{\ell_{2}}$, $\ldots$, $\qb_{n}\in (\Qc_{\qb_0,\ldots,\qb_{n-1}})^{\ell_{n}}$, define the \emph{effective} domain
\begin{equation}
  \label{eq:116}
  \Dc_{\qb_{0},\ldots,\qb_{n};C}
  :=
  \Big\{(\xb,\yb)\in \Dc_{\qb_{0},\ldots,\qb_{n-1};C} : \Big\|\frac{1}{q_{n,j}(\xb,\yb,\yb^*)}\Big\|_{\Ac} \leq C \mbox{ for } 1\leq j \leq \ell_{n}\Big\}
  \subset
  \Dc_{\qb_{0},\ldots,\qb_{n}}
  .
\end{equation}
This domain will allow an effective control of the norm of rational expressions of height $n$.
\begin{rem}
  Similar constructions of rational functions/expressions involving their \emph{height} have been exploited actively in the literature (see, e.g., \cite{Reut96}, \cite[Chapter~4]{BersReutBook}, \cite{Yin18}).
  Note that here we allow the ``denominators'' without constant terms, so they cannot be automatically expanded into geometric series for small $\xb, \yb$.
  Hence we need to introduce and follow explicit domains.
\end{rem} 
\begin{rem}
  In the sequel, the statement ``$q$ is a rational expression of height $n$'' will implicitly mean that there (uniquely) exist $\qb_{1}\in (\Qc_{\qb_0})^{\ell_{1}}$, $\qb_{2}\in (\Qc_{\qb_0,\qb_{1}})^{\ell_{2}}$, $\ldots$, $\qb_{n}\in (\Qc_{\qb_0,\ldots,\qb_{n-1}})^{\ell_{n}}$ and $C>0$ such that $q\in\Qc_{\qb_{0},\ldots,\qb_{n}}$ and $q$ is evaluated on the effective domain $\Dc_{\qb_{0},\ldots,\qb_{n};C}$.
  Note that many of the basic results, in particular about constructing the linearizations, can be formulated in a completely abstract form or without restriction to effective domains.  
\end{rem}
\begin{rem}
  When evaluating a rational expression of height $n$ on different $C^*$-algebras $\Ac_{1},\ldots,\Ac_{k}$, we will use the notation $\Dc_{\qb_{0},\ldots,\qb_{n}}(\Ac_{i})$ and $\Dc_{\qb_{0},\ldots,\qb_{n};C}(\Ac_{i})$ correspondingly.
\end{rem}
\begin{defn}[Self-adjoint rational expression]
  We say that a rational expression $q=q(\xb,\yb, \yb^*)$ is self-adjoint if $q(\xb,\yb, \yb^*) = [q(\xb,\yb, \yb^*)]^{*}$ for all $(\xb,\yb)\in \Dc$.
\end{defn}

\subsection{Linearizations and linearization algorithm}
\label{sec:line-algor}

\begin{defn}
  \label{def:linearization}
  Let $q$ be a self-adjoint rational expression of height $n$ in NC variables $\xb$, $\yb$ and $\yb^*$.
  We say that the self-adjoint matrix
  \begin{equation}
    \label{eq:117}
    \Lb
    =
    \left(
      \begin{array}{c|ccc}
        \lambda & & \bm{\ell}^*&
        \\ \hline
                &&&
        \\
        \bm{\ell} & &\widehat{\Lb} &
        \\
                &&&
      \end{array}
    \right)
    \in (\CC\la \xb,\yb,\yb^*\ra)^{m\times m}
  \end{equation}
  with $\lambda\in\CC\la \xb,\yb,\yb^*\ra$, $\bm{\ell} \in(\CC\la \xb,\yb,\yb^*\ra)^{(m-1)\times 1}$, $\widehat{\Lb}\in(\CC\la \xb,\yb,\yb^*\ra)^{(m-1)\times (m-1)}$,  whose entries are polynomials of degree at most $1$, is a \emph{(self-adjoint) linearization} of $q$ if 
  \begin{itemize}\setlength{\itemsep}{0pt}
  \item[(i)] the submatrix $\widehat{\Lb}\in (\CC\la \xb,\yb,\yb^* \ra)^{(m-1)\times (m-1)}$ is invertible, and
  \item[(ii)] $\lambda - \bm{\ell}^* {\widehat{\Lb}}^{-1} \bm{\ell} = q$
  \end{itemize}
  for all $(\xb,\yb)\in\Dc_{\qb_{0},\ldots,\qb_{n}}$.
\end{defn}

The linearization of $q$ can be written as
\begin{equation}
  \label{eq:118}
  \Lb
  =
  K_0 \otimes \cstone - \sum_{\alpha=1}^{\alpha_*} K_{\alpha} \otimes x_{\alpha} - \sum_{\beta=1}^{\beta_*}\big( L_{\beta} \otimes y_{\beta} + L_{\beta}^* \otimes y_{\beta}^*\big)
  ,
\end{equation}
where $K_0,K_{\alpha},L_{\beta} \in \CC^{m\times m}$.

The idea of studying noncommutative rational functions/expressions via linearizations goes back to Kleene \cite{Klee56}.
Since the publication of this work various approaches and algorithms have been developed for constructing linearizations of general classes of rational functions/expressions (see, e.g., \cite{BersReutBook} or \cite{HeltMaiSpei18} for a pedagogical presentation of the subject).
For reader's convenience we provide below a simple linearization algorithm based on the method described in \cite[Section~A.1]{ErdoKrugNemi_Poly}.
We use the following observation: for matrices $\Ab_{i}\in (\CC\la \xb,\yb,\yb^*\ra)^{m_i\times m_i}$ and $\Bb_{j} \in (\CC\la \xb,\yb,\yb^*\ra)^{m_{j}\times m_{j+1}}$, $m_{j}\in \NN$, the lower right $m_1 \times m_{k}$ submatrix of the inverse of the block matrix
\begin{equation}
  \label{eq:119}
    \left(
      \begin{array}{ccccccc}
        &&&\Bb_{1}&-\Ab_{1}
        \\
        && \Bb_{2} & -\Ab_{2} &
        \\
        &\iddots &\iddots&&
        \\
        \Bb_{k-1}& -\Ab_{k-1} &    && 
        \\
        -\Ab_{k} &  &    &&
      \end{array}
    \right)
\end{equation}
is equal to
\begin{equation}
  \label{eq:120}
  -\Ab_{1}^{-1}\Bb_{1}\Ab_{2}^{-1}\Bb_{2}\cdots \Ab_{k-1}^{-1} \Bb_{k-1} \Ab_{k}^{-1}
  \in
  (\CC\la \xb,\yb,\yb^*\ra)^{m_1\times m_k}
  .
\end{equation}
Now, if $q_{1},\ldots,q_{\ell}$ are rational expressions having known linearizations $\Ab_{1}, \ldots, \Ab_{\ell}$, then one can easily check that the linearization of the product $w_0 \frac{1}{q_{1}}w_{1}\frac{1}{q_{2}}\cdots \frac{1}{q_{k-1}}w_{k-1}\frac{1}{q_{k}}w_{k}$ with $w_{j}\in \{\cstone, x_{\alpha}, y_{\beta}, y_{\beta}^{*} \, | \,1\leq \alpha \leq \alpha_*, 1\leq \beta \leq \beta_{*}\}$ can be given by
\begin{equation}
  \label{eq:121}
  \left(
    \begin{array}{c|ccccc}
      &&&&&\bb_{0}
      \\ \hline
      &&&&\Bb_{1}&-\Ab_{1}
      \\
      &&& \Bb_{2} & -\Ab_{2} &
      \\
      &&\iddots &\iddots&&
      \\
      &\Bb_{k-1}& -\Ab_{k-1} &    && 
      \\
      \bb_{k} &-\Ab_{k} &  &    &&
    \end{array}
  \right)
\end{equation}
with $\bb_0=(w_0,0,\ldots,0)\in (\CC\la \xb,\yb,\yb^*\ra)^{1\times m_{1}}$, $\bb_k=(w_{k},0,\ldots,0)^{t} \in (\CC\la \xb,\yb,\yb^*\ra)^{m_{k}\times 1}$ and  $\Bb_{j}$ being matrices of the form
\begin{equation}
  \label{eq:122}
  \Bb_{j}
  =
  \left(
    \begin{array}{cccc}
      w_{j}&0&\cdots&0
      \\
      0& 0 &\cdots & 0
      \\
      \vdots & \vdots &\ddots   & \vdots
      \\
      0 & 0 & \cdots & 0
    \end{array}
  \right)
  \in
  (\CC\la \xb,\yb,\yb^*\ra)^{m_{j}\times m_{j+1}}
  .
\end{equation}
Using \eqref{eq:121}, we can construct a linearization of rational expressions using the induction on the height.
In the case of a rational expression of height $0$ (polynomial function), linearization can be constructed using, e.g., the algorithm from \cite[Section~A.1]{ErdoKrugNemi_Poly}.
If $\qb_{1}\in (\Qc_{\qb_{0}})^{\ell_{1}}$ and $q\in \Qc_{\qb_0,\qb_1}$ is a rational expression of height $1$, the linearization can be obtain using the following algorithm:
\setlength{\itemsep}{0pt}
\begin{itemize}
\item[(B0)] write $q$ as a sum of monomials in $\xb$, $\yb$, $\yb^*$, $\frac{1}{\qb_{1}}$, $\frac{1}{\qb_{1}^*}$ of the form $w_{1}\frac{1}{q_{1}}w_{2}\frac{1}{q_{2}}\cdots w_{k-1}\frac{1}{q_{k-1}} w_{k}$ with $w_{j} \in \{\cstone,x_{\alpha},y_{\beta},y_{\beta}^* \, | \, 1\leq \alpha \leq \alpha_*,1\leq \beta \leq \beta_*\}$ and $q_{i}\in \{\cstone, q_{1,\gamma}, q_{1,\gamma}^* \, | \, 1\leq \gamma\leq \ell_{1}\}$ using $\frac{1}{\cstone}=\cstone$ if necessary;
\item[(B1)] for each polynomial $q_{1,\gamma_{1}}, q_{1,\gamma_{1}}^*, 1\leq \gamma_{1}\leq \ell_{1}$ construct a linearization (not necessarily self-adjoint) using the algorithm from \cite[Section~A.1]{ErdoKrugNemi_Poly};
\item[(B2)] linearization of a monomial in $\xb$, $\yb$, $\yb^*$, $\frac{1}{\qb_{1}}$, $\frac{1}{\qb_{1}^*}$ of the form $w_{1}\frac{1}{q_{1}}w_{2}\frac{1}{q_{2}}\cdots w_{k-1}\frac{1}{q_{k-1}} w_{k}$ is given by \eqref{eq:121} with $A_{i}$ being either linearizations of polynomials $q_{i}\in \{q_{1,\gamma_{1}}, q_{1,\gamma_{1}}^*\, | \, 1\leq \gamma_{1} \leq \ell_{1} \}$, or $\cstone$ (linearization of $\frac{1}{\cstone} = \cstone$), and $w_{j} \in \{\cstone,x_{\alpha},y_{\beta},y_{\beta}^* \, | \, 1\leq \alpha \leq \alpha_*,1\leq \beta \leq \beta_*\}$;
\item[(B3)] the (possibly not self-adjoint) linearization of a linear combination of monomials (and thus $q$) is constructed by putting the linearizations of monomials obtained at (B2) into a block-diagonal form using a procedure similar to (R1)-(R2) from \cite[Section~A.1]{ErdoKrugNemi_Poly};
\item[(B4)] if after step (B3) the resulting linearization is not self-adjoint,  the symmetrized linearization of $q = (q+q^*)/2$ can be obtained by putting the linearization obtained at step (B3) and its conjugate transpose into a block-skew-diagonal form similarly as in (R3) from \cite[Section~A.1]{ErdoKrugNemi_Poly}.
\end{itemize}
Suppose that we know how to construct linearizations for rational expressions of height $\leq n-1$ and suppose that $q\in \Qc_{\qb_{0},\ldots,\qb_{n}}$ is a rational expression of height $n$.
Then the linearization of $q$ can be constructed using the following algorithm,  that is an adaptation of (B0)-(B4):
\begin{itemize}\setlength{\itemsep}{0pt}
\item[(S0)]  write $q$ as a sum of monomials in $\xb$, $\yb$, $\yb^*$, $\frac{1}{\qb_{1}}$, $\frac{1}{\qb_{1}^*},\ldots,\frac{1}{\qb_{n}},\frac{1}{\qb_{n}^*}$ of the form $w_{1}\frac{1}{q_{1}}w_{2}\frac{1}{q_{2}}\cdots w_{k-1}\frac{1}{q_{k-1}} w_{k}$ with $w_{j} \in \{\cstone,x_{\alpha},y_{\beta},y_{\beta}^* \, | \, 1\leq \alpha \leq \alpha_*,1\leq \beta \leq \beta_*\}$ and $q_{i}\in \{\cstone, q_{t,\gamma_{t}}, q_{t,\gamma_{t}}^* \, | \, 1\leq t \leq n, 1\leq \gamma_{t} \leq \ell_{n}\}$ using $\frac{1}{\cstone}=\cstone$ if necessary;
\item[(S1)]  construct linearizations (not necessarily self-adjoint) of each rational expression  $q_{t,\gamma_{t}}$ for  $1\leq t \leq n$, $1\leq \gamma_{t}\leq \ell_{t}$ of height $\leq n-1$ (rational expressions of height $\leq n-1$);
\item[(S2)] linearization of a monomial in $\xb$, $\yb$, $\yb^*$, $\frac{1}{\qb_{1}}, \frac{1}{\qb_{1}^*}, \ldots, \frac{1}{\qb_{n}}, \frac{1}{\qb_{n}^*}$ of the form $w_{1}\frac{1}{q_{1}}w_{2}\frac{1}{q_{2}}\cdots w_{k-1}\frac{1}{q_{k-1}} w_{k}$ is given by \eqref{eq:121} with $A_{i}$ being either linearizations of polynomials $q_{i}\in \{q_{t,\gamma_{t}}, q_{t,\gamma_{t}}^* \, | \, 1\leq t \leq n, 1\leq \gamma_{t} \leq \ell_{n}\}$, or $\cstone$, and $w_{j} \in \{\cstone,x_{\alpha},y_{\beta},y_{\beta}^* \, | \, 1\leq \alpha \leq \alpha_*,1\leq \beta \leq \beta_*\}$.
\end{itemize}
The steps (S3)-(S4) are identical to (B3)-(B4).
\begin{rem}
  It is always possible to obtain the specific form of monomials required in steps (B0) and (S0) by adding factors $\cstone$ or $\frac{1}{\cstone}$, as, for example, in
  \begin{equation}
    \label{eq:123}
    \frac{1}{(\cstone-x)^{2}}x^2
    =
    \cstone \frac{1}{\cstone - x} \cstone \frac{1}{\cstone - x} x \frac{1}{\cstone} x
    .
  \end{equation}
  Note that there is some ambiguity at steps (B0) and (S0), representing the fact that we have freedom to choose the order in which the monomials appear in the sum, as well as freedom to put the product of the constant terms $\cstone$ and $\frac{1}{\cstone}$ between the terms $w_j$ and $\frac{1}{q_i}$.
  All the results of Section~\ref{sec:spectr-prop-gener} hold for any choice of the representation of $q$ as a sum of monomials, therefore we do not fix any particular order or other rules to guarantee the uniqueness of the representation in (B0) and (S0).
\end{rem}
The condition (ii) in the definition of the linearization is satisfied by construction.
The condition (i) follows from the following lemma, which gives a bound on $\|( \widehat{\Lb}(\xb,\yb,\yb^*))^{-1}\|_{\CC^{(m-1)\times (m-1)}\otimes \,\Ac}$, where for any $n\in \NN$ and $\Rb = (\Rb_{ij})_{i,j=1}^{n} \in \CC^{n\times n}\otimes \Ac$ we denote
\begin{equation}
  \label{eq:124}
  \| \Rb \|_{\CC^{n\times n}\otimes \,\Ac}
  :=
  \max_{1\leq i,j\leq n} \|\Rb_{ij}\|_{\Ac}
  .
\end{equation}

\begin{lem}[Invertibility of $\widehat{\Lb}$]
  \label{lem:invertibility}
Let $q$ be a self-adjoint rational expression of height $n$ and 
  let $\Lb  = \Lb_{q} \in (\CC\la \xb,\yb,\yb^*\ra)^{m\times m} $ be the linearization of $q$ constructed via the above algorithm.
  Let $\widehat{\Lb}$ be the submatrix of $\Lb$ defined using the decomposition \eqref{eq:117}.
  Then there exist $\cstG>0$ and $\cstH \in \NN$ such that for any $(\xb,\yb)\in \Dc_{\qb_{0},\ldots,\qb_{n};C}$
  \begin{equation}
    \label{eq:125}
    \| (\widehat{\Lb}(\xb,\yb,\yb^*))^{-1} \|_{\CC^{(m-1)\times (m-1)}\otimes \,\Ac}
    \leq
    \cstG (1 + C + \max_{\alpha}\|x_{\alpha}\|_{\Ac} + \max_{\beta}\|y_{\beta}\|_{\Ac})^{\cstH}
    .
  \end{equation}
\end{lem}
\begin{proof}
  We prove \eqref{eq:125} by induction on $n$.
  For $n=0$ (the special case of polynomial functions) \eqref{eq:125} follows from, for example, \cite[(3.16)]{ErdoKrugNemi_Poly}.
  Suppose \eqref{eq:125} holds for all rational expressions of height $k\leq n-1$.
  Consider $q$ of height $n$ with linearization obtained via (S0)-(S4).
  Steps (S3) and (S4) of the linearization algorithm endow $\widehat{\Lb}$ with block-diagonal (S3) or block-skew-diagonal (S4) structure with blocks being the linearizations of monomials obtained at step (S2).
  Therefore, in order to obtain the bound \eqref{eq:125} it is enough to consider only the inverses of the blocks of the form
  \begin{equation}
    \label{eq:126}
    \left(
      \begin{array}{ccccccc}
        &&&\Bb_{1}& -\Ab_{1}
        \\
        && \Bb_{2} & -\Ab_{2} &
        \\
        &\iddots &\iddots&&
        \\
        \Bb_{k-1}&- \Ab_{k-1} &    && 
        \\
        - \Ab_{k} &  &    &&
      \end{array}
    \right)
    \in (\CC\la \xb, \yb ,\yb^*\ra)^{m'\times m'}
  \end{equation}
  with $\Ab_{i}$ being the linearizations of the rational expressions $q_{i}\in \{\frac{1}{q_{t,\gamma_{t}}},\frac{1}{q_{t,\gamma_{t}}^*}\, | \, 0 \leq t \leq n, 1\leq \gamma_{t} \leq \ell_{t}\}$, and $\Bb_{j}$ being of the form \eqref{eq:122}.
  One can easily check that the inverse of \eqref{eq:126} consists of the blocks of the type $\Ab_{i}^{-1} \Bb_{i} \Ab_{i + 1}^{-1} \Bb_{i + 1}\ldots$.
  The induction step, together with the Schur complement formula for $\Ab_{i}^{-1}$ and the condition that for $(\xb,\yb) \in \Dc_{\qb_{0},\ldots,\qb_{n};C}$
  \begin{equation}
    \label{eq:127}
    \Big\| \frac{1}{q_{i}(\xb,\yb,\yb^*)}\Big\|_{\Ac}
    \leq
    C
  \end{equation}
  implies that for each block of type \eqref{eq:126} there exist $C'>0$ and $n'\in \NN$ such that
  \begin{equation}
    \label{eq:128}
    \left\|
      \left(
        \begin{array}{ccccccc}
          &&&\Bb_{1}&-\Ab_{1}
          \\
          && \Bb_{2} & -\Ab_{2} &
          \\
          &\iddots &\iddots&&
          \\
          \Bb_{k-1}& -\Ab_{k-1} &    && 
          \\
          -\Ab_{k} &  &    &&
        \end{array}
      \right)^{-1} 
    \right\|_{\CC^{m'\times m'}\otimes \,\Ac}
    \leq
    C'(1 + C + \max_{\alpha}\|x_{\alpha}\|_{\Ac} + \max_{\beta}\|y_{\beta}\|_{\Ac})^{n'}
    ,
  \end{equation}
  where $\Ab_{i}$'s and $\Bb_{j}$'s in the left-hand side are evaluated at $(\xb,\yb)\in \Dc$.
  Taking $\cstG$ and $\cstH$, respectively, the maximum over all $C'$'s and the maximum over all $n'$'s in the bounds \eqref{eq:128} running through all monomials in the representation of $q$, leads to \eqref{eq:125}
\end{proof}
\begin{rem}
  Suppose that $P\in \CC^{m\times m}$ is of the form
  \begin{equation}
    \label{eq:129}
    P
    =
    \left(
      \begin{array}{c|ccc}
        1& 0 & \cdots & 0
        \\ \hline
        0 & & &
        \\
        \vdots && Q & 
        \\
        0 & & &
      \end{array}
      \right)
    \end{equation}
    with $Q\in\CC^{(m-1)\times (m-1)}$ invertible.
    It is easy to see that if  $\Lb\in (\CC\la \xb,\yb,\yb^*\ra)^{m\times m}$ is a linearization of a rational expression of height $q$, then so is $(P\otimes \cstone) \Lb (P^{-1} \otimes \cstone)$.
    We will use this freedom to bring linearizations to more convenient form. 
\end{rem}

\subsection{ A priori  bound on generalized resolvents}
\label{sec:trivial-bound}

 \begin{defn}[Generalized resolvent] \label{defn:a-generalized-resolvent}
  Let $\Lb\in (\CC\la \xb,\yb,\yb^*\ra)^{m\times m}$. 
  We call the matrix-valued function $z \mapsto (\Lb - z J_{m}\otimes \cstone)^{-1}$ defined for $z\in \CC_+$ the \emph{generalized resolvent} of $\Lb$.
\end{defn}

\begin{lem}
  \label{lem:trivialBound}
  Let $q$ be a self-adjoint rational expression of height $n$ and let $\Lb \in (\CC\la \xb,\yb,\yb^*\ra)^{m\times m} $ be the linearization of $q$ constructed using the algorithm from Section~\ref{sec:line-algor}.
  Then there exist $\cstI>0$ and $\cstK\in \NN$ such that for all $C>1$, $(\xb,\yb)\in \Dc_{\qb_{0},\ldots,\qb_{n};C}$ and $z\in \CC_+$
  \begin{equation}
    \label{eq:130}
    \|(\Lb(\xb,\yb,\yb^*)-zJ_{m}\otimes \cstone)^{-1}\|_{\CC^{m\times m}\otimes \,\Ac}
    \leq
    \cstI(1+C + \max_{\alpha}\|x_{\alpha}\|_{\Ac} + \max_{\beta} \|y_{\beta}\|_{\Ac})^{\cstK}\Big(1+\frac{1}{\Im z}\Big)
    .
  \end{equation}
\end{lem}
\begin{proof}
  Rewrite $\Lb - zJ_{m}\otimes \cstone$ using the block decomposition from \eqref{eq:117}
  \begin{equation}
    \label{eq:131}
    \Lb - zJ_{m}\otimes \cstone
    =
    \left(
      \begin{array}{c|ccc}
        \lambda - z \cstone&& \bm{\ell}^*&
        \\ \hline
                           &&&
        \\
        \bm{\ell} & &\widehat{\Lb} &
        \\
                           &&&                          
      \end{array}
    \right)
  \end{equation}
  with $\widehat{\Lb}\in(\CC\la \xb,\yb,\yb^*\ra)^{(m-1)\times (m-1)}$.
  By (ii) in the definition of the linearization and the Schur complement formula we have
  \begin{equation}
    \label{eq:132}
    (\Lb - zJ_{m}\otimes \cstone)^{-1}
    =
    \left(
      \begin{array}{c|ccc}
        (q-z\cstone)^{-1}& &-(q-z\cstone)^{-1} \bm{\ell}^* \widehat{\Lb}^{-1}
        \\ \hline
                         &&&
        \\
        -\widehat{\Lb}^{-1} \bm{\ell} (q-z\cstone)^{-1} & &\widehat{\Lb}^{-1}+\widehat{\Lb}^{-1} \bm{\ell} (q-z\cstone)^{-1} \bm{\ell}^* \widehat{\Lb}^{-1} &
        \\
                         &&&                          
      \end{array}
    \right).
  \end{equation}
  Now \eqref{eq:130} follows from Lemma~\ref{lem:invertibility} and the trivial bound for resolvents of self-adjoint elements
  \begin{equation}
    \label{eq:133}
    \Big\|\frac{1}{q - z\cstone }\Big\|_{\Ac}
    \leq
    \frac{1}{\Im z}
    \quad
    \mbox{uniformly for } z\in\CC_{+}
    .
  \end{equation}
\end{proof}

\subsection{Dyson equation for linearizations of NC rational expressions}
\label{sec:exist-uniq-solut}

Let $q$ be a self-adjoint rational expression of height $n$ and let $\Lb\in (\CC\la \xb,\yb,\yb^*\ra)^{m\times m} $ be its linearization constructed using the algorithm from Section~\ref{sec:line-algor}.
Write $\Lb$ as
\begin{equation}
  \label{eq:134}
  \Lb
  =
  K_{0} \otimes \cstone - \sum_{\alpha = 1}^{\alpha_*} K_{\alpha} \otimes x_{\alpha}- \sum_{\beta = 1}^{\beta^*} \big( L_{\beta} \otimes y_{\beta} + L_{\beta}^* \otimes y_{\beta}^* \big)
\end{equation}
with $K_0,K_{\alpha},L_{\beta} \in \CC^{\,m\times m}$ and $K_{0}$, $K_{\alpha}$ self-adjoint.
Define the completely positive map  $\SuOp : \CC^{\,m \times m} \rightarrow \CC^{\,m \times m}$ by
\begin{equation}
  \label{eq:135}
  \SuOp\big[R\,\big]
  =
  \sum_{\alpha = 1}^{\alpha_*} K_{\alpha} R K_{\alpha} + \sum_{\beta=1}^{\beta_*} \big( L_{\beta} R L_{\beta}^* + L_{\beta}^* R L_{\beta}\big)
  ,\quad
  R\in \CC^{m \times m}
  .
\end{equation}
\begin{defn}[Dyson equation for linearizations]
We call the equation
\begin{equation}
  \label{eq:136}
  -\frac{1}{\Sol}
  =
  z J_{m} - K_{0} + \SuOp[ \Sol ]
\end{equation}
the \emph{Dyson equation for the linearization (DEL)} (of a rational expression).
\end{defn}

\begin{lem}[Solution of \emph{DEL}: existence and basic properties]\label{lem:exist-uniq}
  Let $\Hc$ be a Hilbert space and let $\Sc \subset \Bc(\Hc)$ be a $C^*$-algebra containing a freely independent family $\{\semic_1,\ldots,\semic_{\alpha_*}, c_{1},\ldots,c_{\beta_*}\}$ of $\alpha_*$ semicircular and $\beta_*$ circular elements in a NC probability space $(\Sc,\tau_{\Sc})$.
  Let $q\in \Qc_{\qb_{0},\ldots,\qb_{n}}$ and assume that $(\sbm,\cb)\in \Dc_{\qb_{0},\ldots,\qb_{n};C}$ for $\sbm = (\semic_{1},\ldots,\semic_{\alpha_*})$, $\cb:=(c_1,\ldots,c_{\beta_*})$ and some $C>0$.
  Define
  \begin{equation}
    \label{eq:137}
    \Sol^{(\mathrm{sc})}_{z}
    :=
    (\id_{m}\otimes \,\tau_{\Sc})\bigg( \Big[(K_0-zJ)\otimes \scone - \sum_{\alpha=1}^{\alpha_*} K_{\alpha} \otimes \semic_{\alpha} - \sum_{\beta = 1}^{\beta_*}\big(L_{\beta} \otimes c_{\beta} + L_{\beta}^* \otimes c_{\beta}^*\big)\Big]^{-1}\bigg)
    .
  \end{equation}
  Then
  \begin{itemize}\setlength{\itemsep}{0pt}
  \item[(i)] there exists $\cstJ>0$ such that 
    \begin{equation}
      \label{eq:138}
      \|\Sol^{(\mathrm{sc})}_{z}\|_{ \CC^{m\times m}}
      \leq
      \cstJ \Big(1 + \frac{1}{\Im z}\Big)
      ;  
    \end{equation}
  \item[(ii)] $\Sol^{(\mathrm{sc})}_{z}$ satisfies the DEL \eqref{eq:136};
  \item[(iii)] $\Sol^{(\mathrm{sc})}_{z} $ depends analytically on $z$;
  \item [(iv)] $\Im \Sol^{(\mathrm{sc})}_{z} \geq 0$;
  \item[(v)] $\Sol^{(\mathrm{sc})}_{z}$ admits the representation 
    \begin{equation}
      \label{eq:139}
      \Sol^{(\mathrm{sc})}_{z} 
      =
      \Sol^{\infty} + \int_{\RR} \frac{V(d\lambda)}{\lambda-z}
      ,
    \end{equation}
    where  $\Sol^{\infty}\in \CC^{m \times m }$ is a self-adjoint matrix, and $V(d\lambda)$ is a positive semidefinite  matrix-valued measure  on $\RR$ with compact support;
  \item[(vi)] for almost every  $\lambda \in R$ the limit $\lim_{\eta \rightarrow 0} \pi^{-1} \Im \Sol_{\lambda + \imu \eta} = V(\lambda) \in \CC^{m\times m}$ exists; if the limit is finite on some interval $I\subset \RR$ everywhere, then $V(d \lambda)$ is absolutely continuous on $I$ and $V(d\lambda) = V(\lambda) d\lambda$;
  \item[(vii)] $\mathrm{supp}(V_{11}) = \mathrm{supp}(\Tr V)$.
  \end{itemize}
\end{lem}
\begin{proof}    
  \textbf{Proof of $(i)$.} It follows from Lemma~\ref{lem:trivialBound} and the norm bounds for semicircular and circular operators
  \begin{equation}
    \label{eq:142}
    \|s_{\alpha}\|_{\Sc}
    =
    2,\quad
    \|c_{\beta}\|_{\Sc}
    =
    2,\quad
    1\leq \alpha \leq \alpha_*
    ,\quad
    1\leq \beta \leq \beta_*
  \end{equation}
  that 
  \begin{equation}
    \label{eq:143}
        \Big\|\Big((K_{0}-zJ)\otimes \scone -\sum_{\alpha=1}^{\alpha_*} K_{\alpha} - \sum_{\beta = 1}^{\beta_*} \big(L_{\beta} \otimes c_{\beta} + L_{\beta}^* \otimes c_{\beta}^* \big)\otimes \semic_{\alpha}\Big)^{-1}\Big\|_{\CC^{m\times m}\otimes \,\Sc}
    \leq
    \cstJ \Big(1+\frac{1}{\Im z}\Big)
  \end{equation}
  for $\cstJ:= \cstI (1+C + 4)^{\cstK}$.
  
  \textbf{Proof of $(ii)$.}
    First note, that the real and imaginary parts of free circular elements form a freely independent family of semicirculars.
  Therefore, by defining for each $1\leq \beta \leq \beta_*$
  \begin{equation}
    \label{eq:144}
    s_{\alpha_*+\beta} := \sqrt{2}\Re c_{\beta}
    ,\quad
    s_{\alpha_*+\beta_*+\beta} := \sqrt{2}\Im c_{\beta}
    , \quad
    K_{\alpha_* + \beta} := \sqrt{2} \Re L_{\beta}
    ,\quad
    K_{\alpha_* + \beta_* + \beta} := -\sqrt{2} \Im L_{\beta}
    ,
  \end{equation}
  $\Sol^{(\mathrm{sc})}_{z}$ can be rewritten as 
  \begin{equation}
    \label{eq:145}
    \Sol^{(\mathrm{sc})}_{z}
    =
      (\id_{m}\otimes \,\tau_{\Sc})\bigg( \Big[(K_0-zJ)\otimes \scone - \sum_{\alpha=1}^{\alpha_*+2\beta_*} K_{\alpha} \otimes \semic_{\alpha} \Big]^{-1}\bigg)
    \end{equation}
    and it will be enough to show that $\Sol^{(\mathrm{sc})}_{z}$ satisfies the \emph{DEL} \eqref{eq:136} with $\SuOp[\,\cdot\,] = \sum_{\alpha=1}^{\alpha + 2\beta_*} K_{\alpha} \,\cdot\,  K_{\alpha}$.
    This last fact can be established via the argument similar to the proof of the  existence of the solution to the \emph{DEL} for the linearizations of \emph{polynomials} in \cite[Lemma~2.6 (iv)]{ErdoKrugNemi_Poly}.
    This proof relies on the results of \cite[Proposition~4.1]{Lehn99} and \cite[Lemma~5.4]{HaagThor05} establishing the existence of the solution to a particular class of \emph{matrix Dyson equations (MDE)}, as well as regularization technique which allows to extend these results from \emph{MDE} to \emph{DEL}.
    The trivial bound from \cite[Lemma~2.5]{ErdoKrugNemi_Poly}, which justifies the application of the Schur complement formula and relies on the nilpotency structure of the linearizations, in the setting of the current paper can be replaced by the bound \eqref{eq:143} coming from the specific choice of the domain of evaluation.

  \textbf{Proof of $(iii)-(vii)$.}
  The analyticity of $\Sol^{(\mathrm{sc})}_{z}$ follows from \eqref{eq:143} and the positive semidefiniteness of $\Im \Sol^{(\mathrm{sc})}_{z}$ is a direct consequence of the representation
  \begin{equation}
    \label{eq:152}
    \Im \Sol^{(\mathrm{sc})}_{z}
    =
    \eta\, (\id_{m} \otimes \,\tau_{\Sc})\Big( \big( \Lb^{(\mathrm{sc})} - \overline{z}J_{m}\otimes \scone \big)^{-1} (J_{m} \otimes \scone )\big( \Lb^{(\mathrm{sc})} - z J_{m}\otimes \scone \big)^{-1}\Big)
  \end{equation}
  with $\Lb^{(\mathrm{sc})}:= K_0\otimes \scone - \sum_{\alpha=1}^{\alpha_*+2\beta_*} K_{\alpha} \otimes \semic_{\alpha}$.
  Properties $(v)-(vii)$ follow from the general properties of matrix-valued Herglotz functions, Schur formula \eqref{eq:132} applied to $(\Lb^{(\mathrm{sc})} - zJ_{m}\otimes \scone )^{-1}$ and the bound \eqref{eq:143} using the similar argument as in the proof of \cite[Lemma~2.7]{ErdoKrugNemi_Poly}.
\end{proof}

\subsection{Convergence of spectrum for the rational expressions in random matrices  and the a priori
 bound for  the generalized resolvent in random matrices}
\label{sec:conv-spectr-rati}

The next two sections are devoted to the study of the eigenvalues of a general class of rational expressions evaluated on random Wigner and \emph{iid} matrices.
\begin{ass}[Wigner and \emph{iid} matrices]
Let $X_{1},\ldots, X_{\alpha_*}\in\CC^{N\times N}$ and $Y_{1},\ldots,Y_{\beta_*} \in \CC^{N\times N}$ be two independent families of independent random matrices satisfying the following assumptions
\begin{itemize}
\item[\textrm{(\textbf{H1})}] $X_{\alpha}=(X_{\alpha}(i,j))_{i,j=1}^{N}$, $1\leq \alpha\leq \alpha_*$, are Hermitian random matrices having independent (up to symmetry constraints) centered entries of variance $1/N$;
\item[\textrm{(\textbf{H2})}] $Y_{\beta}=(Y_{\beta}(i,j))_{i,j=1}^{N}$, $1\leq \beta\leq \beta_*$, are (non-Hermitian) random matrices having independent centered entries of variance $1/N$;
\item[\textrm{(\textbf{H3})}] there exist $\cstA_{n}>0$, $n \in \NN$, such that
  \begin{equation}
\label{eq:154}
    \max_{1\leq i,j \leq N} \Big( \max_{1\leq \alpha \leq \alpha_*} \EE\big[|\sqrt{N}X_{\alpha}(i,j)|^{n}\big] + \max_{1\leq \beta \leq \beta_*} \EE\big[|\sqrt{N}Y_{\beta}(i,j)|^{n}\big]\Big)
    \leq
    \cstA_{n}
    .
  \end{equation}

\end{itemize}
We call $X_{\alpha}$ \emph{Wigner} matrices and $Y_{\beta}$ \emph{iid} matrices.
\end{ass}

Denote $\Xb:=(X_{1},\ldots,X_{\alpha_*})$, $\Yb:=(Y_{1},\ldots,Y_{\beta_*})$, $\Yb^*:=(Y_{1}^*,\ldots,Y^*_{\beta_*})$ and let $q$ be a (self-adjoint) rational expression in $\alpha_*$ self-adjoint and $\beta_*$ non self-adjoint noncommutative variables.
In order to prove the local law for $q(\Xb,\Yb,\Yb^*)$ we will need to show that the spectrum of $q(\Xb,\Yb,\Yb^*)$ converges to the spectrum of $q(\sbm,\cb,\cb^*)$.
To this end, for any $\varepsilon>0$, $m\in \NN$ and operator $\Rb\in \CC^{mN\times mN}$, denote by $\mathrm{Spec}_{\varepsilon}(\Rb)$ the $\varepsilon$-pseudospectrum of $\Rb$ defined by
\begin{equation}
  \label{eq:202}
  \mathrm{Spec}_{\varepsilon}(\Rb)
  =
  \mathrm{Spec}(\Rb)\cup \{z\in\CC\, : \| (\Rb - z I_{m}\otimes I_{N})^{-1}\|_{\CC^{mN\times mN}} \geq \varepsilon^{-1}\}
  .
\end{equation}
It is easy to check that for any $\Rb\in \CC^{mN\times mN} \cong \CC^{m\times m} \otimes \CC^{N\times N}$
\begin{equation}
  \label{eq:213}
  \| \Rb \|_{\CC^{mN\times mN}}
  \leq
  m \, \| \Rb \|_{\CC^{m\times m}\otimes \,\CC^{N\times N}}
  ,
\end{equation}
where $\| \, \cdot \, \|_{\CC^{mN\times mN}}$ and $\| \, \cdot \, \|_{\CC^{N\times N}}$ denote the operator norms on $\CC^{mN \times mN}$ and $\CC^{N\times N}$ correspondingly.

For any (not necessarily self-adjoint) rational expression $r(\xb,\yb,\yb^*)$ in NC variables $\xb,\yb,\yb^*$, denote by $\Lb_{r}:=\Lb_{r}(\xb,\yb,\yb^*) \in (\CC\la \xb,\yb,\yb^*\ra)^{m_{r}\times m_{r}}$ its (not necessarily self-adjoint) linearization, which can be constructed using, for example, the algorithm from  Section~\ref{sec:line-algor} omitting steps \textrm{(B4)} and \textrm{(S4)} if $r$ is not self-adjoint.
Define the corresponding \emph{Hermitized linearization} by
\begin{equation}
  \label{eq:191}
  \Lb^{r,z}(\xb,\yb,\yb^*)
  :=
  \left(
    \begin{array}{cc}
      0 & \Lb_{r}(\xb,\yb,\yb^*)-zJ_{m_{r}}\otimes \cstone
      \\
      \big(\Lb_{r}(\xb,\yb,\yb^*)\big)^{*}-\overline{z}J_{m_{r}}\otimes \cstone & 0
    \end{array}
  \right)
\end{equation}
with $\Lb^{r,z}(\xb,\yb,\yb^*)   \in (\CC\la \xb,\yb,\yb^*\ra )^{2m_{r}\times 2m_{r}}$.
For the Hermitized linearization \eqref{eq:191} we define (similarly as in \eqref{eq:135}) the \emph{self-energy operator} $\SuOp^{r,z}:\CC^{2m_{r}\times 2m_{r}} \rightarrow \CC^{2m_{r}\times 2m_{r}}$, given by the completely positive map
\begin{equation}
  \label{eq:192}
  \SuOp^{r,z}[R]
  =
  \sum_{\alpha=1}^{\alpha_*} K_{\alpha}^{r,z} R K_{\alpha}^{r,z} + \sum_{\beta=1}^{\beta_*} \big( L_{\beta}^{r,z} R (L_{\beta}^{r,z})^{*} + (L_{\beta}^{r,z})^* R L_{\beta}^{r,z}\big)
  ,
\end{equation}
where $K_{\alpha}^{r,z}$ and $L_{\beta}^{r,z}$ are the coefficient matrices of $\Lb^{r,z}$ (see, e.g., \eqref{eq:118}).
Note, that if we evaluate $\Lb^{r,z}$ on the tuple of random matrices $(\Xb,\Yb,\Yb^*)$, then $\Lb^{r,z}(\Xb,\Yb,\Yb^*)$ belongs to the class of \emph{Kronecker} random matrices, which were studied in \cite{AltErdoKrugNemi_Kronecker}.
Therefore, from \cite[Lemma~2.2]{AltErdoKrugNemi_Kronecker}, we have that the corresponding \emph{Matrix Dyson equation}
\begin{equation}
  \label{eq:193}
  -\frac{1}{\Sol^{r,z}_{\omega}}
  =
  \omega I_{2m_{r}} -  K^{r,z}_{0} + \SuOp^{r,z}[\Sol^{r,z}_{\omega}]
  ,\quad
  z\in \CC
  ,\quad
  \omega \in \CC_{+}
\end{equation}
has a unique solution with positive semidefinite imaginary part $\Im \Sol^{r,z}_{\omega} \geq 0$.
Moreover, for each $z\in\CC$, the solution matrix $\Sol^{r,z}_{\omega}$ admits the Stieltjes transform representation
\begin{equation}
  \label{eq:224}
  \Sol^{r,z}_{\omega}
  =
  \int_{\RR} \frac{V^{r,z}(d\lambda)}{\lambda-\omega}
  ,
\end{equation}
where $\{V^{r,z}\}_{z\in \CC}$ is a family of measures taking values in the set of positive definite matrices.
In the limit $N \to \infty$ the solution $  \Sol^{r,z}_{\omega}\otimes I_N$ well approximates $(\Lb^{r,z}(\Xb,\Yb,\Yb^*)-\omega I)^{-1}$ in the entrywise maximum norm (see \cite[Lemma~B.1]{AltErdoKrugNemi_Kronecker}).
\begin{rem} \label{rem:knoneckerReduction}
  The statement of \cite[Lemma~2.2]{AltErdoKrugNemi_Kronecker} is formulated in the more general setting of \emph{Wigner-type} matrices allowing independent but not necessarily identically distributed entries.
  This model, in general, leads to a system of $N$ matrix equations.
In our case, the matrices in $\Xb$ and $\Yb$ have \emph{i.i.d.} entries (up to symmetry constraints), which reduces the system of $N$ possibly different matrix equations (see, e.g., \cite[Eq.~(2.6)]{AltErdoKrugNemi_Kronecker}) to $N$ identical matrix equations of the form \eqref{eq:193}.
\end{rem}
For any rational expression $r$ with linearization $\Lb_r$ define the set $\DD_{\varepsilon}^{\Lb_{r}}\subset \CC$ by 
    \begin{equation}
    \label{eq:184}
    \DD_{\varepsilon}^{\Lb_{r}}
    :=
    \Big\{z \, : \, \mathrm{dist}(0,\mathrm{supp}\, \rho^{r,z}) \leq \varepsilon \Big\}
    ,
  \end{equation}
  where  $\rho^{r,z}(d\lambda):=\frac{1}{2m_r} \Tr V^{r,z}(d\lambda)$ and the family of measures $V^{r,z}(d\lambda)$ were defined in \eqref{eq:224}.
  The set $\DD_{\varepsilon}^{\Lb_{r}}$ is called the \emph{self-consistent} $\varepsilon$-pseudospectrum of $r$ related to its linearization  $\Lb_{r}$, and $\rho^{r,z}$ is called the \emph{self-consistent density of states} of $\Lb^{r,z}(\Xb,\Yb,\Yb^*)$.

The next lemma contains the main result of this section.
\begin{lem}[Convergence of the (pseudo)spectrum] \label{lem:pseudospectrum}
  Suppose that $q\in\Qc_{\qb_{0},\ldots,\qb_{n}}$ is a (not necessarily self-adjoint) rational expression of height $n$, and  $(\sbm, \cb)\in \Dc_{\qb_{0},\ldots,\qb_{n};C}(\Sc)$  with some constant $C>0$.
  Then there exists $\cstL>0$ such that 
  \begin{equation}
    \label{eq:187}
    \Big\|\big(\widehat{\Lb}_{q}(\Xb,\Yb,\Yb^*)\big)^{-1}\big\|_{\CC^{(m-1)\times (m-1)}\otimes \,\CC^{N\times N}}
    \leq
    \cstL
    \quad
    \mbox{a.w.o.p.},
  \end{equation}
  where $\widehat{\Lb}_{q}$ is defined as in \eqref{eq:117}.
  There exists also a constant $\widetilde{C}>0$ depending only on the linearization $\Lb_q$ and the constant $C$ such that 
 \begin{equation}
    \label{eq:domain of RM}
   (\Xb,\Yb) \in  \Dc_{\qb_{0},\ldots,\qb_{n};\widetilde{C}}\big(\CC^{N \times N}\big)\qquad a.w.o.p.\,.
 \end{equation}
 Moreover, for any $\varepsilon \in (0,1)$ the $\varepsilon$-pseudospectrum of $q(\Xb,\Yb,\Yb^*)$ satisfies 
  \begin{equation}
    \label{eq:181}
    \mathrm{Spec}_{\varepsilon}(q(\Xb,\Yb,\Yb^*))
    \subset
    \DD_{2\varepsilon}^{\Lb_{q}}
    \quad
    \mbox{a.w.o.p.}
    .
  \end{equation}
\end{lem}
\begin{proof}
  We split the proof of this lemma in two parts.
  First we show that the condition \eqref{eq:187} together with \eqref{eq:domain of RM} implies \eqref{eq:181} for \emph{any}, not necessarily self-adjoint, rational expression and its linearization.
  After that we prove that \eqref{eq:187} and \eqref{eq:domain of RM} are satisfied for $q$ if $(\sbm, \cb)\in \Dc_{\qb_{0},\ldots,\qb_{n};C}(\Sc)$ using induction on the height $n$.

  Suppose that we have an arbitrary rational expression $r$ and its linearization $\Lb_{r}$ of size $m_{r}$, and suppose that there exists $\cstM>0$ such that \emph{a.w.o.p.}
  \begin{equation}
    \label{eq:215}
    \|(\widehat{\Lb}_{r}(\Xb,\Yb,\Yb^*))^{-1}\|_{\CC^{(m_{r}-1)\times (m_{r}-1)}\otimes \,\CC^{N \times N}}\leq \cstM
    ,
  \end{equation}
  where $\widehat{\Lb}_{r}$ is defined similarly as in \eqref{eq:117}.
  Then from the definition of the linearization (Definition~\ref{def:linearization}) and the Schur complement formula \eqref{eq:132}, we can choose $C_{3},C_{4}>0$ such that 
    \begin{equation}
    \label{eq:207}
    \| r(\Xb,\Yb,\Yb^*)  \|_{\CC^{N\times N}}
    \leq
    C_{3}
  \end{equation}
and  the sequence of inequalities
  \begin{align}
    \label{eq:208}
    \| (r(\Xb,\Yb,\Yb^*) - z I_{N})^{-1} \|_{\CC^{N \times N}}
     &\leq
      \|(\Lb_{r}(\Xb,\Yb,\Yb^*)-zJ_{m_{r}}\otimes I_{N})^{-1}\|_{\CC^{m_{r}N\times m_{r}N}}
    \\ \label{eq:212}
    & \leq
     m_{r} \|(\Lb_{r}(\Xb,\Yb,\Yb^*)-zJ_{m_{r}}\otimes I_{N})^{-1}\|_{\CC^{m_{r}\times m_{r}}\otimes \,\CC^{N\times N}}
      \\
    &
      \leq
    C_{4}\| (r(\Xb,\Yb,\Yb^*) - z I_{N})^{-1} \|_{\CC^{N\times N}}
  \end{align}
  hold \emph{a.w.o.p.} for all $z\in\CC$.
  Here the first and third inequalities follow from the Schur complement formula \eqref{eq:132} and the norm bounds $\max_{1\leq \alpha \leq \alpha_{*}}\|X_{\alpha}\|_{\CC^{N\times N}} \leq 3$, $\max_{1\leq \beta \leq \beta_{*}}\|Y_{\beta}\|_{\CC^{N\times N}} \leq 3$ holding \emph{a.w.o.p.}, the second inequality holds deterministically for all realizations of $\Xb$ and $\Yb$ (see \eqref{eq:213}), $C_{3}>0$ depends on $\widehat{C}_{r}$, $m_{r}$ and the norms of $X_{\alpha}$, $Y_{\beta}$, which we bound by $3$, and $C_{4}>1$ additionally depends on $C_{3}$.

    Note again from \eqref{eq:132} that if $\widehat{\Lb}_{r}(\Xb,\Yb,\Yb^*)$ is invertible, then
     \begin{equation}
     \label{eq:209}
     \mathrm{Spec}(r(\Xb,\Yb,\Yb^*))
     =
     \{z\in\CC \, : \Lb_{r}(\Xb,\Yb,\Yb^*)-zJ_{m_{r}}\otimes I_N \mbox{ is not invertible}\}
     .
   \end{equation}
   On the other hand, using the definition of $\Lb^{r,z}$ from \eqref{eq:191}, the set on the right-hand side of \eqref{eq:209} can be described via the spectrum of $\Lb^{r,z}(\Xb,\Yb,\Yb^*)$ as
   \begin{equation}
     \label{eq:210}
     \{z\in\CC \, : \Lb_{r}(\Xb,\Yb,\Yb^*)-zJ_{m_{r}}\otimes I_N \mbox{ is not invertible}\}
     =
     \{z\in \CC \, : 0 \in \mathrm{Spec}(\Lb^{r,z}(\Xb,\Yb,\Yb^*))\}
   \end{equation}
   with the identity \eqref{eq:210} holding deterministically for all realizations of $\Xb$ and $\Yb$.

   Under the condition \eqref{eq:215},
    the equality  \eqref{eq:209} can be rewritten in terms of the pseudospectrum using \eqref{eq:208} as 
      \begin{align}
     \label{eq:203}
     \mathrm{Spec}_{\varepsilon}(r(\Xb,\Yb,\Yb^*))
     & \subset
     \Big\{z\in \CC : \| ( \Lb_{r}(\Xb,\Yb,\Yb^*) - zJ_{m_{r}} \otimes I_{N})^{-1}\|_{\CC^{m_{r}N\times m_{r}N}} \geq \frac{1}{\varepsilon} \Big\}
     \\ & \subset
     \mathrm{Spec}_{C_{4} \varepsilon}(r(\Xb,\Yb,\Yb^*))
   \end{align}
   holding \emph{a.w.o.p.}
   At the same time, from the definition \eqref{eq:191} we have that the set of the singular values of the Hermitian matrix $\Lb^{r,z}(\Xb,\Yb,\Yb^*)$
   coincides with the set of the singular values of $ \Lb_{r}(\Xb,\Yb,\Yb^*) - zJ_{m_{r}} \otimes I_{N}$, 
   so that 
   \begin{multline}
     \label{eq:201}
     \Big\{z\in \CC \, : \| ( \Lb_{r}(\Xb,\Yb,\Yb^*) - zJ_{m_{r}} \otimes I_{N})^{-1}\|_{\CC^{m_{r} N\times m_{r} N}} \geq \frac{1}{\varepsilon} \Big\}
     \\
     =
     \{ z\in \CC \, : \mathrm{dist}(0, \mathrm{Spec}(\Lb^{r,z}(\Xb,\Yb,\Yb^*))\leq \varepsilon \}
     .
   \end{multline}

    In order to study the spectrum of $\Lb^{r,z}(\Xb,\Yb,\Yb^*)$, we  will exploit the fact that $\Lb_{r}(\Xb,\Yb,\Yb^*)-zJ_{m_{r}}\otimes I_{N}$ belongs to the class of \emph{Kronecker} random matrices and thus we can use the results from \cite{AltErdoKrugNemi_Kronecker} about the location of spectrum for this class of random matrix ensembles.
    By applying  part (i) of \cite[Theorem~4.7]{AltErdoKrugNemi_Kronecker} to $\Lb^{r,z}$, we have (similarly as in the proof of  \cite[Lemma~6.1]{AltErdoKrugNemi_Kronecker} for bounded $\zeta$) that for any $z$ satisfying $\mathrm{dist}\,(0,\mathrm{supp}\,\rho^{r,z}) \geq 2\varepsilon$, \emph{a.w.o.p.} 
  \begin{equation}
    \label{eq:206}
    \mathrm{Spec}(\Lb^{r,z}(\Xb,\Yb,\Yb^*))
    \cap
    \Big[-\frac{3\varepsilon}{2} ,\frac{3\varepsilon}{2}\Big]
    =
    \emptyset
    .
  \end{equation}
  Now we claim that \eqref{eq:206} holds simultaneously \emph{a.w.o.p.} for all $\{z \, : |z|\leq 2C_{3},\, z\notin \DD^{\Lb_{r}}_{2\varepsilon} \}$.
  To prove this strengthening, we apply the standard grid argument (again analogously as in the proof of  \cite[Lemma~6.1]{AltErdoKrugNemi_Kronecker}) together with the Lipschitz continuity of the eigenvalues of $\Lb^{r,z}(\Xb,\Yb,\Yb^*)$ in $z$. 
  Together with \eqref{eq:203} and \eqref{eq:201} this implies that \emph{a.w.o.p.}
  \begin{equation}
    \label{eq:216}
    \mathrm{Spec}_{\varepsilon}(r(\Xb,\Yb,\Yb^*))
    \subset
    \{ z\in \CC \, : \mathrm{dist}(0, \mathrm{Spec}(\Lb^{r,z}(\Xb,\Yb,\Yb^*))\leq \varepsilon \}
    \subset
    \Big(\DD^{\Lb_{r}}_{2\varepsilon} \cup \{z \, : |z|\geq 2C_{3} \} \Big)
    .
  \end{equation}
  We conclude from \eqref{eq:207} that \emph{a.w.o.p.}
  \begin{equation}
    \label{eq:226}
    \mathrm{Spec}_{\varepsilon}(r(\Xb,\Yb,\Yb^*))
    \subset
    \DD^{\Lb_{r}}_{2\varepsilon}
    . 
  \end{equation}
  This finishes the first part of the proof by establishing that for any rational expression the conditions \eqref{eq:187} and \eqref{eq:domain of RM} imply \eqref{eq:181}.

  In order to prove \eqref{eq:187} and \eqref{eq:domain of RM}, we proceed with a proof by induction on the height of $q$.
  If $q$ has height $0$ (i.e., if $q$ is a polynomial), then \eqref{eq:domain of RM} is trivially true and  \eqref{eq:187} holds by nilpotency \cite[Lemma~2.5]{ErdoKrugNemi_Poly} and the norm bounds
  \begin{equation}
    \label{eq:195}
    \|X_{\alpha}\|_{\CC^{N\times N}}
    \leq
    3
    ,\quad
    \|Y_{\beta}\|_{\CC^{N\times N}}
    \leq
    3
    \quad
    \mbox{a.w.o.p.}
  \end{equation}
  
  Suppose that the statement of the theorem is true for all rational expressions of height $\leq n-1$.
   Together with $(\sbm, \cb)\in \Dc_{\qb_{0},\ldots,\qb_{n};C}(\Sc)$ this, in particular, means that \emph{a.w.o.p.}
  \begin{equation}
    \label{eq:71}
    (\Xb,\Yb) \in  \Dc_{\qb_{0},\ldots,\qb_{n-1};\widetilde{C}_1}\big(\CC^{N \times N}\big)
  \end{equation}
  for some $\widetilde{C}_1>0$, and
  \begin{equation}
    \label{eq:259}
        \|(\widehat{\Lb}_{r}(\Xb,\Yb,\Yb^*))^{-1}\|_{\CC^{(m_r-1)\times (m_r-1)}\otimes \,\CC^{N\times N}}
    \leq
    \widehat{C}^{\mathrm{w}}_{\Lb_{r}}
  \end{equation}
for all $ r\in \Tc_{q}:= \{q_{i,\gamma_{i}} \,: 0\leq i\leq n, 1\leq \gamma_{i} \leq \ell_{i}\}$ with $q_{i,\gamma_{i}}$ as in the definition of $\Qc_{\qb_{0},\ldots,\qb_{n}}$. 
We now show that there exists $\widetilde{C}>0$ such that $(\Xb,\Yb) \in  \Dc_{\qb_{0},\ldots,\qb_{n};\widetilde{C}}\big(\CC^{N \times N}\big)$ \emph{a.w.o.p.}, or equivalently, that for all $\gamma_{n} \in \{1,\ldots,\ell_n\}$  \emph{a.w.o.p.}
\begin{equation}
  \label{eq:245}
  \Big\| \frac{1}{q_{n,\gamma_n}(\Xb,\Yb,\Yb^*)}\Big\|_{\CC^{N\times N}}
  \leq
  \widetilde{C}
  .
\end{equation} 

Using the result established in the first part of the proof, \eqref{eq:71} and \eqref{eq:259} with $r = q_{n, \gamma_n}$ imply that \emph{a.w.o.p.}
\begin{equation}
  \label{eq:260}
      \mathrm{Spec}_{\varepsilon}(q_{n,\gamma_n}(\Xb,\Yb,\Yb^*))
    \subset
    \DD^{\Lb_{q_{n,\gamma_n}}}_{2\varepsilon}
    .
\end{equation}
Therefore, it is enough to show that there exists $\widetilde{C}>0$ such that for all $\gamma_n \in \{1,\ldots,\ell_n\}$ the point $z=0$ does not belong to $\DD^{\Lb_{q_{n,\gamma_n}}}_{2/\widetilde{C}}$.
By the definition \eqref{eq:184}, the last condition is equivalent to
\begin{equation}
  \label{eq:262}
  \mathrm{dist}(0,\mathrm{supp}\, \rho^{q_{n,\gamma_n},0}) \geq \frac{2}{\widetilde{C}}
  .
\end{equation}
We now show that \eqref{eq:262} holds for one fixed $\gamma_n$.
The desired bound \eqref{eq:245} can then be obtained by taking the maximum over $\widetilde{C}$ for all $\gamma_n\in \{1,\ldots, \ell_n\}$.

Fix $\gamma_{n}\in \{1,\ldots, \ell_n\}$, and denote by $\Lb_{q_{n,\gamma_n}}(\sbm,\cb,\cb^*) \in \CC^{m\times m}\otimes \Sc$ the linearization of $q_{n,\gamma_n}$ evaluated at $(\sbm,\cb,\cb^*)$.
By examining carefully the derivation of the  a priori
   bound \eqref{eq:130}, we observe that 
  there exists a constant $\cstO>0$ depending only on the linearization and the constant $C$, such that 
  \begin{equation}
    \label{eq:253}
    \Big\|\Big(\Lb_{q_{n,\gamma_n}}(\sbm,\cb,\cb^*)\Big)^{-1}\Big\|_{\CC^{m \times m} \otimes \, \Sc}
    \leq
    \cstO
    .
  \end{equation}
  Indeed, each entry of $(\Lb_{q_{n,\gamma_n}}(\sbm,\cb,\cb^*))^{-1}$ is a polynomial in $\sbm$, $\cb$, $\cb^*$ and  $(r(\sbm,\cb,\cb^*))^{-1}$ with $r\in \Tc_{q}$.
  Therefore, the bounds $\|(r(\sbm,\cb,\cb^*))^{-1}\|_{\Sc}\leq C$ that follow from the assumption $(\sbm,\cb)\in \Dc_{\qb_{0},\ldots,\qb_{n};C}$ imply that \eqref{eq:253} holds for some $\cstO>0$.

  Next notice, that from the definition of $\Lb^{r,z}$ (see \eqref{eq:191}), we have that
  \begin{equation}
    \label{eq:227}
    \Big\|\Big(\Lb^{q_{n,\gamma_n},0 }(\sbm,\cb,\cb^*)\Big)^{-1}\Big\|_{\CC^{2m \times 2m} \otimes \, \Sc}
    =
    \Big\|\Big(\Lb_{q_{n,\gamma_n}}(\sbm,\cb,\cb^*)\Big)^{-1}\Big\|_{\CC^{m \times m} \otimes \, \Sc}
    .
  \end{equation}
  The resolvent identity
  \begin{equation}
    \label{eq:221}
    \Big(\Lb^{q_{n,\gamma_n},z }(\sbm,\cb,\cb^*)- \omega I \otimes \scone \Big)^{-1} \Big(I\otimes \scone - \omega \, \Big(\Lb^{q_{n,\gamma_n},z }(\sbm,\cb,\cb^*)\Big)^{-1} \Big)
      =
    \Big(\Lb^{q_{n,\gamma_n},z }(\sbm,\cb,\cb^*)\Big)^{-1}
  \end{equation}
  together with \eqref{eq:253} and \eqref{eq:227} implies that
  \begin{equation}
    \label{eq:222}
    \Big\|\Big(\Lb^{q_{n,\gamma_n},0 }(\sbm,\cb,\cb^*)-\omega I\otimes \scone \Big)^{-1} \Big\|_{\CC^{2m \times 2m} \otimes \, \Sc}
    \leq
     2\cstO
  \end{equation}
  for all $\omega \in \CC$  satisfying $|\omega | \leq 1/(4m  t)$, where we used the relation between the operator and max norms as in \eqref{eq:212} to estimate the operator norm of $(\Lb^{q_{n,\gamma_n},z }(\sbm,\cb,\cb^*))^{-1}$. 
  The bound \eqref{eq:222} means, in particular, that
  \begin{equation}
    \label{eq:254}
    \mathrm{Spec}\Big(\Lb^{q_{n,\gamma_n},0 }(\sbm,\cb,\cb^*)\Big) \cap \Big[-\frac{1}{4mt}, \frac{1}{4mt}\Big]
    =
    \emptyset
    .
  \end{equation}
  The spectrum of the self-adjoint operator $\Lb^{q_{n,\gamma_n},0 }(\sbm,\cb,\cb^*)$ is characterized by the solution to the self-consistent equation \ref{eq:193} (see e.g. \cite[Theorem 4.1.12]{Spei98}) via 
  \begin{equation}
    \label{eq:255}
    \mathrm{Spec}\Big(\Lb^{q_{n,\gamma_n},0 }(\sbm,\cb,\cb^*) \Big)
    =
    \mathrm{supp} \, \rho^{q_{n,\gamma_n},0 }
    ,
  \end{equation}
  which establishes \eqref{eq:262} with $\widetilde{C} = 8 m t$.

  Finally, \eqref{eq:domain of RM} together with the Schur complement formula, \eqref{eq:127}-\eqref{eq:128}, and the special form of the linearization blocks \eqref{eq:121}-\eqref{eq:122}, yields \eqref{eq:187} for $q\in\Qc_{\qb_{0},\ldots,\qb_{n}}$.
  This finishes the proof of the lemma.
  \end{proof}
  
  The bound \eqref{eq:187} together with the Schur complement formula \eqref{eq:132} applied to $(\Lb_{q}(\Xb,\Yb,\Yb^*)-zJ_{m}\otimes I_{N})^{-1}$, the trivial bound \eqref{eq:133} applied to $q(\Xb,\Yb,\Yb^*)$ and the norm bounds $\|X_{\alpha}\|_{\CC^{N\times N}}\leq 3$ and $\|Y_{\beta}\|_{\CC^{N\times N}}\leq 3$ holding (similarly as in \eqref{eq:11}-\eqref{eq:12}) \emph{a.w.o.p.}, imply the trivial bound for the generalized resolvent in random matrices.
  \begin{cor} \label{cor:trivBndG}
    There exists $\cstR>0$ depending only on the  linearization $\Lb_{q}$, such that \emph{a.w.o.p.}
    \begin{equation}
    \label{eq:180}
    \|(\Lb_{q}(\Xb,\Yb,\Yb^*)-zJ_{m}\otimes I_{N})^{-1} \|_{\CC^{m\times m} \otimes \CC^{N\times N}}
    \leq \cstR \Big(1+\frac{1}{\Im z} \Big)
    .
  \end{equation}
\end{cor}

\subsection{ Global and local laws  for  rational expressions in random matrices}
\label{sec:proof-local-law}
Denote, as before, by $\Sc$ a $C^*$-algebra containing a freely independent family $\{\semic_1,\ldots,\semic_{\alpha_*}, c_{1},\ldots,c_{\beta_*}\}$ of $\alpha_*$ semicircular and $\beta_*$ circular elements in a NC probability space $(\Sc,\tau_{\Sc})$.
  Let $q\in \Qc_{\qb_{0},\ldots,\qb_{n}}$ be a rational expression of height $n$ and assume that $(\sbm,\cb)\in \Dc_{\qb_{0},\ldots,\qb_{n};C}(\Sc)$ for $\sbm = (\semic_{1},\ldots,\semic_{\alpha_*})$, $\cb:=(c_1,\ldots,c_{\beta_*})$ and some $C>0$.
Let
\begin{equation}
  \label{eq:153}
  \Lb
  =
  \Lb(\xb,\yb,\yb^*)
  :=
  K_0\otimes \cstone - \sum_{\alpha=1}^{\alpha_*} K_{\alpha} \otimes x_{\alpha} - \sum_{\beta=1}^{\beta_*}\big( L_{\beta} \otimes y_{\beta} + L_{\beta}^* \otimes y_{\beta}^*\big)
\end{equation}
be the linearization of $q$ constructed via the algorithm from Section~\ref{sec:line-algor}.

In order to formulate the local law, we need to introduce the notion of the \emph{stochastic domination}.
\begin{defn}[Stochastic domination] \label{defn:stochastic-domination}
  Let $\Phi :=  (\Phi_{N})_{N\in \NN}$ and $ \Psi :=  (\Psi_{N})_{N\in \NN}$ be two sequences of nonnegative random variables.
  We say that $\Phi$ is \emph{stochastically dominated} by $\Psi$ (denoted $\Phi \prec \Psi$), if for any $\varepsilon, D > 0$ there exists $C(\varepsilon,D)>0$ such that for all $N\in \NN$
  \begin{equation}
    \label{eq:7}
    \PP[\Phi_{N} \geq N^{\varepsilon} \Psi_{N}]
    \leq
    \frac{C(\varepsilon,D)}{N^{D}}
    .
  \end{equation}
\end{defn}
Let $q(\Xb,\Yb,\Yb^*)\in \CC^{N\times N}$ be the evaluation of $q$ on the $\alpha_*$-tuple of Wigner and $\beta_*$-tuple of \emph{iid} random matrices satisfying (\textrm{\textbf{H1}})-(\textrm{\textbf{H3}}), and define the linearization matrix
\begin{equation}
  \label{eq:155}
  \Hb
  :=
  \Lb(\Xb,\Yb,\Yb^*)
  =
  K_0\otimes I_{N} - \sum_{\alpha=1}^{\alpha_*} K_{\alpha} \otimes X_{\alpha} - \sum_{\beta=1}^{\beta_*}\big( L_{\beta} \otimes Y_{\beta} + L_{\beta}^* \otimes Y_{\beta}^*\big)
  .
\end{equation}
Let $\GRes_{z}: = (\Hb - zJ\otimes I_N)^{-1} \in \CC^{mN \times mN}$ be the \emph{generalized} resolvent of $\Hb$.
Note that the generalized resolvent $\GRes_{z}$, when viewed as taking values in $\CC^{m\times m} \otimes \CC^{N \times N}$, can be written as
\begin{equation}
  \label{eq:156}
  \GRes_{z} = \sum_{i,j=1}^{N} G_{z, ij} \otimes E_{ij},
\end{equation}
where  the collection of
 matrices $E_{ij} := (\delta_{ki} \delta_{jl})_{1\leq k,l \leq N}$ form a standard basis of $\CC^{N \times N}$
 and  $G_{z,ij} \in \CC^{m \times m}$ is an $m\times m$ matrix for each $(i,j)$ pair.  In general, we will follow
 the convention that for any $\Ab \in \CC^{m\times m} \otimes \CC^{N \times N}$ we denote by $ \Ab_{kl}  \in \CC^{N \times N}$,
 $k,l=1,2,\ldots, m$ the $(k,l)$-th block according to the $\CC^{m\times m} $ factor in the tensor product, while
 $A_{ij} \in \CC^{m\times m}$,  $i,j =1,2,\ldots , N$ denotes the $(i,j)$-th block in the second factor, i.e.
 \begin{equation}\label{notation}
      \Ab =  \sum_{k,l=1}^{m} E_{kl} \otimes \Ab_{kl}=\sum_{i,j=1}^{N} A_{ij} \otimes E_{ij},
 \end{equation}
in particular $A_{ij}(k,l) = \Ab_{kl}(i,j)$.

Let $\Sol^{(\mathrm{sc})}_{z} : \CC_{+} \rightarrow \CC^{m \times m}$ be a matrix valued function
given by~\eqref{eq:137}.
For each
$z\in\CC_{+}$ we define the \emph{stability operator} $\StOp_{z}: \CC^{m\times m} \rightarrow \CC^{m \times m}$ corresponding to $\Sol^{(\mathrm{sc})}_{z} $ by
\begin{equation}
  \label{eq:157}
  \StOp_{z}\big[R\,\big]
  =
  R - \Sol^{(\mathrm{sc})}_{z}  \SuOp \big[R\,\big] \Sol^{(\mathrm{sc})}_{z} 
  ,\quad
  R \in \CC^{m \times m}
  .
\end{equation}
For $n \in \NN$ and an operator $\Rcc : \CC^{n\times n} \rightarrow \CC^{n \times n}$ we denote by $\| \Rcc \|_{\CC^{n \times n} \rightarrow \CC^{n \times n}}$ the operator norm of $\Rcc$ generated by the the operator norm on $\CC^{n\times n}$.
Then the following holds.
\begin{thm}[Local law for rational expressions]
  \label{thm:locallawL}
Let $\Sol^{(\mathrm{sc})}_{z}$ be defined as in \eqref{eq:137} and let $\StOp_{z}$ be the stability operator corresponding to $\Sol^{(\mathrm{sc})}_{z} $.
  If there exist $C_{0}>0$ and $\Ic \subset \RR$ such that for all $z$ with $\Re z \in \Ic $  and $0\leq \Im z < \infty$
  \begin{itemize}
  \item[(\textrm{\textbf{M1}})] $\|\Sol^{(\mathrm{sc})}_{z}\|_{\CC^{m \times m} } \leq C_{0}$;
  \item[(\textrm{\textbf{M2}})] $\| \StOp_{z}^{-1} \|_{\CC^{m \times m} \rightarrow \CC^{m \times m} } \leq C_{0}$,
  \end{itemize}
  then the optimal local law holds for $\GRes(z)$ on the set $\Ic$, i.e., uniformly for $\Re z \in \Ic $
  \begin{equation}
    \label{eq:158}
    \max_{1\leq i,j \leq N} \Big\| G_{z,ij}- \Sol^{(\mathrm{sc})}_{z} \delta_{ij} \Big\|_{\CC^{m\times m}}
    \prec
    \sqrt{\frac{1}{N\Im z}}
    ,\quad
    \Big\|\frac{1}{N} \sum_{i=1}^{N} G_{z,ii} -  \Sol^{(\mathrm{sc})}_{z} \Big\|_{\CC^{m\times m} }
    \prec
    \frac{1}{N\Im z}
    .
  \end{equation}
  In particular, this implies that on the set $\Ic $ the optimal local law holds for the rational expression in random matrices $q(\Xb,\Yb,\Yb^*)$, i.e., uniformly for $\Re z \in \Ic $
  \begin{equation}
    \label{eq:159}
    \max_{1 \leq i,j\leq N}\Big|g_{z,ij} - \Sol^{(\mathrm{sc})}_{z}(1,1) \delta_{ij} \Big|
    \prec
    \sqrt{\frac{1}{N\Im z}}
    ,\quad
    \Big|\frac{1}{N}\sum_{i=1}^{N} g_{z,ii} -\Sol^{(\mathrm{sc})}_{z}(1,1) \Big|
    \prec
    \frac{1}{N\Im z}
    ,
  \end{equation}
  where $\gb_{z}=(g_{z,ij})_{i,j=1}^{N} :=(q(\Xb,\Yb,\Yb^*)-zI_N)^{-1}\in \CC^{N\times N}$.
\end{thm}
Note, that by the definition of the generalized resolvent and \eqref{eq:132} and the notational convention~\eqref{notation},
 we have that  $g_{z,ij} = G_{z,ij}(1,1)$ for all $1\leq i,j \leq N$.

\begin{proof}
  Our proof of the local law for linearizations of rational expressions \eqref{eq:158} is analogous to the proof of the corresponding result for linearizations of \emph{polynomials} in Wigner and \emph{iid} matrices \cite[Theorem~5.1]{ErdoKrugNemi_Poly}.
  Below we provide a summary of the important steps of that proof and show how these steps are adjusted to the current setting of rational expressions.
   \medskip
  \\ 
  \noindent
  \textit{1. Restricting analysis to the set where $q(\Xb,\Yb,\Yb^*)$ is well defined.} In contrast to the case when $q$ is a polynomial, the evaluation $q(\Xb,\Yb,\Yb^*)$ may not always be well defined and the generalized resolvent $\Gb_{z}$ may not always be bounded, even when $z\in \CC_{+}$, i.e. the a priori bound analogous to (2.5) in \cite[Lemma~2.5]{ErdoKrugNemi_Poly} may not hold. But according to Lemma~\ref{lem:pseudospectrum} and  Corollary~\ref{cor:trivBndG} this bound can be replaced by \eqref{eq:180} and the existence of $q(\Xb,\Yb,\Yb^*)$ can be guaranteed on an event $\Theta=\Theta_{N}$ of asymptotically overwhelming probability. The entire analysis is then restricted to this set. In particular, the indicator sets $\chi(\cdot)$ in the proof of \cite[Theorem~5.1]{ErdoKrugNemi_Poly} should be replaced by $\chi(\cdot)\cap \Theta$.
  \medskip
  \\ 
  \noindent
  \textit{2. Exploiting the Kronecker structure of $\Hb$ and regularization of the \emph{DEL}.}
  Similarly as in the polynomial case, any linearization of a rational expression evaluated on Wigner and \emph{iid} ensembles belongs to the class of \emph{Kronecker} random matrices.
  Therefore, in order to obtain the initial estimates on the error term, we use the results of \cite[Lemma~4.4]{AltErdoKrugNemi_Kronecker}. As in the polynomial setup these results require  introducing a small regularization $\omega = i u$ with $u>0$ in order to use the stability theory of the MDE \eqref{eq:193}. The bounds in  \cite[Lemma~5.2]{ErdoKrugNemi_Poly} are uniform in this regularization and are a consequence of the a priori estimate \cite[Lemma~2.5]{ErdoKrugNemi_Poly}. In our current setting they remain true when this a priori estimate is replaced by  \eqref{eq:130}.
  \medskip
  \\ 
  \noindent
  \textit{3. Effective replacement of the Ward identity.}
  Using again the  a priori bound \eqref{eq:130} instead of \cite[Lemma~2.5]{ErdoKrugNemi_Poly}, one can obtain the Ward identity type estimates (see, e.g., \cite[formula (5.13)]{ErdoKrugNemi_Poly}) for the error terms involving the generalized resolvent.
  \medskip
  \\ 
  \noindent
  \textit{4. Finishing the proof.}
  With the above modifications the proof of Theorem~\ref{thm:locallawL} can be obtained by following the proof of \cite[Theorem~5.1]{ErdoKrugNemi_Poly} line by line.
\end{proof}

 In the same spirit we can follow line by line the proof of \cite[Proposition~2.17]{ErdoKrugNemi_Poly}, replacing the use of \cite[Lemma~2.5]{ErdoKrugNemi_Poly} with the trivial bound \eqref{eq:130}, to obtain the following \emph{global law} for (the linearizations of) rational expressions.
 Note that in the proof of Proposition~\ref{pr:glaw} below  we do not have to  assume that the conditions (\textrm{\textbf{M1}})-(\textrm{\textbf{M2}}) hold as stated in Theorem~\ref{thm:locallawL}.
 For the global law the boundedness of $\Sol^{(sc)}_z$ and $\mathscr{L}_z^{-1}$ is needed only for $z$ away from the real axis, in which case it follows from \eqref{eq:130} and the representation \cite[equation (5.22)]{ErdoKrugNemi_Poly} of $\mathscr{L}_z^{-1}$. 
 
\begin{pr}[Global law]
  \label{pr:glaw}
  For any $\cstP>0$, uniformly on $\{\,z\, : \, \Im z \geq \cstP^{-1}, \, |z| \leq \cstP \,\}$
  \begin{equation}
    \label{eq:185}
        \max_{1\leq i,j \leq N} \Big\| G_{z,ij}- \Sol^{(\mathrm{sc})}_{z} \delta_{ij} \Big\|_{\CC^{m\times m}}
    \prec_{\cstP}
    \frac{1}{\sqrt{N}}
    ,\quad
    \Big\|\frac{1}{N} \sum_{i=1}^{N} G_{z,ii} -  \Sol^{(\mathrm{sc})}_{z} \Big\|_{\CC^{m\times m} }
    \prec_{\cstP}
    \frac{1}{N}
    \,,
  \end{equation}
  where $\prec_{\cstP}$ indicates that the constant in the definition of the stochastic domination (see Definition~\ref{defn:stochastic-domination}) may depend on $\cstP$.
  In particular, uniformly on $\{\,z\, : \, \Im z \geq \cstP^{-1}, \, |z| \leq \cstP \,\}$
  \begin{equation}
    \label{eq:160}
    \Big|\frac{1}{N} \Tr \gb_{z} - \Sol^{(\mathrm{sc})}_{z}(1,1) \Big|
    \prec_{\cstP}
    \frac{1}{N}
    \,.
  \end{equation}
\end{pr}

The global law for  the trace of the resolvent of a  rational expression, $\Tr \gb_{z}$,
 has already been proven in~\cite{Yin18}   with a somewhat different method; we comment on this 
 point in Remark~\ref{rmk:yin}. 
\begin{proof}
  Firstly we show that for any rational expression $q$ as defined at the beginning of this section and any $\cstP>0$, the bounds (\textrm{\textbf{M1}})-(\textrm{\textbf{M2}}) from Theorem~\ref{thm:locallawL} hold uniformly on $\{\,z\, : \, \Im z \geq \cstP^{-1}, \, |z| \leq \cstP \,\}$, namely that for any $\cstP>0$ there exist $C_{\cstP}>0$ such that for $\{\,z\, : \, \Im z \geq \cstP^{-1}, \, |z| \leq \cstP \,\}$
  \begin{equation}
    \label{eq:238}
    \|\Sol^{(\mathrm{sc})}_{z}\|_{\CC^{m \times m} } \leq C_{\cstP},
    \qquad
    \| \StOp_{z}^{-1} \|_{\CC^{m \times m} \rightarrow \CC^{m \times m} } \leq C_{\cstP}
    .
  \end{equation}

  The first estimate in \eqref{eq:238} follows directly from \eqref{eq:138} and $(\Im z)^{-1}\leq \cstP$.
  In order to obtain the second estimate, we use the identity
  \begin{equation}
    \label{eq:239}
    \StOp_{z}[\,R\,]
    =
    (\id \otimes \,\tau_{\Sc}) \bigg( \Big(\Lb^{\mathrm{sc}} -z J_m \otimes \scone \Big)^{-1}  \Big( (\Sol_z^{\mathrm{sc}})^{-1}\, R\, (\Sol_{z}^{\mathrm{sc}})^{-1} \otimes \scone \Big) \Big(\Lb^{\mathrm{sc}} -z J_m \otimes \scone \Big)^{-1}\bigg)
  \end{equation}
  for all $\RR\in \CC^{m\times m}$, where $\Lb^{\mathrm{sc}}:=\Lb(\sbm,\cb, \cb^*)$ is the linearization of the rational expression $q$ evaluated on the freely independent semicircular and circular elements (see the proof of \cite[Proposition 2.17]{ErdoKrugNemi_Poly} for the derivation of this identity).
  The boundedness of $(\Sol_{z}^{\mathrm{sc}})^{-1}$ follows from the boundedness of $\Sol_{z}^{\mathrm{sc} }$, the Dyson equation \eqref{eq:136} and $|z|<\cstP$, while the trivial bound in Lemma~\ref{lem:trivialBound} implies the boundedness of $\|(\Lb^{\mathrm{sc}} -z J_m \otimes \scone )^{-1}\|_{\CC^{m\times m} \otimes \Sc}$ holding polynomially in $\cstP$.
  Combining this with \eqref{eq:239} and $(\Im z)^{-1} \leq \cstP$ we obtain the second inequality in \eqref{eq:238}.
  Note that in the above argument we rely on the fact that $(\sbm,\cb)\in\Dc_{\qb_{0},\ldots,\qb_{n};C}(\Sc)$.

  Now we can proceed with the proof \eqref{eq:185} by following the argument used in Theorem~\ref{thm:locallawL} but restricting the analysis to a subset $\{\,z\, : \, \Im z \geq \theta^{-1}, \, |z| \leq \theta \,\}$ bounded away from the real line.
  The concentration inequalities \eqref{eq:185} then follow from \eqref{eq:159} with $(\Im z)^{-1} \leq \cstP$.
  Finally taking the (1,1)-component in the averaged global law for the linearization in \eqref{eq:185} yields \eqref{eq:160}.
\end{proof}

  From \eqref{eq:139} the function $\Sol^{(\mathrm{sc})}_{z}(1,1)$ is a Stieltjes transform of a probability measure $\la e_1, V(d\lambda) e_1 \ra$, which together with Proposition~\ref{pr:glaw} implies that in probability (and almost surely) the empirical spectral measure of $q(\Xb,\Yb,\Yb^*)$ converges weakly to $\la e_1, V(d\lambda) e_1 \ra$ as $N\rightarrow \infty$.
  
\begin{rem}\label{rmk:yin}
In  \cite{Yin18} an induction argument on the height of rational functions (similar as in the proof of Lemma~\ref{lem:pseudospectrum}) was used to show, that the trace  and the norm of a rational expression in GUE (or Wigner) matrices converges almost surely to the trace and norm of the same rational expression in semicircular elements.
  This result implies in particular the global law for self-adjoint rational expressions in Wigner matrices as in \eqref{eq:160}
  (without giving the optimal convergence rate) and the convergence of the (pseudo)spectrum.
  We use a slightly different approach compared to \cite{Yin18}, in particular in the proof of Lemma~\ref{lem:pseudospectrum}, by working with the linearization matrix and the generalized resolvent rather than directly with the rational expressions.
  This allows us to prove  not only the global  law but also the local laws  for the \emph{linearized} models in Section~\ref{sec:proof-local-law}, leading to the global and local laws for the rational expressions \eqref{eq:159} and \eqref{eq:160} with a precise control of the convergence speed.
  This also allows us to link the main outcomes of the present paper with the results about the convergence of the (pseudo)spectrum and the global and local laws established in \cite{AltErdoKrugNemi_Kronecker} and \cite{ErdoKrugNemi_Poly}.   \end{rem}
  
  \begin{rem}
    \label{rem:matrix-valued-ext}
    The results of Sections~\ref{sec:trivial-bound}-~\ref{sec:proof-local-law} can be extended to the setting when the NC variables $x_{1},\ldots,x_{\alpha_*},y_1,\ldots,y_{\beta_*}$ are replaced by the \emph{matrices} of NC variables
    \begin{equation}
      \label{eq:237}
      x_{\alpha}
      =
      \big(x_{\alpha}(i,j)\big)_{1\leq i,j\leq l_{\alpha}}\in \Ac^{l_{\alpha}\times l_{\alpha}}
      , \quad
      y_{\beta}
      =
      \big(y_{\beta}(i,j)\big)_{\substack{1\leq i \leq l_{\beta} \\ 1 \leq j \leq k_{\beta}}}\in \Ac^{l_{\beta}\times k_{\beta}}    
    \end{equation}
    and $q(\xb,\yb,\yb^*)\in \Ac^{l_{q}\times l_{q}}$ with properly chosen dimensions that  make the corresponding matrix operations well defined.
    The linearization of such  model has the same structure \eqref{eq:118} and we can analyze it by considering the matrix entries $x_{\alpha}(i,j)$, $1\leq i,j\leq k_{\alpha}$, and $y_{\beta}(i,j)$, $1\leq i \leq l_{\beta}, 1 \leq j \leq k_{\beta}$, as the new NC variables.
    Note that in this case the diagonal entries of $x_{\alpha}$ are self-adjoint.
    
    In order to relate the spectrum of the rational expression $q(\xb,\yb,\yb^*)$ to its linearization, the generalized resolvent should be replaced by 
    \begin{equation}
      \label{eq:240}
      z\mapsto (\Lb - z \Jb_{l_q})^{-1}
      .
    \end{equation}
    where $\Jb_{l_{q}} = \sum_{i=1}^{l_{q}} E_{ii} \in \CC^{m\times m}$ with $E_{ij}$ being the standard basis of $\CC^{m\times m}$ and $m$ the dimension of the linearization.
    With the above definition the upper-left $l_{q}\times l_{q}$ corner of $(\Lb - z \Jb_{l_q})^{-1}$ is equal to  $(q(\xb,\yb,\yb^*) - I_{l_{q}}\otimes \cstone )^{-1}$, the resolvent of the rational expression $q(\xb,\yb,\yb^*)$.

    We can then proceed with the study of the corresponding random matrix model, with $x_{\alpha}(i,i)$ independent Wigner matrices and all other elements being independent \emph{i.i.d.} matrices. The linearization of this model is again a \emph{Kronecker} random matrix, therefore all the probabilistic estimates in the above analysis remain valid.
    The proof of the main results of Sections~\ref{sec:trivial-bound}-~\ref{sec:proof-local-law} follows the same lines as in the $l_{q}=l_{\alpha} = l_{\beta} = k_{\beta} = 1$ setting, the main changes are notational and are reduced to incorporating the additional structure \eqref{eq:237} and \eqref{eq:240}.
  \end{rem}

\section{Norm bounds for random sample covariance matrices}
\label{sec:norm-bounds-certain}
In this section we prove the norm bounds for random sample covariance ensembles that are used frequently in the paper.
Although various forms of these bounds are well known in the literature (see e.g. \cite{BaiYin93}, \cite{FeldSodi10} or \cite{RudeVers09}), we provide a short alternative proof to make sure that the assumptions about the random matrix ensembles and the probabilistic estimates match the setting of the present paper.

\begin{lem}
\label{lem:norm-bounds-random}  
Let $l,m \in \NN$ and let $W\in\CC^{\,l n\times m n}$ be a (non-Hermitian) random matrix with independent centered entries of variance $1/(ln)$.
Suppose that the entries of $W$ have finite moments of all orders.
Then for any $m<l$ and $\delta>0$, as $n\to \infty$, \emph{a.w.o.p.} 
\begin{equation}
  \label{eq:263}
  \big\| (WW^*)^{-1} \big\|_{\CC^{ln\times ln}}
  \leq
  \frac{1}{\Big(1-\frac{1}{\sqrt{m/l}}\Big)^{2}(1-\delta)}
  .
\end{equation}
\end{lem}
\begin{proof}
  Let $\widehat{W}_{ij}\in \CC^{n\times n}$, $1\leq i \leq l$, $1 \leq j \leq m$, be independent $n\times n$ random matrices having i.i.d. centered entries with variance $1/n$,  so that $W$ can be written as a block matrix
  \begin{equation}
    \label{eq:264}
    W=\Big(\frac{1}{\sqrt{l}}\widehat{W}_{ij}\Big)_{\substack{i = 1\ldots l \\ j=1 \ldots m}}
    .
  \end{equation}
  Denote $Q:=WW^*$.
  For any $z \in \CC$, the linearization of $Q- zI_{ln}$ has a simple structure
  \begin{equation}
    \label{eq:265}
    \Lb_{Q,z}
    =
    \left(
      \begin{array}{cc}
        -zI_{ln} & W
        \\
        W^* & -I_{mn}
      \end{array}
    \right)
    ,
  \end{equation}
  and using the representation \eqref{eq:264} can be viewed as a Kronecker matrix in i.i.d. matrices $\widehat{W}_{ij}$.
  Denote also
  \begin{equation}
    \label{eq:266}
    \Lb_{Q,z}(\cb,\cb^*)
    :=
    \left(
      \begin{array}{cc}
        -zI_{l}\otimes \scone & W^{\mathrm{(c)}}
        \\
        (W^{\mathrm{(c)}})^* & -I_{m}\otimes \scone
      \end{array}
    \right)
    ,
  \end{equation}
where $W^{\mathrm{(c)}}:=(\frac{1}{\sqrt{l}}c_{ij})_{\substack{i=1\ldots l \\ j=1 \ldots m}}$, and $\{c_{ij}\}$ is a family of freely independent circular elements.
By using the properties of the Kronecker random matrices from \cite[Theorems 2.4 and 6.1]{AltErdoKrugNemi_Kronecker}, we get that for any $\varepsilon \in (0,1)$ \emph{a.w.o.p.}
\begin{equation}
  \label{eq:270}
  \mathrm{Spec}\,(W W^*)
  \subset
  \big\{ z \, : \, \mathrm{dist}\big(0,\mathrm{supp}\, \rho^{z}\big) \leq \varepsilon\big\}
  ,
\end{equation}
where $\rho^{z}$ is the self-consistent density of states satisfying $  \mathrm{supp}\, \rho^{z} = \mathrm{Spec}\,(\Lb^{Q,z}_{0}(\cb,\cb^*))$ and 
\begin{equation}
  \label{eq:271}
\Lb^{Q,z}_{\omega}(\cb,\cb^*)
  =
  \left(
    \begin{array}{cc}
      -\omega I_{l+m}\otimes \scone & \Lb_{Q,z}(\cb,\cb^*)
      \\
      \big(\Lb_{Q,z}(\cb,\cb^*)\big)^{*} & -\omega I_{l+m}\otimes \scone 
    \end{array}
  \right)
  .
\end{equation}
Inclusion \eqref{eq:270} can be obtained by repeating the argument from Lemma~\ref{lem:pseudospectrum} leading to \eqref{eq:216}, and taking into account the $(l+m)\times (l+m)$ decomposition in \eqref{eq:265} and \eqref{eq:266}.

Fix $\delta\in (0,1)$. Then we claim that there exists a sufficiently small $\varepsilon >0$ such that
\begin{equation}
  \label{eq:272}
  \big\{ z \, : \, \mathrm{dist}\big(0,\mathrm{supp}\, \rho^{z}\big) \leq \varepsilon\big\}
  \cap
  \Big\{z \, : \, \Re z < \bigg(1-\frac{1}{\sqrt{m/l}} \bigg)^{2}(1-\delta) \Big\}
  =
  \emptyset
  .
\end{equation}
Indeed, $W^{\mathrm{(c)}} (W^{\mathrm{(c)}})^*$ follows the free Poisson distribution with rate $l/m > 1$ (equivalent to Marchenko-Pastur distribution with parameter $m/l<1$) with
  \begin{equation}
    \label{eq:269}
    \mathrm{Spec}\, (W^{\mathrm{(c)}} (W^{\mathrm{(c)}})^*)
    =
    \Big[\Big(1-\sqrt{m/l}\Big)^2,\Big(1+\sqrt{m/l}\Big)^2\Big]
    .
  \end{equation}
  There exists $C>0$ such that for any $z\in \CC$ with $\Re z< (1-\sqrt{m/l})^2(1-\delta)$
  \begin{equation}
    \label{eq:273}
    \Big\| \Big( W^{\mathrm{(c)}} (W^{\mathrm{(c)}})^* - z I_l\otimes \scone \Big)^{-1}\Big\|_{\CC^{l\times l}\otimes \, \Sc}
    \leq
    C
    ,
  \end{equation}
  and thus, by the Schur complement formula, there exists $\widetilde{C}>0$ depending only on $C$, such that
  \begin{equation}
    \label{eq:274}
    \big\| \big( \Lb_{Q,z}(\cb,\cb^*) \big)^{-1}\big\|_{\CC^{(l+m)\times (l+m)}\otimes \, \Sc}
    \leq
    \widetilde{C}
    .
  \end{equation}
  From \eqref{eq:271} and \eqref{eq:274} we get that the resolvent of the Hermitized matrix \eqref{eq:271} is bounded at $\omega=0$
  \begin{equation}
    \label{eq:275}
    \big\| \big( \Lb^{Q,z}_0(\cb,\cb^*) \big)^{-1}\big\|_{\CC^{2(l+m)\times 2(l+m)}\otimes \, \Sc}
    =
    \big\| \big( \Lb_{Q,z}(\cb,\cb^*) \big)^{-1}\big\|_{\CC^{(l+m)\times (l+m)}\otimes \, \Sc}
    \leq
    \widetilde{C}
  \end{equation}
  uniformly for $\Re z<  (1-\sqrt{m/l})^2(1-\delta)$, from which \eqref{eq:272} can be obtained using the resolvent identity (similarly as in \eqref{eq:221}-\eqref{eq:254}).

  Finally, \eqref{eq:270} and \eqref{eq:272} imply that for any $\delta \in (0,1)$ \emph{a.w.o.p.} the smallest eigenvalue of the positive definite matrix $WW^*$ is greater than $(1-1/\sqrt{m/l})^2(1-\delta)$, from which \eqref{eq:263} follows.
\end{proof}

\bibliographystyle{abbrv}
\bibliography{./bib.bib}

\begin{thebibliography}{10}

\bibitem{AjanErdoKrug19}
O.~H. Ajanki, L.~Erd\H{o}s, and T.~Kr\"{u}ger.
\newblock Stability of the matrix {D}yson equation and random matrices with
  correlations.
\newblock {\em Probab. Theory Related Fields}, 173(1-2):293--373, 2019.

\bibitem{AkemBaikDiFrBook}
G.~{Akemann}, J.~{Baik}, and P.~{Di Francesco}, editors.
\newblock {\em {The Oxford handbook of random matrix theory.}}
\newblock Oxford: Oxford University Press, 2011.

\bibitem{AltErdoKrugNemi_Kronecker}
J.~Alt, L.~Erd\H{o}s, T.~Kr\"{u}ger, and {\relax Yu}.~Nemish.
\newblock Location of the spectrum of {K}ronecker random matrices.
\newblock {\em Ann. Inst. Henri Poincar\'{e} Probab. Stat.}, 55(2):661--696,
  2019.

\bibitem{AndeGuioZeitBook}
G.~W. Anderson, A.~Guionnet, and O.~Zeitouni.
\newblock {\em An introduction to random matrices}, volume 118 of {\em
  Cambridge Studies in Advanced Mathematics}.
\newblock Cambridge University Press, Cambridge, 2010.

\bibitem{BaiSilvBook}
Z.~{Bai} and J.~W. {Silverstein}.
\newblock {\em Spectral analysis of large dimensional random matrices}.
\newblock Springer Series in Statistics. Springer, New York, second edition,
  2010.

\bibitem{BaiYin93}
Z.~D. Bai and Y.~Q. Yin.
\newblock Limit of the smallest eigenvalue of a large-dimensional sample
  covariance matrix.
\newblock {\em Ann. Probab.}, 21(3):1275--1294, 1993.

\bibitem{Been97}
C.~W.~J. Beenakker.
\newblock Random-matrix theory of quantum transport.
\newblock {\em Rev. Mod. Phys.}, 69:731--808, Jul 1997.

\bibitem{Been11b}
C.~W.~J. Beenakker.
\newblock Condensed matter physics.
\newblock In {\em The {O}xford handbook of random matrix theory}, pages
  723--743. Oxford Univ. Press, Oxford, 2011.

\bibitem{BersReutBook}
J.~Berstel and C.~Reutenauer.
\newblock {\em Noncommutative rational series with applications}, volume 137 of
  {\em Encyclopedia of Mathematics and its Applications}.
\newblock Cambridge University Press, Cambridge, 2011.

\bibitem{Brou95}
P.~W. Brouwer.
\newblock Generalized circular ensemble of scattering matrices for a chaotic
  cavity with nonideal leads.
\newblock {\em Phys. Rev. B}, 51:16878--16884, Jun 1995.

\bibitem{BrouBeen96}
P.~W. Brouwer and C.~W.~J. Beenakker.
\newblock Diagrammatic method of integration over the unitary group, with
  applications to quantum transport in mesoscopic systems.
\newblock {\em J. Math. Phys.}, 37(10):4904--4934, 1996.

\bibitem{Butt90}
M.~B{\"u}ttiker.
\newblock Scattering theory of thermal and excess noise in open conductors.
\newblock {\em Phys. Rev. Lett.}, 65:2901--2904, Dec 1990.

\bibitem{ErdoKrugNemi_Poly}
L.~Erd\H{o}s, T.~Kr\"{u}ger, and {\relax Yu}.~Nemish.
\newblock Local laws for polynomials of {W}igner matrices.
\newblock {\em J. Funct. Anal.}, 278(12):108507, 59, 2020.

\bibitem{ErdoKrugSchr18}
L.~Erd\H{o}s, T.~Kr\"{u}ger, and D.~Schr\"{o}der.
\newblock Random matrices with slow correlation decay.
\newblock {\em Forum Math. Sigma}, 7:e8, 89, 2019.

\bibitem{ErdoKnowYauYin13b}
L.~Erd{\H{o}}s, A.~Knowles, H.-T. Yau, and J.~Yin.
\newblock The local semicircle law for a general class of random matrices.
\newblock {\em Electron. J. Probab.}, 18:no. 59, 58, 2013.

\bibitem{ErdoYauBook}
L.~Erd{\H{o}}s and H.-T. Yau.
\newblock {\em A dynamical approach to random matrix theory}, volume~28 of {\em
  Courant Lecture Notes in Mathematics}.
\newblock Courant Institute of Mathematical Sciences, New York; American
  Mathematical Society, Providence, RI, 2017.

\bibitem{FeldSodi10}
O.~N. Feldheim and S.~Sodin.
\newblock A universality result for the smallest eigenvalues of certain sample
  covariance matrices.
\newblock {\em Geom. Funct. Anal.}, 20(1):88--123, 2010.

\bibitem{ForrBook}
P.~J. Forrester.
\newblock {\em Log-gases and random matrices}, volume~34 of {\em London
  Mathematical Society Monographs Series}.
\newblock Princeton University Press, Princeton, NJ, 2010.

\bibitem{Fyod16}
Y.~V. Fyodorov.
\newblock {\em Random Matrix Theory of resonances: An overview}, pages
  666--669.
\newblock Institute of Electrical and Electronics Engineers Inc., 9 2016.

\bibitem{Fyod20}
Y.~V. Fyodorov.
\newblock Reflection time difference as a probe of s-matrix zeroes in chaotic
  resonance scattering.
\newblock {\em Acta Phys. Polon. A}, 136(5):785--789, 1 2019.

\bibitem{FyodSavi11}
Y.~V. Fyodorov and D.~V. Savin.
\newblock Resonance scattering of waves in chaotic systems.
\newblock In {\em The {O}xford handbook of random matrix theory}, pages
  703--722. Oxford Univ. Press, Oxford, 2011.

\bibitem{GeszTsek00}
F.~Gesztesy and E.~Tsekanovskii.
\newblock On matrix-valued {H}erglotz functions.
\newblock {\em Math. Nachr.}, 218:61--138, 2000.

\bibitem{HaagThor05}
U.~Haagerup and S.~Thorbj{\o}rnsen.
\newblock A new application of random matrices: {${\rm Ext}(C^*_{\rm
  red}(F_2))$} is not a group.
\newblock {\em Ann. of Math. (2)}, 162(2):711--775, 2005.

\bibitem{HeltMaiSpei18}
J.~W. Helton, T.~Mai, and R.~Speicher.
\newblock Applications of realizations (aka linearizations) to free
  probability.
\newblock {\em J. Funct. Anal.}, 274(1):1--79, 2018.

\bibitem{HeltRaFaSpei07}
J.~W. Helton, R.~Rashidi~Far, and R.~Speicher.
\newblock Operator-valued semicircular elements: solving a quadratic matrix
  equation with positivity constraints.
\newblock {\em Int. Math. Res. Not. IMRN}, (22):Art. ID rnm086, 15, 2007.

\bibitem{Klee56}
S.~C. Kleene.
\newblock Representation of events in nerve nets and finite automata.
\newblock In {\em Automata studies}, Annals of mathematics studies, no. 34,
  pages 3--41. Princeton University Press, Princeton, N. J., 1956.

\bibitem{Lehn99}
F.~Lehner.
\newblock Computing norms of free operators with matrix coefficients.
\newblock {\em Amer. J. Math.}, 121(3):453--486, 1999.

\bibitem{MahaWeid68}
C.~Mahaux and H.~A. Weidenm\"uller.
\newblock Comparison between the $r$-matrix and eigenchannel methods.
\newblock {\em Phys. Rev.}, 170:847--856, Jun 1968.

\bibitem{MaiSpeiYin19}
T.~{Mai}, R.~{Speicher}, and S.~{Yin}.
\newblock {The free field: realization via unbounded operators and Atiyah
  property}.
\newblock ArXiv:1905.08187, 2019.

\bibitem{MellPereSeli85}
P.~A. Mello, P.~Pereyra, and T.~H. Seligman.
\newblock Information theory and statistical nuclear reactions. {I}. {G}eneral
  theory and applications to few-channel problems.
\newblock {\em Annals of Physics}, 161(2):254 -- 275, 1985.

\bibitem{OberSukhStruSchoHeinHoll01}
S.~Oberholzer, E.~V. Sukhorukov, C.~Strunk, C.~Sch\"onenberger, T.~Heinzel, and
  M.~Holland.
\newblock Shot noise by quantum scattering in chaotic cavities.
\newblock {\em Phys. Rev. Lett.}, 86:2114--2117, Mar 2001.

\bibitem{Reut96}
C.~Reutenauer.
\newblock Inversion height in free fields.
\newblock {\em Selecta Math. (N.S.)}, 2(1):93--109, 1996.

\bibitem{RudeVers09}
M.~Rudelson and R.~Vershynin.
\newblock Smallest singular value of a random rectangular matrix.
\newblock {\em Comm. Pure Appl. Math.}, 62(12):1707--1739, 2009.

\bibitem{SchiBook}
J.~L. Schiff.
\newblock {\em Normal families}.
\newblock Universitext. Springer-Verlag, New York, 1993.

\bibitem{Scho17}
H.~Schomerus.
\newblock Random matrix approaches to open quantum systems.
\newblock In {\em Stochastic processes and random matrices}, pages 409--473.
  Oxford Univ. Press, Oxford, 2017.

\bibitem{Spei98}
R.~Speicher.
\newblock Combinatorial theory of the free product with amalgamation and
  operator-valued free probability theory.
\newblock {\em Mem. Amer. Math. Soc.}, 132(627):x+88, 1998.

\bibitem{VerbWeidZirn85}
J.~Verbaarschot, H.~Weidenm{\"u}ller, and M.~Zirnbauer.
\newblock Grassmann integration in stochastic quantum physics: The case of
  compound-nucleus scattering.
\newblock {\em Physics Reports}, 129(6):367 -- 438, 1985.

\bibitem{Wign58}
E.~P. Wigner.
\newblock On the distribution of the roots of certain symmetric matrices.
\newblock {\em Ann. of Math. (2)}, 67:325--327, 1958.

\bibitem{Yin18}
S.~Yin.
\newblock Non-commutative rational functions in strongly convergent random
  variables.
\newblock {\em Adv. Oper. Theory}, 3(1):178--192, 2018.

\end{thebibliography}

\end{document}